\documentclass[english]{article}
\usepackage[T1]{fontenc}
\usepackage[latin9]{inputenc}
\usepackage{geometry}
\usepackage{babel}
\usepackage{float}
\usepackage{mathtools}
\usepackage{bm}
\usepackage{amsmath}
\usepackage{amsthm}
\usepackage{amssymb}
\usepackage{wasysym}
\usepackage[unicode=true,pdfusetitle,
 bookmarks=true,bookmarksnumbered=false,bookmarksopen=false,
 breaklinks=false,pdfborder={0 0 1},backref=false,colorlinks=false]
 {hyperref}

\usepackage[authoryear,round]{natbib}
\makeatletter
\allowdisplaybreaks

\floatstyle{ruled}
\newfloat{algorithm}{tbp}{loa}
\providecommand{\algorithmname}{Algorithm}
\floatname{algorithm}{\protect\algorithmname}

\newtheorem{lemma}{\textbf{Lemma}}
\newtheorem{theorem}{\textbf{Theorem}}
\newtheorem{corollary}{\textbf{Corollary}}
\newtheorem{assumption}{\textbf{Assumption}}

\newtheorem{definition}{\textbf{Definition}}

\newtheorem{proposition}{\textbf{Proposition}}

\DeclareMathOperator*{\argmin}{arg\,min}

\newcommand{\peakmPREV}{B_{m,j}}
\newcommand{\instfamPREV}{\mathcal{F}}
\newcommand{\omgdistPREV}{\mathrm{Unif}(\Omega)}
\newcommand{\eoverpiPREV}{\mathbb{E}_{\pi,\omega}}
\newcommand{\wlooPREV}{\omega_{-(m,j)}}
\newcommand{\omglooPREV}{\Omega_{-(m,j)}}
\newcommand{\pminusPREV}{\mathbb{P}_{\pi,\omega_{m,j}^{}=-1}}
\newcommand{\pplusPREV}{\mathbb{P}_{\pi,\omega_{m,j}^{}=1}}

\newcommand{\crSymbol}{\mathrm{RI}(\Gamma,\pi)}

\newcommand{\peaki}{B_{i,j}}

\newcommand{\pminus}{\mathbb{P}_{\Gamma, \pi;\sigma_{i,j}=-1}}
\newcommand{\pplus}{\mathbb{P}_{\Gamma, \pi;\sigma_{i,j}=1}}

\newcommand{\gmin}{\gamma_\mathrm{min}}
\newcommand{\gmax}{\gamma_\mathrm{max}}
\newcommand{\sset}{\mathcal{S}}
\newcommand{\nullp}{P_0}
\newcommand{\altq}{Q}
\newcommand{\nsample}{n}
\newcommand{\binht}{\delta}
\newcommand{\nbin}{z}
\newcommand{\supportsize}{s}
\newcommand{\multinomvec}{\bm{N}}
\newcommand{\nheadvec}{\bm{R}}

\newcommand{\singlecordtv}{\frac{1}{4M}}
\newcommand{\chisquareconstant}{32}
\newcommand{\Mexponent}{-8}
\newcommand{\Mexponentoracle}{-4}
\newcommand{\upperprefactor}{M^{5}(\log T)}

\usepackage{xspace}
\newcommand{\myalg}{\texttt{BaSEDB}\xspace}

\newcommand{\rbalg}{\texttt{RoBIN}\xspace}

\usepackage{color}
\definecolor{cm}{RGB}{0,0,200}

\definecolor{rj}{RGB}{0,0,200}

\allowdisplaybreaks

\usepackage{tikz}

\usepackage{caption}
\usetikzlibrary{decorations.pathreplacing}
\usepackage{enumitem}

\usepackage{algorithm}
\usepackage{algorithmic}

\newcommand{\Env}{\mathcal{P}}
\newcommand{\EnvAlpha}{\Env_{\alpha}}

\newcommand{\knowledge}{\mathcal{K}}

\newcommand{\gridset}{\mathcal{U}_M}

\newcommand{\optexp}{\psi_{M}^\star}

\newcommand{\posrev}{f^{(1)}}
\newcommand{\negrev}{f^{(-1)}}

\usepackage{authblk}
\usetikzlibrary{decorations.pathreplacing}
\author[1]{Rong Jiang}
\author[1,2]{Cong Ma}
\affil[1]{Committee on Computational and Applied Mathematics, University of Chicago}
\affil[2]{Department of Statistics, University of Chicago}

\makeatother

\begin{document}
\title{The Adaptivity Barrier in Batched Nonparametric Bandits: \\ Sharp Characterization of the Price of Unknown Margin}
\date{November 2025}
\maketitle

\begin{abstract}
  We study batched nonparametric contextual bandits under a margin condition when the margin parameter $\alpha$ is unknown. 
  To capture the statistical cost of this ignorance, we introduce the regret inflation criterion, defined as the ratio between the regret of an adaptive algorithm and that of an oracle knowing $\alpha$. 
  We show that the optimal regret inflation grows polynomially with the horizon $T$, with exponent $\optexp$ given by the value of a convex optimization problem that depends on the dimension, smoothness, and number of batches $M$. Moreover, the minimizer of this optimization problem directly prescribes the batch allocation and exploration strategy of a rate-optimal algorithm. Building on this principle, we develop \rbalg (RObust batched algorithm with adaptive BINning), which achieves the optimal regret inflation up to polylogarithmic factors.  These results reveal a new adaptivity barrier: under batching, adaptation to an unknown margin parameter inevitably incurs a polynomial penalty, sharply characterized by a variational problem.
  Remarkably, this barrier vanishes once the number of batches exceeds order $\log \log T$; with only a doubly logarithmic number of updates, one can recover the oracle regret rate up to polylogarithmic factors.  
\end{abstract}


\section{Introduction}

A central question in sequential decision making is the cost of adaptation: how much performance is lost when key complexity parameters are unknown.
Nonparametric contextual bandits provide a canonical setting to study this question~\citep{woodroofe1979one,yang2002nonp,rigollet2010nonparametric,perchet2013multi}.
In the fully online regime, the problem is well understood. 
Under smoothness and margin assumptions, algorithms that attain minimax-optimal regret  can even adapt to an unknown margin parameter at no extra cost. 
In particular, the foundational work of~\cite{rigollet2010nonparametric} and the 
ABSE policy of~\cite{perchet2013multi} demonstrate that margin adaptation comes at no statistical cost in the \emph{online} regime.

In many settings, including clinical trials, education, and digital platforms, fully online interaction is infeasible because of logistical, ethical, or computational constraints. 
Instead, data collection proceeds in a limited number of \emph{batches}: actions are fixed for a group of covariates, feedback is revealed only at the end of the batch, and subsequent policies must adapt accordingly. 
While minimax-optimal rates have been established for batched nonparametric contextual bandits when the margin parameter is known, these procedures require oracle knowledge to tune batch sizes and exploration schedules~\citep{jiang2025batched}.
This raises a fundamental question: 
{\center \textit{What is the statistical price of not knowing the margin parameter when learning under batch constraints?}} 
\medskip

This paper provides a sharp answer.
We introduce the regret inflation criterion, defined as the ratio between the regret of an adaptive algorithm and that of an oracle who knows the true margin parameter. 
We show that the optimal regret inflation grows polynomially as $T^{\optexp}$ with the horizon $T$, with an exponent $\optexp$ precisely characterized by a convex variational optimization problem that depends on the dimension, smoothness, and batch budget.
Strikingly, the minimizer of this program also prescribes the batch allocation and exploration schedule of a rate-optimal algorithm, yielding matching upper and lower bounds up to polylogarithmic factors.

A key corollary of this result is the identification of an adaptivity barrier unique to the batch constraint. 
In the online regime, the margin parameter admits free adaptation, but when updates are limited, adaptation becomes inherently costly. 
We prove that the barrier vanishes when the number of batches exceeds order $\log\log T$; that is, with a doubly logarithmic number of updates, one can match the oracle regret rate up to polylogarithmic factors. 
Conversely, when the batch budget grows more slowly than $\log\log T$, the regret inflation is unavoidably polynomial.
This threshold cleanly delineates the transition between the regimes where batching constrains adaptation and where it does not.

Conceptually, our analysis unifies statistical limits and algorithm design through a single variational object. The variational characterization exposes the precise dependence of adaptivity cost on dimension, smoothness, and batch budget, while its minimizer yields an explicit constructive principle for designing robust batched policies under unknown complexity.

\subsection{Related work}\label{sec:related-work}

\paragraph{Contextual bandits.} The concept of contextual bandits was introduced by~\citet{woodroofe1979one}. For linear contextual bandits, \citet{auer2002using,abbasi2011improved,goldenshluger2013linear,bastani2020online,qian2023adaptive} established regret guarantees in both
low- and high-dimensional settings. Meanwhile, modeling the mean reward function as a smooth function of the contexts was studied in~\cite{yang2002nonp}. \cite{rigollet2010nonparametric} proved a minimax regret lower bound for this setup and designed an upper-confidence-bound-type (UCB-type) policy to attain the near-optimal rate. \cite{perchet2013multi}
refined this result by proposing the Adaptively Binned Successive
Elimination (ABSE) policy that can also adapt to the unknown margin
parameter in the fully online setting. Additional insights in nonparametric contextual  bandits were obtained in~\cite{qian2016kernel,reeve2018k,guan2018nonparametric,hu2022smooth,suk2021self,gur2022smoothness,cai2024transfer}.

\paragraph{Margin condition in classification.}
The margin condition originates from nonparametric classification, where it was studied by~\cite{audibert2007fast}. This condition, often referred to as the Tsybakov margin condition, quantifies how well-separated the optimal decision boundary is and directly governs learning rates. Its adaptation to contextual bandits was initiated by~\cite{Goldenshluger_2009,rigollet2010nonparametric,perchet2013multi}, who showed that the margin parameter $\alpha$ fundamentally shapes the complexity of bandit learning. In the online setting, adaptation to unknown $\alpha$ is possible without additional regret cost~\citep{perchet2013multi}. However, under the batch constraint, as explored in our work, such adaptivity becomes costly, giving rise to a new barrier.

\paragraph{Batch learning.} The multi-armed bandit problem under the batched setting was studied by \cite{perchet2016batch,zijun2020batch}. Batch learning in linear contextual bandits was studied by \cite{han2020sequential,ren2020batched,ruan2021linear} and \cite{ren2023dynamic,wang2020online,fan2023provably} further
considered the problem with high-dimensional covariates. \cite{jiang2025batched,arya2025batched} studied the nonparametric contextual bandit problem under the batch constraint. \cite{karbasi2021parallelizing,kalkanli2021batched}
developed batched Thompson sampling algorithms. \cite{feng2022lipschitz} considered the Lipschitz continuum-armed bandit problem under the batched setting. Further insights in batched bandits were developed in~\cite{zhang2020inference,tianyuan2021anytime, tianyuan2021double,liu2025batched}. Another related topic is online learning with switching costs~\citep{cesa2013online}. Best arm identification with limited rounds of interaction has been studied by~\cite{tao2019collaborative}. Reinforcement learning with low switching costs has been considered by~\cite{bai2019provably,zhang2020almost,gao2021provably,wang2021provably,qiao2022sample}.

\paragraph{Adaptation.} 
In the fully online setting, adaptivity to the margin parameter is feasible at no extra cost~\citep{perchet2013multi}. One might ask whether the same is true for the smoothness parameter. The answer is negative: even without the batch constraint, adaptation to smoothness is impossible~\citep{locatelli2018adaptivity,gur2022smoothness,cai2022stochastic}. Minimax regret rates depend explicitly on the H\"older smoothness $\beta$, and no single procedure can achieve the optimal rate simultaneously across different values of $\beta$. This impossibility parallels classical results in nonparametric estimation and classification, where smoothness adaptation requires additional structure or necessarily incurs a penalty~\citep{qian2016random,gur2022smoothness,cai2024transfer}. Thus, the margin parameter is the quantity of genuine interest: it admits free adaptation online, yet, becomes costly under the batch constraint.

\subsection{Notation and paper organization}
For any positive integer $n$, we use the shorthand $[n]$ to denote the set $\{1,2,\ldots, n\}$. We use the notations $\apprle$, $\apprge$, and $\asymp$ to indicate relationships that hold up to constant factors. Specifically, $f(n) \apprle g(n)$ means there exists a constant $C > 0$ such that $f(n) \le C\,g(n)$, while $f(n) \apprge g(n)$ indicates that $f(n) \ge c\,g(n)$ for some constant $c > 0$. We write $f(n) \asymp g(n)$ when both $f(n) \apprle g(n)$ and $f(n) \apprge g(n)$ hold.

In terms of paper organization, in  Section~\ref{sec:setup}, we introduce the problem setup and the regret inflation criterion. 
Section~\ref{sec:main} presents the main results, including the variational characterization of regret inflation and the master theorem. Section~\ref{sec:algorithm} describes the optimal algorithm guided by the variational principle. Section~\ref{sec:lower} contains the proof of the lower bound.  
We conclude with a discussion of extensions and future directions in Section~\ref{sec:discussion}.




\section{Problem setup and the regret inflation criterion\label{sec:setup}}

We begin by introducing the model and assumptions for batched contextual bandits, then review the oracle regret when the margin parameter is known, and finally introduce the regret inflation criterion that drives the rest of our analysis.

\subsection{Model and assumptions}

We study a two-armed nonparametric contextual bandit with horizon \(T\). At each round,
\[
(X_t, Y_t^{(1)}, Y_t^{(-1)}), \qquad t=1,\dots,T,
\]
are drawn i.i.d.~from a distribution $P$, where the context \(X_t \in \mathcal{X}\coloneqq [0,1]^d\) follows a distribution \(P_X\).
The rewards take values in \([0,1]\) and satisfy
\[
\mathbb{E}\!\big[Y_t^{(k)} \,|\, X_t\big] = f^{(k)}(X_t), \qquad k\in\{1,-1\},
\]
for unknown mean reward functions \(\posrev,\negrev\).

\paragraph{Batch policies.}
Under an \(M\)-batch constraint, the learner specifies
(i) a partition \(\Gamma=\{0=t_0<t_1<\cdots<t_M=T\}\) of the horizon, and
(ii) a sequence of decision rules \(\pi=(\pi_t)_{t=1}^T\).
At time \(t\), only contexts up to \(t\) and rewards from \emph{previous} batches are available. 
Let \(\Gamma(t)\) denote the batch index of round \(t\), i.e., $\Gamma(t) \coloneqq  i$ if $t_{i-1} < t \leq t_i$. The information set at time $t$ is
\[
\mathcal{H}_t=\{X_\ell\}_{\ell=1}^t \cup \{Y_\ell^{(\pi_{\ell}(X_\ell))}\}_{\ell=1}^{\,t_{\Gamma(t)-1}}.
\]
The grid $\Gamma$ may be chosen adaptively, meaning that the statistician can use all information up to $t_{i-1}$ to determine $t_i$. The statistician's policy $\pi_{t}$
at time $t$ may depend on $\mathcal{H}_{t}$. The goal is to design an $M$-batch policy $(\Gamma,\pi)$
that can compete with an oracle that knows the environment, i.e., the law $P$ of $(X_{t},Y_{t}^{(1)},Y_{t}^{(-1)})$.


\paragraph{Distributional assumptions.}
We impose the following standard conditions in the nonparametric bandits literature
\citep{rigollet2010nonparametric,perchet2013multi,cai2024transfer,jiang2025batched}:

The first assumption is concerned with the regularity of the covariate distribution~$P_X$.

\begin{assumption}[Bounded density]\label{ass:bdd-density}
There exist constants \(\underline{c},\bar{c}>0\) such that 
\[
\underline{c}\, r^d \;\le\; P_X\!\big(B(x,r)\big) \;\le\; \bar c\, r^d,
\qquad \forall x\in \mathrm{supp}(P_X), \forall r \in (0,1],
\]
where $B(x,r)$ is the $\ell_\infty$ ball centered at $x$ with radius $r$.
\end{assumption}

The second assumption
is on the smoothness of the mean reward functions. 
\begin{assumption}[Smoothness]\label{ass:smooth}
Each \(f^{(k)}\) is \((\beta,L)\)-H\"older smooth:
\[
|f^{(k)}(x)-f^{(k)}(x')|\le L\|x-x'\|_2^\beta, \qquad \forall x,x'\in\mathcal{X},\; k\in\{1,-1\}.
\]
\end{assumption}

The last assumption measures the closeness between the reward
functions of the two actions.

\begin{assumption}[Margin]\label{ass:margin}
For some \(\alpha\ge0\), there exist \(\delta_0\in(0,1)\) and \(D_0>0\) such that
\[
P_X\!\left(0<\big|\posrev(X)-\negrev(X)\big|\le \delta\right) \le D_0\, \delta^\alpha,
\qquad \forall \delta\in[0,\delta_0].
\]
\end{assumption}
\noindent For a fixed margin parameter \(\alpha\), let \(\EnvAlpha\) be the class of distributions satisfying Assumptions~\ref{ass:bdd-density}--\ref{ass:margin}, where we implicitly assume that $d$ and $\beta$ are fixed and known.

Assumption~\ref{ass:margin} pertains to the margin condition
in nonparametric classification \citep{mammen1999smooth,tsybakov2004optimal,audibert2007fast}, and has been adapted to the bandit setup by~\cite{Goldenshluger_2009,rigollet2010nonparametric,perchet2013multi}.
The margin parameter governs the fundamental complexity of the problem. When $\alpha=0$, the margin assumption becomes vacuous, and the reward functions of the two arms can be arbitrarily close to each other, making it challenging to identify the optimal one. When $\alpha$ increases, the reward functions of the two actions exhibit strong separation over a region of high probability mass, and discerning the optimal action is less difficult.

The following proposition adapted from~\cite{perchet2013multi} depicts the interplay
between the smoothness parameter $\beta$ and the margin parameter
$\alpha$. 

\begin{proposition}\label{propo:alpha}
Under Assumptions~\ref{ass:bdd-density}-\ref{ass:margin}:
\begin{itemize}
\item When $\alpha>d/\beta$, there is a constant gap between the reward
functions of the two arms and one can take $\alpha=\infty$. 
\item When $\alpha\le d/\beta$, there exist nontrivial contextual bandit
instances in $\EnvAlpha$.
\end{itemize}
\end{proposition}

\noindent In other words, $\alpha>d/\beta$ is the regime where the problem class is reduced to multi-armed bandits
without covariates 
and one equivalently has $\alpha=\infty$. On the other hand, $\alpha\le d/\beta$ is the regime where $\EnvAlpha$ corresponds to a non-degenerate class of nonparametric bandits.

\subsection{Oracle regret with known margin}

Given an \(M\)-batch policy \((\Gamma,\pi)\) and an environment $P$ with reward functions \((\posrev,\negrev)\),  we define the cumulative regret
\begin{align}
\label{eq:regret}
R_T(\Gamma,\pi;P)
\coloneqq 
\mathbb{E}_{P}\!\left[\sum_{t=1}^T \Big(f^\star(X_t)-f^{(\pi_t(X_t))}(X_t)\Big)\right],
\end{align}
where $f^{\star}(x)\coloneqq \max_{k\in\{1,-1\}}f^{(k)}(x)$ is the
maximum mean reward one could obtain on the context $x$.

The oracle minimax regret with known \(\alpha\) is
\begin{equation}\label{eq:oracle-rate-def}
 R_T^\star(\alpha)
\coloneqq 
\inf_{(\Gamma,\pi)}
\;\sup_{P \in \EnvAlpha}
\;
R_T(\Gamma,\pi; P),   
\end{equation}
where \(\EnvAlpha\) denotes the class of environments with margin $\alpha$.
It is known that the rate depends on the dimension \(d\), smoothness \(\beta\), margin \(\alpha\), and batch budget \(M\).
Define
\begin{equation}\label{eq:def-gamma-h}
\gamma(\alpha)\coloneqq  \frac{\beta(\alpha+1)}{2\beta+d},
\qquad \text{and} \qquad
h_M(\alpha)\coloneqq  \frac{1-\gamma(\alpha)}{1-\gamma(\alpha)^M}.    
\end{equation}
When $\alpha = \infty$, we have $\gamma(\infty) = \infty$, and $h_M(\infty) = 0$. 
The optimal rates with known margin have been established in~\cite{jiang2025batched}, which are stated in the following proposition.

\begin{proposition}\label{proposition:oracle-rate}
Fix a margin $\alpha \in [0,d/\beta] \cup \{\infty\}$. We have
    \begin{align}\label{eq:oracle-regret}
c_1 M^{\Mexponentoracle}  T^{\,h_M(\alpha)} \leq R_T^\star(\alpha) \leq c_2 M^{} (\log T)  T^{\,h_M(\alpha)},
    \end{align}
    where $c_1,c_2>0$ are constants independent of $T$ and $M$.
\end{proposition}

\noindent Although~\cite{jiang2025batched} established the minimax rate for the case $\alpha\le1/\beta$, we extend their result to $\alpha \in [0,d/\beta] \cup \{\infty\}$ for $d\ge 1$. Importantly, we make the dependence on $M$ explicit in the lower bound. See Appendix~\ref{sec:lower-bound-extension} for the proof.
\medskip

\subsection{The regret inflation criterion\label{sec:regInf}}
When the margin parameter is unknown, the key question is: how much additional regret must we pay to adapt? To capture this, we introduce the notion of regret inflation.

\begin{definition}\label{eq:comp_ratio}
Denote by $\knowledge \coloneqq  [0, d/\beta] \cup \{\infty\}$ the set of possible margin parameters. 
For any $M$-batch policy $(\Gamma,\pi)$, define the regret inflation as 
\begin{equation}
\mathrm{RI}(\Gamma,\pi)
\coloneqq  \sup_{\alpha \in \knowledge} \sup_{P \in \EnvAlpha}
\frac{R_T(\Gamma,\pi;P)}{R_T^\star(\alpha)}.
\end{equation}

\end{definition}

This ratio compares the regret of the adaptive policy to that of the oracle who knows $\alpha$. 
Our notion of regret inflation parallels the risk inflation criterion in model selection for linear models \citep{foster1994risk}. In that context, risk inflation measures the excess risk of a data-driven model selection rule relative to an oracle that knows the true model. Analogously, regret inflation quantifies the excess regret of an adaptive bandit policy relative to an oracle that knows the margin parameter. Both notions capture the statistical price of adaptation to unknown complexity.

Our goal is to characterize the rate of the optimal regret inflation, 
\[
\inf_{\Gamma, \pi} \crSymbol,
\]
and its dependence on the batch budget $M$.

\section{Sharp characterization of the optimal regret inflation\label{sec:main}}

The central task of this section is to analyze the optimal regret inflation. We show that it admits an exact characterization: its exponent is given by the value of a convex optimization problem, and the minimizers of this problem prescribe the design of the rate-optimal algorithm. Thus, the variational problem provides both a fundamental statistical limit and a constructive principle for algorithm design.

\subsection{Key variational problem}

Minimizing regret inflation can be formulated as a two-player zero-sum game.  The learner commits to an $M$-batch policy $(\Gamma, \pi)$ without knowledge of the margin parameter, while nature selects a distribution $P$ over contexts and rewards consistent with some margin parameter $\alpha$. 
The payoff is the ratio between the learner's regret and that of the oracle who knows $\alpha$. 

Although both the learner's strategy space (policies) and nature's strategy space (distributions) are infinite-dimensional, 
the complexity of this game can be captured by a \emph{finite-dimensional} reduction. 
The learner's choice reduces to specifying a batch allocation across the $M$ updates, and nature's choice reduces to selecting a margin parameter $\alpha$ from the admissible range $\knowledge$. 
We parameterize the batch schedule $\Gamma = \{t_0=0 < t_1 < \cdots < t_M=T\}$ by exponents $u_i \in [0,1]$, so that $t_i \approx T^{u_i}$. This exponent parameterization captures exponentially growing batch sizes, which balance exploration and exploitation under geometric time scales. We refer to $\bm{u} = (u_1,\dots,u_{M-1})$ as an exponent grid.
This reduction yields a finite-dimensional convex optimization problem.

Formally, let $\bm{u} \in 
    \gridset \coloneqq  \{\bm{u}\in\mathbb{R}^{M-1}:0\le u_{1}\le\cdots\le u_{M-1}\le1\}$ be the grid choice of the learner, and let $\alpha \in \knowledge$ be the margin parameter selected by nature. 
Define the payoff function 
\begin{align}\label{eq:def-payoff}
    \Psi_M(\bm u,\alpha) \;\coloneqq  \; \max_{1 \le i \le M} \eta_i(\bm u,\alpha) - h_M(\alpha), 
\end{align}
where 
\begin{align}\label{eq:def-eta}
    \eta_1(\bm{u},\alpha)=u_{1}, \quad 
    \eta_{i}(\bm{u},\alpha)=u_{i}-u_{i-1}\gamma(\alpha),\quad2\le i\le M-1, \quad 
    \eta_{M}(\bm{u},\alpha)=1-u_{M-1}\gamma(\alpha). 
\end{align}
Then the optimal value of this finite-dimensional two-player game is 
\begin{equation}\label{eq:min-opt}
\optexp \coloneqq  \inf_{\bm{u}\in\gridset} \sup_{\alpha \in \knowledge} \; \Psi_M(\bm u,\alpha).
\end{equation}
For notational convenience, we also define the objective function w.r.t.~$\bm{u}$ as 
\begin{equation}\label{eq:def-psi-u}
  \psi_M(\bm{u}) \coloneqq  \sup_{\alpha \in \knowledge} \Psi_{M} (\bm{u}, \alpha).   
\end{equation}

\subsection{Optimal regret inflation}

Our main theorem establishes a tight characterization of the optimal regret inflation.
\begin{theorem}\label{thm:main}
Let $\optexp$ be the optimal value of the variational problem as defined in~\eqref{eq:min-opt}. Then there exist 
constants $c_1,c_2>0$ independent of $T$ and $M$
such that
\begin{enumerate}
    \item For any $M$-batch policy $(\Gamma,\pi)$,
    \[
    \mathrm{RI}(\Gamma,\pi) \;\ge\; c_1 M^{\Mexponent} \, (\log T)^{-1}\, T^{\optexp}.
    \]
    \item There exists an $M$-batch policy $(\hat\Gamma,\hat\pi)$ such that
    \[
    \mathrm{RI}(\hat\Gamma,\hat\pi) \;\le\; c_2\, \upperprefactor\,  T^{\optexp}.
    \]
\end{enumerate}

\end{theorem}

To better interpret Theorem~\ref{thm:main}, 
we treat $M$ as fixed and focus on how the regret inflation scales with the horizon~$T$;the effect of varying 
$M$ will be analyzed in Section~\ref{sec:phase-transition}. 
Theorem~\ref{thm:main} shows that the variational value $\optexp$ fully determines this dependence:
up to logarithmic factors, any adaptive policy incurs inflation of order $T^{\optexp}$. 
If $\optexp$ is \emph{strictly positive}, then Theorem~\ref{thm:main} implies an unavoidable polynomial adaptation cost in~$T$. 
The next subsection analyzes the structure of the variational problem 
and establishes that $\optexp$ is indeed strictly positive.


\subsection{Structural properties of the variational problem}

\begin{proposition}
 \label{proposition:opt-prop}
 The following properties hold for the variational problem: 
    \begin{enumerate}
        \item The function $\psi_M$ is convex and admits a minimizer $\bm{u}^\star$ in the interior of $\gridset$ with positive optimal value, i.e., $\optexp > 0$. 
\item 
There exists a non-increasing sequence $\{\alpha_i\}_{1 \leq i \leq M}$ such that 
\[\psi_{M}(\bm{u}^{\star})=\;\eta_{i}\bigl(\bm{u}^{\star},\alpha_{i}\bigr)\;- h_M(\alpha_i).\]
Moreover, one has $\alpha_{1}=\infty$, and 
$\alpha_{i} \leq d/\beta $ for $2\le i\le M$. 
     
\end{enumerate}  
\end{proposition}
\noindent Part~1 of Proposition~\ref{proposition:opt-prop} ensures the exponent is \emph{strictly} positive, implying polynomial inflation (up to polylogs) in the fixed $M$ case via Theorem~\ref{thm:main}. 
In Part 2 of Proposition~\ref{proposition:opt-prop}, 
the sequence $\{\alpha_i\}$ identifies the branches of $\Psi_M$ that are active at $\bm u^\star$.
This structure is used in the lower-bound construction (Section~\ref{sec:lower}) to select difficulty levels of the hard instances.
 See Appendix~\ref{sec:variational} for the proof.

\subsection{How many batches suffice for no regret inflation?}\label{sec:phase-transition}
When we allow the number $M$ of batches to increase with the horizon $T$, a natural question arises: how many batches are sufficient for no regret inflation?

The following proposition provides a crude order-wise bound on $\optexp$ in terms of the number $M$ of batches. 

\begin{proposition}\label{prop:optexp-control}
    Let $\gmax \coloneqq  \frac{d+\beta}{d+2\beta}$. Then we have the following control on the optimal exponent
    \[
    \frac{\gmax^{M-1}(1-\gmax)^{2}}{(1-\gmax^{M})^{2}} \leq \optexp \leq \frac{(M+1)\gmax^{M-1}}{(1-\gmax)^{2}}.
    \]
\end{proposition}
\noindent See Appendix~\ref{sec:variational} for the proof. 
\medskip 

As a direct consequence of the upper bound in Proposition~\ref{prop:optexp-control}, we know that when $M \geq c_1 \log \log T$ for some large constant $c_1$, we have $T^{\optexp} = O(1)$. 
Therefore, we know that $\log \log T$ batches are sufficient for no regret inflation. 
In other words, even if we do not know the exact margin parameter $\alpha$ of the bandit instance, 
with $\log \log T$ batches, there exists an algorithm that enjoys the same order of regret as the oracle algorithm with the knowledge of the $\alpha$, up to the $(\log \log T)^5 \log T$ factor. 
This is summarized in the following corollary.
\begin{corollary}
There exists a constant $c_1$ such that if $M \geq c_1 \log \log T$, then 
\[
\inf_{\Gamma, \pi} \; \mathrm{RI} (\Gamma,\pi) \leq C (\log \log T)^5 \log T, 
\]
where $C$ is some constant independent of $T$.
\end{corollary}

Conversely, by the lower bound in Proposition~\ref{prop:optexp-control}, we also know that $M \asymp \log \log T$ is necessary for small regret inflation. 
That is, when $M \leq c_2 \log \log T$ for some small constant $c_2$, the regret inflation is super-polylogarithmic in $T$. 
Together, $M \asymp \log\log T$ marks a sharp phase transition between regimes with costly and cost-free margin adaptation.

\subsection{Numerical illustrations}\label{sec:numerics}

 \begin{figure}[t]
\centering
\begin{minipage}{0.48\textwidth}
    \centering
    \includegraphics[scale=0.4]{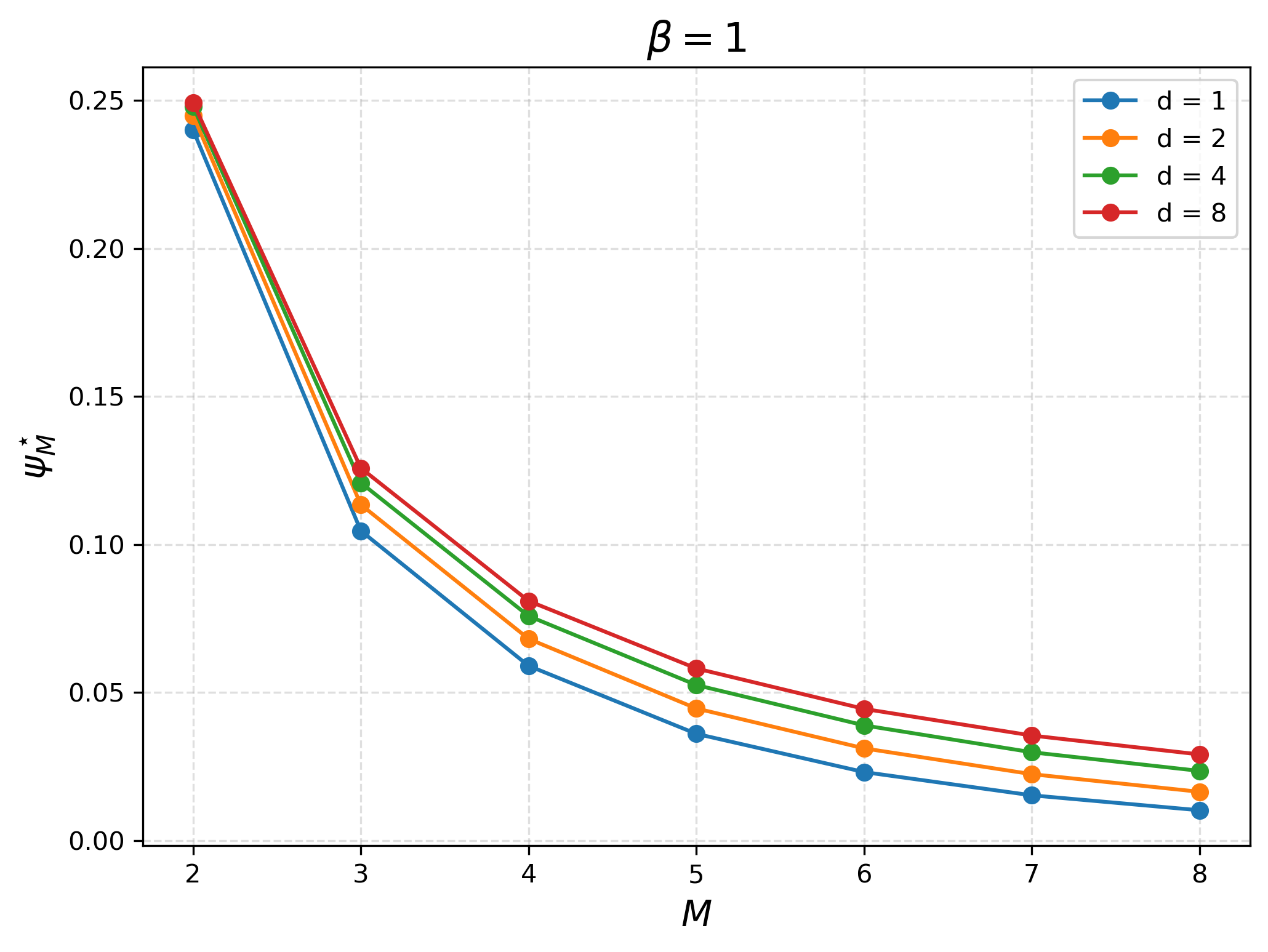}
    \caption{$\optexp$ vs.~batch budget $M$ when $\beta=1$.}
    \label{fig:varying-d}
\end{minipage}
\hfill
\begin{minipage}{0.48\textwidth}
    \centering
    \includegraphics[scale=0.4]{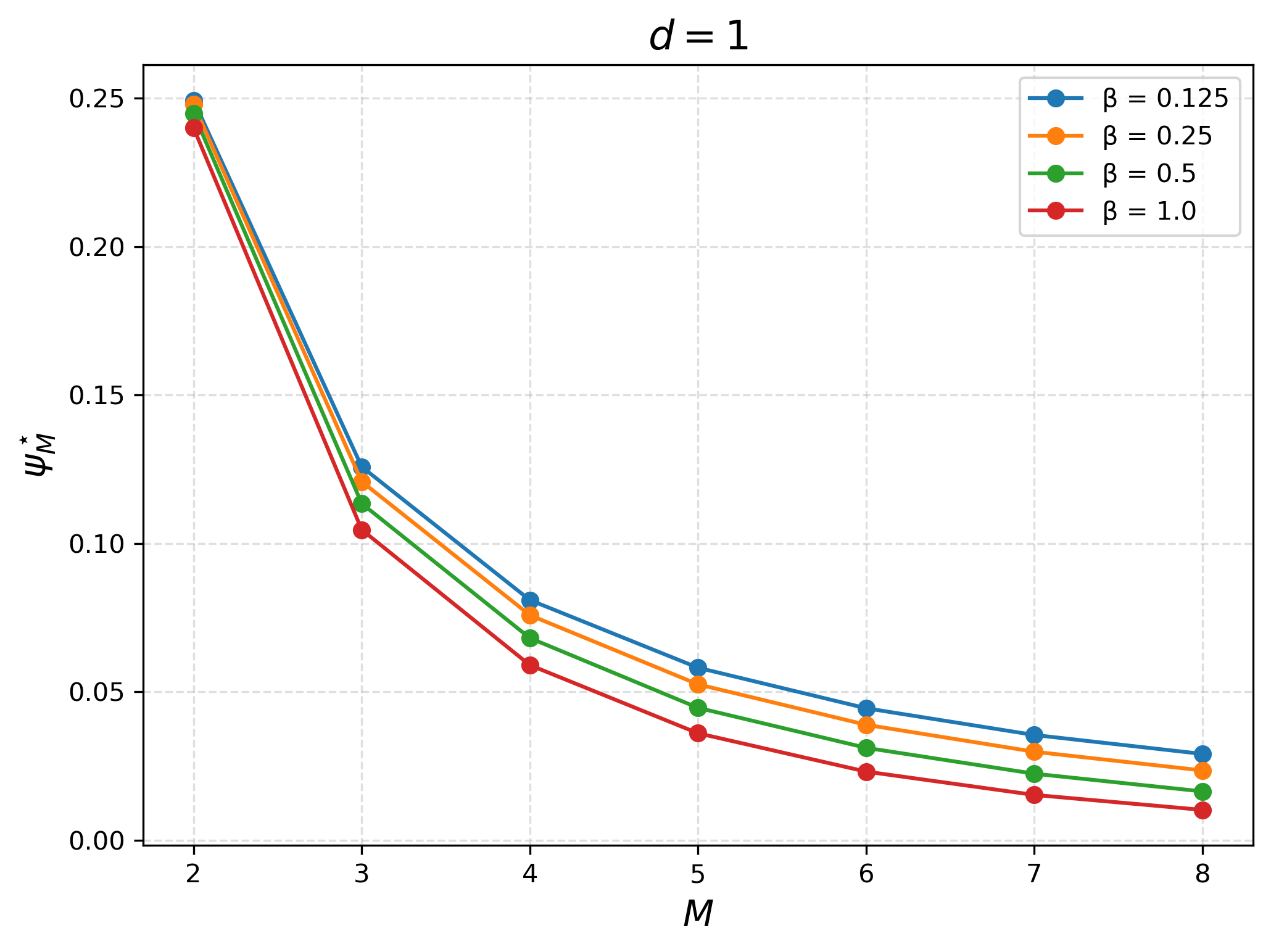}
    \caption{$\optexp$ vs.~batch budget $M$ when $d=1$.}
    \label{fig:varying-beta}
\end{minipage}
\end{figure}

Since $\optexp$ generally lacks a closed-form expression, we first turn to numerical solutions. These experiments shed light on how the difficulty of adaptation varies with smoothness, dimension, and batch budget.

Figure~\ref{fig:varying-d} fixes $\beta =1$ and varies $d$. It plots the trend of $\optexp$ as the number $M$ of batches increases. 
Similarly, Figure~\ref{fig:varying-beta} fixes $d =1$, and varies $\beta$. It also plots the trend of $\optexp$ as the number $M$ of batches increases. 
A few observations are in order. 

\begin{enumerate}
    \item All the exponents $\optexp$ are strictly positive, indicating a polynomial regret inflation when the margin parameter is not known. 
    \item When fixing $\beta$ and $M$, the exponent $\optexp$ is increasing in $d$. This demonstrates that adapting to $\alpha$ is increasingly difficult for high-dimensional problems. 
    \item Similarly, when fixing $d$ and $M$, the exponent $\optexp$ is increasing as $\beta$ decreases. This demonstrates that adapting to $\alpha$ is increasingly difficult for non-smooth problems. 
\end{enumerate}
    
In fact, the last two observations can be understood from the definition of the variational problem. 
Recall that under the setup for both plots, 
we have $\knowledge = [0,d/\beta] \cup \{\infty\}$. 
Correspondingly, the allowable range for $\gamma(\alpha)$ is 
\[
\gamma \in \left[\frac{\beta}{2\beta + d}, \frac{d +\beta}{2\beta + d} \right] \cup \{\infty\}.
\]
It is clear that as $d$ increases and $\beta$ decreases, the allowable range for $\gamma$ increases. 
Since the inner maximization problem is effectively over $\gamma$, the supremum is non-decreasing. 
As a result, the optimal exponent $\optexp$ is non-decreasing.

\section{The \rbalg algorithm: optimal adaptation to unknown margin\label{sec:algorithm}}

We now 
introduce \rbalg (RObust batched algorithm with adaptive BINning) 
to tackle batched contextual bandits with unknown margin.  The algorithm builds on the \myalg framework from~\cite{jiang2025batched}, but introduces a new design principle: the batch schedule and split factors are selected based on the solution to the key variational convex program~\eqref{eq:min-opt}. 
As we shall see, this design principle enables robust adaptation across all margin parameter values.

We first describe the \myalg framework with fixed grid size and split factors, then present our new robust choices informed by the variational problem, and finally state the regret inflation guarantee for \rbalg.

\begin{algorithm}
\caption{RObust batched algorithm with adaptive BINning (\rbalg)}
\label{alg:adaptive-bin}
\begin{algorithmic}[1]
\REQUIRE Batch size $M$, grid $\Gamma=\{t_{i}\}_{i=0}^{M}$, split factors $\{g_{i}\}_{i=0}^{M-1}$ as in Equations~\eqref{eq:split-factors} and~\eqref{eq:batch-bin-size}
\STATE $\mathcal{L}\leftarrow\mathcal{B}_{1}$.
\FOR{each $C\in\mathcal{L}$}
    \STATE $\mathcal{I}_{C}=\mathcal{I}$.    
\ENDFOR
\FOR{$i=1$ to $M-1$}
    \FOR{$t=t_{i-1}+1$ to $t_i$}
        \STATE $C\leftarrow\mathcal{L}(X_t)$.
        \STATE Pull an arm from $\mathcal{I}_C$ in a round-robin way.
    \ENDFOR
    \STATE Observe the outcomes for batch $i$
    \STATE Update $\mathcal{L}$ and $\{\mathcal{I}_{C}\}_{C\in\mathcal{L}}$ by invoking Algorithm \ref{algo-subroutine} with inputs $(\mathcal{L},\{\mathcal{I}_{C}\}_{C\in\mathcal{L}},i,g_{i})$.
\ENDFOR
\FOR{$t=t_{M-1}+1$ to $T$}
    \STATE $C\leftarrow\mathcal{L}(X_t)$.
    \STATE Pull any arm from $\mathcal{I}_C$.
\ENDFOR
\end{algorithmic}
\end{algorithm}

\subsection{The \myalg framework with fixed parameters}

We begin by reviewing the \myalg algorithm, which serves as the foundation for our robust variant. 
The \myalg algorithm operates over a horizon $T$ divided into $M$ batches, indexed by $i = 1, \ldots, M$. 
It consists of three main components:
\begin{itemize}
    \item A \emph{batch schedule}, specified by a grid $\Gamma = \{t_0 = 0 < t_1 < \cdots < t_M = T\}$, which determines the number of time steps in each batch;
    \item A sequence of \emph{split factors} $\{g_i\}_{i=0}^{M-1}$, which control how the covariate space $[0,1]^d$ is iteratively partitioned into bins;
    \item A \emph{Successive Elimination (SE)} subroutine (see Algorithm~\ref{algo-subroutine}) that runs independently in each active bin to eliminate suboptimal arms.
\end{itemize}

\paragraph{The high-level idea behind \myalg. }
On a high level, initially, the covariate space is divided into $g_0^d$ bins of equal width. In each batch, the algorithm pulls arms uniformly at random within each bin and uses SE to discard clearly suboptimal arms. After the $i$-th batch, bins with unresolved ambiguity are split into finer bins using the corresponding split factor $g_i$, yielding a refined partition. The final batch is reserved for exploitation, where the algorithm chooses arms based on the surviving set in each bin; see Algorithm~\ref{alg:adaptive-bin}.



\paragraph{A detailed description of \myalg. }
More formally the \myalg algorithm can be described using a tree diagram. Let $\mathcal{T}$ be a tree of depth $M$. The $i$-th level of the tree  $\mathcal{T}$  consists of a collection of bins that is a regular partition of the covariate space $\mathcal{X}$, which we denote by $\mathcal{B}_i$. Each bin $C\in\mathcal{B}_i$ has the same width $w_i$ given by $w_0 = 1$ and 
\begin{equation}
w_{i}\coloneqq (\prod_{l=0}^{i-1}g_{l})^{-1},\quad i\ge 1.\label{eq:width}
\end{equation}
Here $\{g_{i}\}_{i=0}^{M-1}$ is a list of split factors. 
More precisely, $\mathcal{B}_{i}$ is composed of all the bins 
\[
C_{i,\bm{v}}=\{x\in\mathcal{X}:(v_{j}-1)w_{i}\le x_{j}<v_{j}w_{i},1\le j\le d\},
\]
where $\bm{v}=(v_{1},v_{2},\ldots,v_{d})\in[\frac{1}{w_{i}}]^{d}$. Clearly, $\mathcal{B}_{i}$ has $(1/w_i)^{d}$  bins in total.

Algorithm~\ref{alg:adaptive-bin} operates in batches and keeps track of the following key variables: a collection $\mathcal{L}$ of active bins, and the set of active arms $\mathcal{I}_{C}$ for each $C\in\mathcal{L}$.
The collection of active bins $\mathcal{L}$ is initialized
to be $\mathcal{B}_{1}$, while the set of active arms $\mathcal{I}_{C}$ is set to be $\{1,-1\}$ for all $C\in\mathcal{L}$ at the beginning. During the $i$-th batch, each arm in $\mathcal{I}_{C}$ is pulled for an equal amount of times. At the end of that batch, the active arm set $\mathcal{I}_{C}$ is updated by doing a hypothesis testing based on the revealed rewards from this batch. If after the arm elimination process  $|\mathcal{I}_{C}|>1$ for some  $C\in\mathcal{L}$, this means the current bin $C$ is too coarse to distinguish the optimal action. Consequently, this bin $C$  is further split into its children $\textrm{child}(C,g_i)$ in tree $\mathcal{T}$, which is a set of $g_i^d$ bins, and the child nodes $\textrm{child}(C,g_i)$ will replace the original bin $C$ in $\mathcal{L}$. For the last batch, we simply pull any arm from $\mathcal{I}_C$ whenever the covariate $X_t$ lands in some $C\in\mathcal{L}$.\footnote{For the final batch $M$, the split factor $g_{M-1}=1$ by default because there is no need to further partition the nodes for estimation.}


\begin{algorithm}
\caption{Node splitting and arm elimination procedure}
\label{algo-subroutine}
\begin{algorithmic}[1]
\REQUIRE Active bin list $\mathcal{L}$, active arm sets $\{\mathcal{I}_{C}\}_{C\in\mathcal{L}}$, batch number $i$, split factor $g_i$.
\STATE $\mathcal{L}^\prime\leftarrow\{\}$
\FOR{each $C\in\mathcal{L}$}
    \IF{$|\mathcal{I}_C|=1$}
        \STATE $\mathcal{L}^\prime\leftarrow\mathcal{L}^\prime\cup\{C\}$.
        \STATE Proceed to next $C$ in the iteration.
    \ENDIF
    \STATE $\bar{Y}_{C,i}^{\max}\leftarrow\max_{k\in\mathcal{I}_{C}}\bar{Y}_{C,i}^{(k)}$.
    \FOR{each $k\in\mathcal{I}_C$}
        \IF{$\bar{Y}_{C,i}^{\max}-\bar{Y}_{C,i}^{(k)}>U(m_{C,i},T,C)$}
            \STATE $\mathcal{I}_{C}\leftarrow\mathcal{I}_{C}-\{k\}$
        \ENDIF
    \ENDFOR
    \IF{$|\mathcal{I}_{C}|>1$}
        \FOR{each $C^{\prime}\in\textrm{child}(C,g_{i})$}
            \STATE $\mathcal{I}_{C^{\prime}}\leftarrow\mathcal{I}_{C}$.
        \ENDFOR
        \STATE $\mathcal{L^{\prime}}\leftarrow\mathcal{L^{\prime}}\cup\textrm{child}(C,g_{i})$.
    \ELSE
        \STATE $\mathcal{L}^{\prime}\leftarrow\mathcal{L}^{\prime}\cup\{C\}$.
    \ENDIF
\ENDFOR
\STATE Return $\mathcal{L}^{\prime}$.
\end{algorithmic}
\end{algorithm}

Next, we turn to the arm elimination part in Algorithm~\ref{algo-subroutine}. 
The underlying idea is based on Successive Elimination (SE) from the bandit literature
\citep{even2006action,perchet2013multi,zijun2020batch}. An arm is eliminated from the active arm set $\mathcal{I}_C$ if the revealed rewards from this batch provide sufficient evidence of the suboptimality of this arm. For any node $C\in\mathcal{T}$, denote by $m_{C,i}\coloneqq \sum_{t=t_{i-1}+1}^{t_{i}}\mathbf{1}\{X_{t}\in C\}$ the number of times the covariates go into the bin $C$ during the $i$-th batch. For $k\in\{1,-1\}$, define the empirical estimate arm $k$'s reward in bin $C$ during batch $i$ as 
\[
\bar{Y}_{C,i}^{(k)}\coloneqq \frac{\sum_{t=t_{i-1}+1}^{t_{i}}Y_{t}\cdot\mathbf{1}\{X_{t}\in C,A_{t}=k\}}{\sum_{t=t_{i-1}+1}^{t_{i}}\mathbf{1}\{X_{t}\in C,A_{t}=k\}}.
\]
The expectation of $\bar{Y}_{C,i}^{(k)}$ is equal to 
\[
\bar{f}_{C}^{(k)}\coloneqq \mathbb{E}[f^{(k)}(X)\mid X\in C]=\frac{1}{P_{X}(C)}\int_{C}f^{(k)}(x)\mathrm{d}P_{X}(x).
\]
A key quantity for SE is the uncertainty level of the estimates in bin $C$, which is given by
\[
U(\tau,T,C)\coloneqq 4\sqrt{\frac{\log(2T|C|^{d})}{\tau}},
\]
where $|C|$ is the width of the bin. The uncertainty level is defined in a way so that with high probability the suboptimal arm for bin $C$ is eliminated while the near optimal ones remain in it. If $|\mathcal{I}_{C}|>1$, then the surviving arms are statistically close to each other; we further split $C$ into finer bins to obtain a more precise estimate of the rewards of these actions in later batches.

\subsection{Robust parameter design via variational optimization}

In the original \myalg algorithm, both the batch schedule $\Gamma$ and split factors $\{g_i\}$ are chosen assuming knowledge of the true margin parameter $\alpha$, which allows the algorithm to achieve minimax-optimal regret in that setting.

To achieve optimal regret without knowledge of $\alpha$, \rbalg selects the split factors and the batch points based on the minimizer $\bm{u}^\star$ to the convex program~\eqref{eq:min-opt}.

\paragraph{Split factor design.}
Based on $\bm{u}^\star$,  we define the split factors as 
\begin{equation}
g_{0}=\lfloor T^{\frac{1}{2\beta+d}\cdot u_{1}^{\star}}\rfloor,\qquad\text{and}\qquad g_{i}=\lfloor T^{\frac{1}{2\beta+d}(u_{i+1}^{*}-u_{i}^{\star})}\rfloor,\qquad \qquad i=1,\dots,M-2.\label{eq:split-factors}
\end{equation}

\paragraph{Batch grid construction.} 
In addition, we choose the grid to satisfy
\begin{align}
t_{i}-t_{i-1} & =\lfloor l_{i}w_{i}^{-(2\beta+d)}\log(Tw_{i}^{d})\rfloor,\quad1\le i\le M-1,\label{eq:batch-bin-size}
\end{align}
for some $l_i>0$ sufficiently large. 
Here, we recall that $w_i$ is given in Equation~\eqref{eq:width}. 
It can be shown that under these choices, we have 
\[
t_{i} \approx T^{u_{i}^{\star}},\quad1\le i\le M-1,
\]
where the approximation sign ignores log factors.

To summarize, \rbalg runs the \myalg procedure using these robust parameters $(\Gamma, \{g_i\})$ specified in Equations~\eqref{eq:split-factors} and~\eqref{eq:batch-bin-size}.

\subsection{Regret inflation guarantee}\label{sec:upper-proof}

We now state the main guarantee for \rbalg.

\begin{theorem}\label{thm:c-ratio-upper}Equipped with the grid and split
factors list that satisfy~(\ref{eq:batch-bin-size}) and~(\ref{eq:split-factors}),
the policy $(\hat{\Gamma}, \hat{\pi})$ given by $\rbalg$ obeys
\[
\mathrm{RI}(\hat {\Gamma}, \hat{\pi})
\apprle M^{5}(\log T)\cdot T^{\optexp}.
\]
\end{theorem}
This upper bound on the amount of regret inflation matches the lower bound result in Theorem~\ref{thm:main} (up to $\log$ factors). It demonstrates that \rbalg could achieve the optimal adaptation cost when the margin parameter is unknown. 

As we will soon see in its proof, the key to achieve optimal regret inflation is the use of the minimizer $\bm{u}^\star$ of the variational problem.

\subsection{Proof of Theorem~\ref{thm:c-ratio-upper}}

We now prove that \rbalg achieves the optimal regret inflation rate. The argument proceeds in two steps. First, we establish a regret bound for the \myalg algorithm with an arbitrary grid $\bm{u}$ in the interior of $\gridset$. Second, we specialize to the minimizer $\bm{u}^\star$ of the variational problem~\eqref{eq:min-opt}, which yields the optimal rate.

\paragraph{Step 1: Regret of \myalg with a fixed grid.}

Let $\bm{u}$ be any interior point of $\gridset$, i.e., $0< u_1 < u_2 < \cdots < u_{M-1} < 1$.  
Consider the split factors and the grid size given by~Equations~\eqref{eq:split-factors} and~(\ref{eq:batch-bin-size}). 
The following lemma provides a generic regret bound.

\begin{lemma}\label{lem:basedb-regret} Fix any $\alpha \geq 0$, and let $\hat{\pi}_{\bm{u}}$ be as above. Then
\begin{align*}
\sup_{P \in\EnvAlpha}R_{T}(\hat{\Gamma}_{\bm{u}}, \hat{\pi}_{\bm{u}}; P ) & \le c\left(t_{1}+\sum_{i=2}^{M-1}(t_{i}-t_{i-1})\cdot w_{i-1}^{\beta+\alpha\beta}+(T-t_{M-1})w_{M-1}^{\beta+\alpha\beta}\right),
\end{align*}
where $c$ depends only on $(\beta ,d)$. 
\end{lemma}
\noindent See Appendix~\ref{subsec:proof-lem-basedb-regret} for the proof. 
\medskip

Applying the relations~\eqref{eq:split-factors}-\eqref{eq:batch-bin-size}, this bound simplifies to
\[
\sup_{P \in \EnvAlpha} R_T(\hat{\Gamma}_{\bm{u}}, \hat{\pi}_{\bm{u}}; P) \;\lesssim\; (\log T)
\Big( T^{u_1} + \sum_{i=2}^{M-1} T^{u_i - u_{i-1}\gamma(\alpha)} + T^{1 - u_{M-1}\gamma(\alpha)} \Big).
\]
Since the maximum exponent dominates, we obtain
\[
\sup_{P \in \EnvAlpha} R_T(\hat{\Gamma}_{\bm{u}}, \hat{\pi}_{\bm{u}}; P) \;\lesssim\; (\log T) M \, T^{\max_{1 \le i \le M} \eta_i(\bm{u}, \alpha)}.
\]
Dividing by the oracle regret $R_T^\star(\alpha) \asymp T^{h_M(\alpha)}$ from Proposition~\ref{proposition:oracle-rate} yields 
\begin{equation}\label{eq:ratio-fixed-alpha-clean}
\sup_{P \in \EnvAlpha}
\frac{R_T(\hat{\Gamma}_{\bm{u}}, \hat{\pi}_{\bm{u}}; P)}{R_T^\star(\alpha)}
\;\lesssim\; \upperprefactor T^{\max_{1 \le i \le M} \eta_i(\bm{u}, \alpha) - h_M(\alpha)}.
\end{equation}
Taking the supremum of~\eqref{eq:ratio-fixed-alpha-clean} over all $\alpha \in \knowledge$ gives
\[
\sup_{\alpha \in \mathcal{A}} \sup_{P \in \EnvAlpha}
\frac{R_T(\hat{\Gamma}_{\bm{u}}, \hat{\pi}_{\bm{u}}; P)}{R_T^\star(\alpha)}
\;\lesssim\; \upperprefactor \,
T^{\sup_{\alpha \in \mathcal{A}} \big(\max_{1 \le i \le M} \eta_i(\bm{u}, \alpha) - h_M(\alpha)\big)}.
\]

\paragraph{Step 2: optimizing over $\bm{u}$.}

Finally, by construction, \rbalg corresponds to the policy $\hat\pi_{\bm{u}^\star}$ where $\bm{u}^\star$ minimizes the right-hand side. This gives
\[
\sup_{\alpha \in \knowledge} \sup_{P \in \EnvAlpha}
\frac{R_T(\hat{\Gamma}_{\bm{u}^\star}, \hat{\pi}_{\bm{u}^\star}; P)}{R_T^\star(\alpha)}
\;\lesssim\; \upperprefactor \, T^{\optexp}.
\]
This matches the lower bound of Theorem~\ref{thm:main} up to logarithmic factors, completing the proof of Theorem~\ref{thm:c-ratio-upper}.

\subsection{Solve the variational problem}
The final step in implementing \rbalg is computing the optimal batch grid $\bm{u}^\star$, which solves the variational problem \eqref{eq:min-opt} underlying our theoretical analysis.
Although Proposition~\ref{proposition:opt-prop} establishes that the objective function $\psi(\bm{u})$ is convex in $\bm{u}$, convexity alone does not immediately yield an efficient numerical solver. 
Instead, we exploit an equivalent characterization of the optimal solution.

\begin{proposition}\label{prop:equiv-variational}
Define $\phi_{M}(x)=\min_{\alpha \in [0,d/\beta]}\gamma(\alpha) x+h_{M}(\alpha)$
for $x\in(0,1)$. The unique solution $\bm{u}^\star$ to the variational problem~\eqref{eq:min-opt} is also the unique root to the following nonlinear system of equations:
\begin{align*}
u_{2}&=u_{1}+\phi_{M}(u_{1}),\\
&\cdots \\
u_{m}&=u_{1}+\phi_{M}(u_{m-1}), \\
&\cdots \\
1 = u_{M}&=u_{1}+\phi_{M}(u_{M-1}).
\end{align*}    
\end{proposition}
\noindent See Appendix~\ref{sec:variational} for the proof of this proposition. 
\medskip

This equivalence provides a simple and robust computational procedure.
Two structural properties are immediate:
(1) the final value $u_M$ is a strictly increasing function of $u_1$, and (2) evaluating the univariate  function $\phi_{M}(x)$ amounts to solving the convex problem $\min_{\alpha \in [0,d/\beta]}\gamma(\alpha) x+h_{M}(\alpha)$. 
Consequently, we can recover $\bm{u}^\star$ by a one-dimensional bisection search on $u_1$:
start with an interval $[a,b]\subset(0,1)$ such that $u_M(a)<1<u_M(b)$; iteratively update $u_1$ by halving the interval until $u_M$ computed from the recursive relations above equals $1$ within numerical tolerance.
This routine yields the optimal grid $\bm{u}^\star$ efficiently and stably even for large $M$.

\section{Proof of the lower bound\label{sec:lower}} 

We now establish the lower bound in Theorem~\ref{thm:main}, proving that every $M$-batch policy must suffer regret inflation of at least order $T^{\optexp}$.





\subsection{Proof overview}

Let $\bm{u}^\star$ be a minimizer of the variational problem~\eqref{eq:min-opt}.
By Proposition~\ref{proposition:opt-prop}(ii), there exist margin parameters
$\alpha_1=\infty \ge \alpha_2 \ge \cdots \ge \alpha_M$ with $\alpha_i\le d/\beta$ for $i\ge2$ such that
\begin{equation}\label{eq:margin-selection-again}
\optexp \;=\; \eta_i(\bm{u}^\star,\alpha_i)-h_M(\alpha_i)
\quad\text{for each } i\in[M].
\end{equation}
Define batch cutoffs $T_0=0$, $T_i=\lceil T^{u_i^\star}\rceil$ for $1\le i\le M-1$, and $T_M=T$.
We interpret index $i$ as a \emph{difficulty level}: level~1 has a constant gap (easy), while levels $i\ge2$ have nonincreasing margin parameters (harder).
Crucially, resolving level~$i$ requires on the order of $T^{u_{i-1}^\star}$ samples, which is exactly the size of the preceding window.

The key insight is that an algorithm that does not know which level it is facing must allocate grids sub-optimally for at least one level. 
The proof proceeds in three main steps:

\begin{itemize}
    \item \textbf{Hard-instance construction (Section~\ref{sec:lower-family-design}).} We partition the covariate space into $M$ ``stripes'', each operating at a different resolution. Within stripe $i$, we randomly place ``active cells'' with reward gaps aligned with margin $\alpha_i$. An algorithm cannot determine which cells are active without sufficient exploration.
    \item \textbf{Regret in the pivotal window (Section~\ref{sec:lower-lb-Qi}).} 
    Let $\Gamma = \{0 < t_1 < t_2 < \cdots < t_{M-1} < T\}$ be the (possibly adaptively chosen) grid points by the policy.  
Define the bad events
\begin{equation}\label{eq:bad-event-i}
    A_{i}=\{t_{i-1}<T_{i-1}<T_{i}\le t_{i}\}, \qquad i=1,\ldots, M,
\end{equation}
which form a partition of
the sample space. Roughly speaking, $A_i$ represents the event that the algorithm's chosen grid points are suboptimal between the $(i-1)$-th and $i$-th batch.
We show that if $A_i$ happens, the algorithm incurs large regret on level-$i$ cells within the window $[T_{i-1},T_i]$.
\item \textbf{Indistinguishability across consecutive levels (Section~\ref{sec:mix-indis-proof}).} 
We define mixtures $Q_i$ by randomizing the active sets and signs for levels $m\le i$.
The key is that observations up to time $T_{i-1}$ cannot reliably distinguish whether the environment is drawn from $Q_i$ or $Q_{i+1}$:
\[
\mathrm{TV}\!\big(Q_i^{T_{i-1}},Q_{i+1}^{T_{i-1}}\big)
\ \text{is small (Lemma~\ref{lemma:consecutive-family-tv})}.
\]
This is proved via an equivalent coin model (biased vs.\ fair coins) and a $\chi^2$ bound (Lemma~\ref{lemma:chi-square-bound}), followed by Pinsker.
By a telescoping argument over $i$ and the triangle inequality, there exists an index $i^\star$ such that
\[
\mathbb{P}_{Q_{i^\star}}(A_{i^\star})\ \ge\ \frac{1}{2M}
\qquad\text{(Lemma~\ref{lemma:exist-i})}.
\]
Thus, on at least one level, the pivotal window occurs with nontrivial probability under the corresponding mixture.

\end{itemize}


\begin{figure}         
\centering         
\includegraphics[scale=0.5]{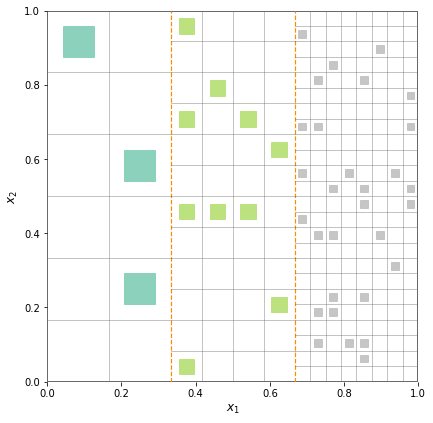}
\caption{Visualization of the active cells when $d=2,M=3$. The domain is partitioned into $M=3$ vertical stripes, each subdivided into fine grids of micro-cells. Colored squares indicate active regions, with each color corresponding to a different stripe resolution parameter $z_m$.}
\label{fig:hard-inst}     
\end{figure}

\subsection{The hard instance construction}\label{sec:lower-family-design}

We now construct the family of hard instances $\{Q_1, Q_2, \ldots, Q_M\}$. 

\paragraph{Step 1: designing the covariate distribution. }
Recall that an instance $P$ dictates a law over $X_t, Y^{(1)}_t, Y^{(-1)}_t$. 
We begin with describing the covariate distribution $P_X$, which is shared among 
$\{Q_1, Q_2, \ldots, Q_M\}$. 

We start with partition the covariate space $[0,1]^{d}$. 
Define 
\[
z_1 = 1, \quad z_m = \lceil 16M^{-1}(M^5T^{u^\star_{m-1}})^{\frac{1}{2\beta+d}} \rceil, \text{ for } 2 \leq m \leq M. 
\]
Split coordinate $x_1$ into $M$ stripes $\mathcal{S}_m=\{x\in[0,1]^d:x_1\in[(m-1)/M,m/M)\}$, $m=1,\ldots,M$. Fix integers $z_m$ and set
\[
w_m=\frac{1}{M z_m},\qquad r_m=\frac{1}{4 M z_m}.
\]
Inside stripe $m$, form an axis-aligned grid of \emph{micro-cells} $\{C_{m,j}\}_{j=1}^{Z_m}$ of side-length $w_m$ by using $z_m$ cuts along $x_1$ (within the stripe) and $M z_m$ cuts along each of the remaining $d-1$ coordinates. Thus
\[
Z_m = z_m\cdot (M z_m)^{d-1} = M^{d-1} z_m^d .
\]
Let $q_{m,j}$ be the center of $C_{m,j}$, and define $\ell_\infty$ balls
\[
B_{m,j}\coloneqq  B_{\infty}(q_{m,j},r_m)\subset C_{m,j}.
\]

With this partition, we define $P_X$ to be the uniform distribution on $\bigcup_{m=1}^M\bigcup_{j=1}^{Z_m} B_{m,j}$. Then
\begin{equation}\label{eq:covariate-density}
 P_X(B_{m,j})=(Mz_m)^{-d}\qquad\text{for all }m,j.   
\end{equation}
It is straightforward to check that $P_X$ obeys Assumption~\ref{ass:bdd-density}.

\paragraph{Step 2: designing the reward family. }

Now we are ready to construct the reward functions. 
Across the families, we will let $f_{(-1)} \equiv \frac{1}{2}$. 
Fix a bump $\phi:[0,\infty)\to[0,1]$:
\[
\phi(r)=\begin{cases}
1,& 0\le r<\tfrac14,\\
2-4r,& \tfrac14\le r<\tfrac12,\\
0,& r\ge\tfrac12.
\end{cases}
\]
For level $m$, define
\begin{equation}
\label{eq:xi}
\xi_{m,j}(x)\coloneqq \delta_m\,\phi^{\beta}\big(M z_m\|x-q_{m,j}\|_\infty\big)\,\mathbf 1{\{x\in C_{m,j}\}},
\qquad \delta_m \coloneqq  D_\phi\,(M z_m)^{-\beta},
\end{equation}
with $D_{\phi}=\min(4^{-\beta}L,1/4)$. Then $\xi_{m,j}$ is supported on $B_{m,j}$, equals $\delta_m$ on the inner quarter, and is $(\beta,L)$-H\"older.

Choose a subset $S_m\subset[Z_m]$ with size 
\begin{equation}
\label{eq:sm}
|S_m| = s_m \coloneqq  \Big\lceil M^{-1}(M z_m)^{\,d-\alpha_m\beta}\Big\rceil ,\quad2\le m\le M, \quad|S_1| = s_1 \coloneqq M^{d-1},
\end{equation}
and attach i.i.d.\ Rademacher signs $\{\sigma_{m,i}\}_{i\in S_m}$. The subset $S_m$ controls which micro-cells are active for the reward function in the $m$-th stripe; see Figure~\ref{fig:hard-inst} for an illustration.

For $i\in[M]$, define the level-$i$ reward family $\mathcal{F}_i$ to be 
\begin{equation}
\label{eq:family}
\mathcal{F}_i = \Big\{\, (f^{(1)}_{S,\sigma,i}(x) = \tfrac12 + \sum_{m=1}^{i}\ \sum_{j\in S_m} \sigma_{m,j}\,\xi_{m,j}(x),\quad f^{(-1)}\equiv\tfrac12):\textrm{for all possible configurations of } S,\sigma\,\Big\}.
\end{equation}

\begin{proposition}\label{prop:smooth-margin}
For every $i$ and $f^{(1)}_{S,\sigma,i}\in \mathcal{F}_i$, the pair $(f^{(1)}_{S,\sigma,i},f^{(-1)})$ is $(\beta,L)$-H\"older and satisfies the margin with parameter $\alpha_i$. Hence $\mathcal{F}_i\subset\Env_{\alpha_i}$.
\end{proposition}
\noindent See Section~\ref{sec:lower-remain-pf} for the proof.

\paragraph{Step 3: designing the mixture $Q_i$. }

For $i\in[M]$, define the \emph{mixture} $Q_i$ by drawing $S_m$ uniformly among subsets of size $s_m$ and i.i.d.\ signs for $m\le i$.

\subsection{Lower bounding the regret  on $Q_{i}$ via indistinguishability}\label{sec:lower-lb-Qi}

Fix $i\in\{1,\dots,M\}$ and recall $T_{i-1}=\lceil T^{u^\star_{i-1}}\rceil$, $T_i=\lceil T^{u^\star_i}\rceil$ with $T_M=T$.
Let $Q_i$ be the level-$i$ mixture from \eqref{eq:family} (random $S_m, \sigma_m$ for $m\le i$) and recall
the bad event $A_i=\{t_{i-1}<T_{i-1}<T_i\le t_i\}$; the $A_i$'s partition the sample space. Our goal is to
show that, \emph{on some $i^\star$ for which $Q_{i^\star}(A_{i^\star})$ is bounded below}, the expected regret
incurred between rounds $T_{i^\star-1}\!+\!1$ and $T_{i^\star}$ is large.

\paragraph{Step 1: Restricting to the pivotal window.}

Since the maximum is larger than the average and the single-step regret is nonnegative, for any policy $(\Gamma, \pi)$, we have 
\begin{align*}
  \sup_{P \in \Env_{\alpha_i}}R_{T} (\Gamma, \pi; P) &\geq \mathbb{E}_{\Env \sim Q_i} [R_{T} (\Gamma, \pi; P)] \\
  &= \mathbb{E}_{\Env \sim Q_i} \left[ \mathbb{E}_{P} \left[\sum_{t=1}^{T}\left(f^{\star}(X_{t})-f^{(\pi_t(X_{t}))}(X_{t})\right)\right] \right] \\
  &\geq \sum_{t=T_{i-1}+1}^{T_i}
\mathbb{E}_{\Env \sim Q_i} \left[ \mathbb{E}_{P} \left[\left(f^{\star}(X_{t})-f^{(\pi_t(X_{t}))}(X_{t})\right)\right] \right].
\end{align*}

\paragraph{Step 2: Localizing to active level-$i$ cells.}

For each $1 \leq m \leq i$, let $S_{m}$ and $\sigma_{m}$ be randomly and uniformly generated, i.e., $S_{m}$ is a random subset of $[Z_{m}]$ with size $s_m$, and $\sigma_{m}$ be a random binary vector. 
Use the definition of $Q_i$ to rewrite
\begin{align*}
\mathbb{E}_{\Env \sim Q_i} \left[ \mathbb{E}_{P} \left[ \left(f^{\star}(X_{t})-f^{(\pi_t(X_{t}))}(X_{t})\right)\right] \right] &=  \mathbb{E}_{\{S_{m}\}_{1 \leq m \leq i}} \mathbb{E}_{\{\sigma_{m}\}_{1 \leq m \leq i}}  
\mathbb{E}_{\Env_{S,\sigma}} \left[ \left(f^{\star}(X_{t})-f^{(\pi_t(X_{t}))}(X_{t})\right)\right] \\
&= \mathbb{E}_{\{S_{m}\}_{1 \leq m \leq i-1}} \mathbb{E}_{\{\sigma_{m}\}_{1 \leq m \leq i-1}} \mathbb{E}_{S_i} \mathbb{E}_{\sigma_i} 
\mathbb{E}_{\Env_{S,\sigma}} \left[ \left(f^{\star}(X_{t})-f^{(\pi_t(X_{t}))}(X_{t})\right)\right].
\end{align*}

On level $i$ the only locations where the two arms differ are the active micro-cells $\{C_{i,j}\}_{j\in S_i}$,
and there the gap equals $\delta_i$ with sign $\sigma_{i,j}\in\{\pm1\}$ (see \eqref{eq:xi} and
\eqref{eq:sm}). Therefore,
\[
\mathbb{E}_{\Env_{S,\sigma}} \left[ \left(f^{\star}(X_{t})-f^{(\pi_t(X_{t}))}(X_{t})\right)\right] \geq \delta_{i} \cdot \mathbb{E}_{\Env_{S,\sigma}} \left[  \sum_{j \in S_{i}} \mathbf{1}\{X_{t} \in C_{i,j}, \pi_t(X_{t}) \neq \sigma_{i,j} \}\right]
\]

Now fix the realization for $\{S_{m}\}_{1 \leq m \leq i}$, $\{\sigma_{m}\}_{1 \leq m \leq i-1}$, and fix any $j \in S_{i}$. 
We aim to lower bound $\mathbb{E}_{\sigma_i} 
 \mathbb{E}_{\Env_{S,\sigma}} \left[ \mathbf{1}\{X_{t} \in C_{i,j}, \pi_t(X_{t}) \neq \sigma_{i,j} \}\right]$. 
Denote by $\sigma_{i,-j}$ the random vector excluding the $j$-th coordinate.
We have
\begin{align*}
 &\mathbb{E}_{\sigma_i} 
 \mathbb{E}_{\Env_{S,\sigma}} \left[ \mathbf{1}\{X_{t} \in C_{i,j}, \pi_t(X_{t}) \neq \sigma_{i,j} \}\right] \\
 &\quad = \frac{1}{2}\mathbb{E}_{\sigma_{i,-j}} \left [\mathbb{P}_{S,\sigma \mid \sigma_{i,j} = 1} (X_{t} \in C_{i,j}, \pi_t(X_{t}) \neq 1 )+ \mathbb{P}_{S,\sigma \mid \sigma_{i,j} = -1} (X_{t} \in C_{i,j}, \pi_t(X_{t}) \neq -1 ) \right] \\
 &\quad = \frac{1}{2 (M z_i)^{d}} \mathbb{E}_{\sigma_{i,-j}} \Big [\underbrace{\mathbb{P}_{S,\sigma \mid \sigma_{i,j} = 1} (\pi_t(X_{t}) \neq 1 \mid X_{t} \in C_{i,j} )+ \mathbb{P}_{S,\sigma \mid \sigma_{i,j} = -1} (\pi_t(X_{t}) \neq -1 \mid X_{t} \in C_{i,j})}_{U^t_{i,j}} \Big],
\end{align*}
where we have used the fact that $P_X ( X_t \in C_{i,j}) =1/(M z_i)^d$. 


\paragraph{Step 3: Localizing the TV to $A_i$ (Le Cam on a subset).}

Define $\mathbb{P}_{\Gamma,\pi;\sigma_{i,j}}^t$ to be the law of observations up to time $t$ under the environment with $\sigma_{i,j}$ and under the policy $(\Gamma,\pi)$. By Le Cam's method, one has
\begin{align*}
U_{i,j}^{t} & \ge1-\|\pminus^{t}-\pplus^{t}\|_{\mathrm{TV}}\\
 & \ge1-\|\pminus^{T_{i}}-\pplus^{T_{i}}\|_{\mathrm{TV}}\\
 & =\int\min\left\{ \mathrm{d}\pminus^{T_{i}},\mathrm{d}\pplus^{T_{i}}\right\} \\
 & \ge\int_{A_{i}}\min\left\{ \mathrm{d}\pminus^{T_{i}},\mathrm{d}\pplus^{T_{i}}\right\} ,
\end{align*}
where the second inequality holds since $t\le T_i$. 
Here we recall that 
$A_i = \{t_{i-1}<T_{i-1}<T_{i}\le t_{i}\}$.
Under $A_i$, the available
observations at $T_{i}$ are the same as those at $T_{i-1}$ under $A_{i}$, we therefore have
\begin{align*}
U_{i,j}^{t}  
 &\ge\int_{A_{i}}\min\left\{ \mathrm{d}\pminus^{T_{i-1}},\mathrm{d}\pplus^{T_{i-1}}\right\} \\
 &=\frac{1}{2}\int_{A_{i}}\left(\mathrm{d}\pminus^{T_{i-1}}+\mathrm{d}\pplus^{T_{i-1}}-|\mathrm{d}\pminus^{T_{i-1}}-\mathrm{d}\pplus^{T_{i-1}}|\right)\\
 &\ge\frac{1}{2}\left(\pminus^{T_{i-1}}(A_{i})+\pplus^{T_{i-1}}(A_{i})\right)-\|\pminus^{T_{i-1}}-\pplus^{T_{i-1}}\|_{\mathrm{TV}}.
\end{align*}
For the TV distance, we have the following bound, whose proof is deferred to Section~\ref{sec:lower-lb-Qi-proof}. 
\begin{lemma}\label{lemma:single-bin-tv}
   Fix any $n\in[T]$ and any policy $(\Gamma, \pi)$. For any $i\in[M]$ and $j \in S_i$, 
   \[\|\pminus^{n}-\pplus^{n}\|_\mathrm{TV}\le \sqrt{n(Mz_i)^{-(2\beta+d)}}.\]
\end{lemma}
Applying Lemma~\ref{lemma:single-bin-tv} with $n=T_{i-1}$, we obtain
\[
U_{i,j}^{t} \geq \frac{1}{2}\left(\pminus^{T_{i-1}}(A_{i})+\pplus^{T_{i-1}}(A_{i})\right) - \singlecordtv.
\]

\paragraph{Step 4: Averaging over $(S_m, \sigma_m)$.}

Combining Steps 1-3, we arrive at 
\begin{align*}
    &\sup_{P \in \Env_{\alpha_i}}R_{T} (\Gamma, \pi; P) \\ 
    & \geq  \frac{\delta_{i}}{2 (M z_i)^{d}} \sum_{t=T_{i-1}+1}^{T_i} \mathbb{E}_{\{S_{m}\}_{1 \leq m \leq i-1}} \mathbb{E}_{\{\sigma_{m}\}_{1 \leq m \leq i-1}} \mathbb{E}_{S_i} \sum_{j \in S_{i}} 
     \mathbb{E}_{\sigma_{i,-j}} \left [\frac{1}{2}\left(\pminus^{T_{i-1}}(A_{i})+\pplus^{T_{i-1}}(A_{i})\right) - \singlecordtv \right] \\
     & = \frac{\delta_{i} \cdot s_i }{2 (M z_i)^{d}} \sum_{t=T_{i-1}+1}^{T_i} \left ( \mathbb{P}_{Q_i}^{T_{i-1}} (A_i) - \singlecordtv \right ) \\
     & = \frac{\delta_{i} \cdot s_i }{2 (M z_i)^{d}} (T_i - T_{i-1}) \left ( \mathbb{P}_{Q_i}^{T_{i-1}} (A_i) - \singlecordtv \right ). 
\end{align*}
Here the first equality essentially uses the definition of $Q_i$.

Since the event $A_i$ can be determined by observations up to $T_{i-1}$, we have $\mathbb{P}_{Q_i}^{T_{i-1}} (A_i)=\mathbb{P}_{Q_i}^{} (A_i)$. It boils down to lower bounding $\mathbb{P}_{Q_i}^{} (A_i)$, for which we have the following lemma. 
\begin{lemma}\label{lemma:exist-i}
There exists some $1 \leq i^\star \leq M$ 
such that 
$\mathbb{P}_{Q_{i^\star}}^{} (A_{i^\star}) \geq 1/(2M)$. 
\end{lemma}

From now on, we identify $i$ with $i^\star$. 
As a result, we have 
\begin{align*}
    \sup_{P \in \Env_{\alpha_i}}R_{T} (\Gamma, \pi; P) \gtrsim  \frac{\delta_{i} \cdot s_i }{2 (M z_i)^{d}} \frac{T_i}{M}.  
\end{align*}
When $i=1$, $s_1=M^{d-1}$, one has
\begin{align*}
    \sup_{P \in \Env_{\alpha_1}}R_{T} (\Gamma, \pi; P) \gtrsim  \frac{\delta_{1} \cdot s_1 }{2 (M z_1)^{d}} \frac{T_1}{M}
    \apprge M^{-3}\cdot T^{u_1^\star}.
\end{align*}

When $i\ge2$, recall that $\delta_i = D_{\phi} (M z_i)^{-\beta}$, $s_i = M^{-1}(Mz_i)^{d - \alpha_i \beta}$, $T_i \asymp T^{u_i^\star}$, and $M z_i \asymp (M^5T^{u^\star_{m-1}})^{\frac{1}{2\beta+d}} $. 
We therefore obtain
\[
\sup_{P \in \Env_{\alpha_i}}R_{T} (\Gamma, \pi; P) \gtrsim \frac{1}{M^{7}}\cdot T^{u_i^\star - u_{i-1}^\star \gamma(\alpha_i)}.
\]
Combining the above relations with the inequality
\begin{align}
\mathrm{RI}(\Gamma,\pi)
= \sup_{\alpha \in \knowledge} \sup_{P \in \EnvAlpha}
\frac{R_T(\Gamma,\pi;P)}{R_T^\star(\alpha)} 
\ge 
\sup_{P \in \mathcal{P}_{\alpha_i}}
\frac{R_T(\Gamma,\pi;P)}{R_T^\star(\alpha_i)}\label{eq:c-ratio-lower-i}
\end{align}
yields 
\begin{align}
    \mathrm{RI}(\Gamma,\pi)
\geq   \sup_{P \in \mathcal{P}_{\alpha_i}}
\frac{R_T(\Gamma,\pi;P)}{R_T^\star(\alpha_i)} 
\apprge \frac{1}{\log T}\cdot M^{\Mexponent}\cdot T^{u_{i}^{\star}-u_{i-1}^{\star}\gamma(\alpha_{i})-h_M(\alpha_i)}
\asymp \frac{1}{\log T}\cdot M^{\Mexponent}\cdot T^{\optexp}.
\end{align}



\subsection{Proving the indistinguishability}\label{sec:mix-indis-proof}

In this section, we aim to demonstrate that the family $\{Q_1, Q_2, \ldots, Q_M\}$ is indistinguishable from finite samples. 
As a consequence, we establish Lemma~\ref{lemma:exist-i}.

The following lemma is the key result of this section, which establishes the fact that for any policy $(\Gamma, \pi)$, given observations up to time $T_{i-1}$, it is not possible to distinguish if the bandit instance is from $Q_i$ or from $Q_{i+1}$.

\begin{lemma}\label{lemma:consecutive-family-tv}
    Fix any policy $(\Gamma, \pi)$. Denote by $Q_{i}^{T_{i-1}}$ the law of observation up to time $T_{i-1}$ under the mixture distribution $Q_i$ and under the policy $(\Gamma, \pi)$. 
    Then for any $1\le i\le M-1$, one has
\[\mathrm{TV}(Q_{i}^{T_{i-1}},Q_{i+1}^{T_{i-1}})
\le\frac{1}{2M^2}\cdot T^{u_{i-1}^{\star}-u_{i}^{\star}\cdot(\gamma_{i+1}^{\star}+\frac{1}{2})}.\]
\end{lemma}

Before diving into the proof of this lemma, we prove Lemma~\ref{lemma:exist-i} based on Lemma~\ref{lemma:consecutive-family-tv}. 
By the triangle inequality, we have 
    \begin{align}\label{eq:tv-to-m}
        \mathrm{TV}(Q_{i}^{T_{i-1}}, Q_{M}^{T_{i-1}}) \le\sum_{m=i}^{M-1} \mathrm{TV}(Q_{m}^{T_{i-1}}, Q_{m+1}^{T_{i-1}})
        &\overset{\mathrm{(i)}}{\le}\sum_{m=i}^{M-1} \mathrm{TV}(Q_{m}^{T_{m-1}}, Q_{m+1}^{T_{m-1}})\nonumber\\
        &\overset{\mathrm{(ii)}}{\le} \frac{1}{2M^2}\sum_{m=i}^{M-1}T^{ u^\star_{m-1}-u_{m}^{\star}\cdot(\gamma_{m+1}^{\star}+\frac{1}{2})}\nonumber\\
        &\overset{\mathrm{(iii)}}{\le}\frac{1}{2M}.
    \end{align}
Here, step (i) uses the fact that $T_{m-1} \geq T_{i-1}$, step (ii) uses Lemma~\ref{lemma:consecutive-family-tv}, and step (iii) uses the fact that $u^\star_{m-1}-u_{m}^{\star}\cdot(\gamma_{m+1}^{\star}+\frac{1}{2})\le0$ from Lemma~\ref{lemma:precondition}.

As a result, we obtain
\begin{equation}
|Q_{M}(A_{i})-Q_{i}(A_{i})|\overset{\mathrm{}}{=}|Q_{M}^{T_{i-1}}(A_{i})-Q_{i}^{T_{i-1}}(A_{i})|\overset{\mathrm{}}{\le}\mathrm{TV}(Q_{M}^{T_{i-1}},Q_{i}^{T_{i-1}})\overset{\mathrm{}}{\le}\frac{1}{2M},\label{eq:Ai-diff}
\end{equation}
where the first step holds since $A_{i}$ can be determined by observations
up to $T_{i-1}$, the second step  uses the definition of TV, and the last step 
is due to relation~\eqref{eq:tv-to-m}.

Consequently,
\begin{align*}
\sum_{i=1}^{M}Q_{i}(A_{i}) & =Q_{M}(A_{M})+\sum_{i=1}^{M-1}Q_{i}(A_{i})\\
 & =Q_{M}(A_{M})+\sum_{i=1}^{M-1}(Q_{i}(A_{i})-Q_{M}(A_{i})+Q_{M}(A_{i}))\\
 & \overset{\mathrm{(iv)}}{\ge}Q_{M}(A_{M})+\sum_{i=1}^{M-1}(Q_{M}(A_{i})-\frac{1}{2M})\ge\sum_{i=1}^{M}Q_{M}(A_{i})-\frac{1}{2}\overset{\mathrm{(v)}}{=}\frac{1}{2},
\end{align*}
where step (iv) uses inequality (\ref{eq:Ai-diff}), and step (v)
uses the fact that $\sum_{i=k}^{M}Q_{M}(A_{i})=1$.
Lemma~\ref{lemma:exist-i} follows from the pigeonhole principle.  

Now we return to the proof of Lemma~\ref{lemma:consecutive-family-tv}.
When $i=1$, one has $T_0=0$ and the statement trivially holds. 
Hence in the remaining proof, we consider $i\ge2$.

\subsubsection{An equivalent coin model}
In this section, we introduce an equivalent coin model to help us control $\mathrm{TV}(Q_{i}^{T_{i-1}},Q_{i+1}^{T_{i-1}})$. 
The coin model is indexed by four parameters: $z$, the number of coins, $s$, the number of possibly biased coins, $\delta$, the effective bias of the coin, and $n$, the total number of tosses.  

Suppose there are $z$ coins labeled by $1, 2, \ldots, z$. We perform $\nsample$ rounds of experiments. In each round $t$: we pick a random coin $I_t \sim \text{Unif}\{1, \ldots, z\}$ independently, flip that coin, and observe the outcome $Y_t \in \{0,1\}$. We define for each coin $i$:
\[
N_i \coloneqq  \sum_{t=1}^N \mathbf{1}\{I_t = i\}, \quad R_i \coloneqq  \sum_{t=1}^N \mathbf{1}\{I_t = i, Y_t = 1\}.
\]
In words, $N_i$ is the number of tosses for coin $i$, and $R_i$ is the number of heads for coin $i$. 

We consider two possible hypothesis for the bias of the coins. 

\paragraph{Null model $H_0$.}
Under $H_0$, every coin is fair: the probability of heads $p_i = 1/2$ for all $i$. Conditional on the number of times a coin was used:
\[
R_i \mid N_i \sim \text{Bin}(N_i, 1/2),
\]
and the different coins' results $(R_i)_{i=1}^z$ are independent given the counts $(N_i)_{i=1}^z$. Denote by $\nullp$ the joint law of the observed data under $H_0$.

\paragraph{Alternative model $H_1$.}
Now, under $H_1$, we introduce a small number of biased coins by randomly choosing a subset of coins
\[
    S \subset [z], \quad |S| = \supportsize,
\]
uniformly among all $\supportsize$-element subsets. 
Also, Let $\sigma_i\in\{\pm 1\}$ be i.i.d.~Rademacher random variables for $i\in[z]$. For coins $i \in S$, they are biased either upwards or downwards,
        \[
        R_i \mid (S, \sigma, N_i) \sim \text{Bin}\left(N_i, \frac{1}{2} + \sigma_i \delta\right),
        \]
where $\delta \in (0, 1/2)$ is the bias magnitude. Coins $i \notin S$ are still fair,
        \[
        R_i \mid (S, \sigma, N_i) \sim \text{Bin}(N_i, 1/2).
        \]
Denote by $P_{S,\sigma}$ the joint law of the observed data given $S,\sigma$. Define $Q=\mathbb{E}_{S,\sigma}[P_{S,\sigma}]$ to be the mixture distribution under $H_1$.

We have the following control on the chi-squared divergence between the null model and the alternative model. 
\begin{lemma}\label{lemma:chi-square-bound}
    Assume that $\nsample\binht^2\nbin^{-1}<o(1)$ and $\nsample^2\supportsize^2\binht^4/\nbin^3\le1/\chisquareconstant$, then
    \[\chi^2(\altq, \nullp)
    \le\frac{\chisquareconstant\nsample^2\supportsize^2\binht^4}{\nbin^3}.\]
\end{lemma}

\subsubsection{Connect $\mathrm{TV}(Q_{i}^{T_{i-1}},Q_{i+1}^{T_{i-1}})$ to the coin model}
Define $Q_i^{\otimes T_{i-1}}$ to be the joint law of the full observations 
\[(X_t,Y_t^{(1)},Y_t^{(-1)}),\quad 1\le t\le T_{i-1},\]
under the mixture distribution $Q_i$. 
It is worth noting that $Q_i^{\otimes T_{i-1}}$ is independent from any policy $(\Gamma,\pi)$, as opposed to $Q_i^{T_{i-1}}$. 

By the data processing inequality, we know that 
\begin{equation}\label{eq:data-proc}
    \mathrm{TV}(Q_i^{T_{i-1}},Q_{i+1}^{T_{i-1}})
\le\mathrm{TV}(Q_i^{\otimes T_{i-1}},Q_{i+1}^{\otimes T_{i-1}}).
\end{equation}

Recall the definitions of $Q_i$ and $Q_{i+1}$. 
We note that they only differ when $X_t\in\mathcal{S}_{i+1}$. 
Due to the Bernoulli reward structure, in this region, $Q_{i+1}$ now corresponds to the alternative model $H_1$ where a subset of coins are biased and $Q_i$ corresponds to the null model $H_0$. 
Since only samples landing into $\mathcal{S}_{i+1}$ can help distinguish $Q_i^{\otimes T_{i-1}}$ and $Q_{i+1}^{\otimes T_{i-1}}$, one has 
\begin{equation}\label{eq:connect-coin}
  \mathrm{TV}(Q_i^{\otimes T_{i-1}},Q_{i+1}^{\otimes T_{i-1}})\le\mathrm{TV}(\nullp, \altq),  
\end{equation}
where $\nullp, \altq$ denote the coin model with parameters 
$\nbin=Z_{i+1}$, $\supportsize=s_{i+1}^{\mathrm{tot}}$, $\binht=\delta_{i+1}$, and $n=T_{i-1}$.  

Under the choice of $\nsample,\binht,\nbin$ and $\supportsize$, it can be verified that $\nsample\binht^2\nbin^{-1}<o(1)$ and $\nsample^2\supportsize^2\binht^4/\nbin^3\le1/\chisquareconstant$. Hence, by Pinsker's inequality and Lemma~\ref{lemma:chi-square-bound}, 
\begin{equation}\label{eq:tv-null-alt}
    \mathrm{TV}(\nullp, \altq)
\le\sqrt{\frac{1}{2}\chi^2(\altq, \nullp)}
\le \sqrt{\frac{16\nsample^2\supportsize^2\binht^4}{\nbin^3}}
=\frac{4\nsample\supportsize\binht^2}{\nbin^{1.5}}.
\end{equation}
Combining relations~\eqref{eq:data-proc}, \eqref{eq:connect-coin} and~\eqref{eq:tv-null-alt}, 
\[
\mathrm{TV}(Q_i^{T_{i-1}},Q_{i+1}^{T_{i-1}})
\le \frac{4\nsample\supportsize\binht^2}{\nbin^{1.5}}
=\frac{4T_{i-1}s_{i+1}\delta_{i+1}^2}{Z_{i+1}^{1.5}}
\le\frac{1}{2M^2}\cdot T^{u_{i-1}^{\star}-u_{i}^{\star}\cdot(\gamma_{i+1}^{\star}+\frac{1}{2})}.
\]
This completes the proof of Lemma~\ref{lemma:consecutive-family-tv}. The remaining of this section is devoted to proving Lemma~\ref{lemma:chi-square-bound}.


\subsubsection{Proof of Lemma~\ref{lemma:chi-square-bound}}

We bound the chi-squared divergence using the second-moment representation, i.e., 
\[
\chi^2(\altq, \nullp)
    =\mathbb{E}_{\nullp}[\mathbb{E}[\Lambda^2\mid \multinomvec]]-1,
\]
where we further condition on $\multinomvec=(N_1,\dots,N_\nbin)$, the multinominal vector which counts the number of times each coin is flipped.

\paragraph{The conditional likelihood ratio.} 
 Conditioned on $\multinomvec$, coins are independent under both $H_0$ and $H_1$. Denote by $\bm{R}=(R_1,\dots,R_\nbin)$.  For the null model,
\[
\nullp(R\mid \multinomvec) = \prod_{i=1}^\nbin \binom{N_i}{R_i} 2^{-N_i}.
\]
Under a fixed $(S,\sigma)$,
\[
P_{S,\sigma}(\nheadvec\mid\multinomvec) = \prod_{i\notin S} \binom{N_i}{R_i} 2^{-N_i}
\prod_{i\in S} \binom{N_i}{R_i}\left(\frac{1}{2}+\sigma_i\binht\right)^{R_i}\left(\frac{1}{2}-\sigma_i\binht\right)^{N_i-R_i}.
\]
Hence the per-coin likelihood ratio factor for $i\in S$ is
\[
r_i(\sigma_i)
= \frac{\left(\frac{1}{2}+\sigma_i\binht\right)^{R_i}\left(\frac{1}{2}-\sigma_i\binht\right)^{N_i-R_i}}{(1/2)^{N_i}}
= (1+2\sigma_i\binht)^{R_i}(1-2\sigma_i\binht)^{N_i-R_i}.
\]
Consequently,
\[
\altq(\nheadvec\mid \multinomvec)
= \frac{1}{\binom{\nbin}{\supportsize}} \sum_{S:\,|S|=\supportsize}
\mathbb{E}_{\sigma}\!\left[\,P_{S,\sigma}(\nheadvec\mid\multinomvec)\,\right].
\]
Dividing by $\nullp(\nheadvec\mid\multinomvec)$ gives
\[
\Lambda(\nheadvec\mid\multinomvec)
= \frac{1}{\binom{\nbin}{\supportsize}} \sum_{S:\,|S|=\supportsize}
\mathbb{E}_{\sigma}\!\left[\prod_{i\in S} r_i(\sigma_i)\right],
\]
where we use $\Lambda$ to denote the likelihood ratio between $\altq$ and $\nullp$.
Because the $\sigma_i$'s are i.i.d., the expectation factorizes:
\[
\mathbb{E}_{\sigma}\!\left[\prod_{i\in S} r_i(\sigma_i)\right]
= \prod_{i\in S} m_i,
\quad\text{where}\quad
m_i \coloneqq  \frac{1}{2}\big(r_i(+1)+r_i(-1)\big).
\]
Putting it together,
\[
\Lambda(\nheadvec\mid\multinomvec)
= \frac{1}{\binom{\nbin}{\supportsize}} \sum_{S:\,|S|=\supportsize} \prod_{i\in S} m_i.
\]

\paragraph{Representation of $\Lambda^2$.} 
To control the chi-square divergence, it suffices to bound the second moment of the likelihood ratio. We compute
\begin{align*}
    \mathbb{E}_{\nullp}[\Lambda^2\mid \multinomvec]
    &=\mathbb{E}_{\nullp}\left[
    \left(\frac{1}{\binom{\nbin}{\supportsize}} \sum_{S\subset[\nbin],\,|S|=\supportsize} \prod_{i\in S} m_i\right)
    \left(\frac{1}{\binom{\nbin}{\supportsize}} \sum_{S^\prime\subset[\nbin],\,|S^\prime|=\supportsize} \prod_{j\in S^\prime} m_j \right)
    \,\Big|\,\multinomvec\right]\\
    &=\frac{1}{\binom{\nbin}{\supportsize}^2}\sum_{\substack{|S|=\supportsize\\|S^\prime|=\supportsize}}
\mathbb{E}_{\nullp}\!\left[\prod_{i\in S} m_i\prod_{j\in S^\prime} m_j\,\Big|\, \multinomvec\right].
\end{align*}
Fix any ordered pair $(S,S^\prime)$. For each index $k\in [z]$, its contribution to the product
$\left(\prod_{i\in S} m_i\right)\left(\prod_{j\in S^\prime} m_j\right)$
depends on which of the sets $S,S^\prime$ it belongs to:
\begin{itemize}
    \item If $k\in S\cap S^\prime$: the factor contributes $m_k \cdot m_k = m_k^2$.
    \item If $k\in S\setminus S^\prime$ or $k\in S^\prime\setminus S$: the factor contributes a single $m_k$.
    \item If $k\notin S\cup S^\prime$: the factor contributes $1$.
\end{itemize}
Define the function $g_b(a)=((1+4b^2)^{a}+(1-4b^2)^{a})/2$ for some $0<b<1/2$. The following lemma helps control the moments of $m$, whose proof is deferred to Section~\ref{sec:chi-sq-remain}.
\begin{lemma}\label{lemma:moment-of-m}
    Under $\nullp$ and conditional on $\multinomvec$,\[
\mathbb{E}_{\nullp}[m_i\mid N_i]=1,\quad
\mathbb{E}_{\nullp}[m_i^2\mid N_i]=\tfrac{1}{2}\Big((1+4\delta^2)^{N_i}+(1-4\delta^2)^{N_i}\Big)=g_\delta(N_i).
\]
\end{lemma}
Hence,
\[
\mathbb{E}_{\nullp}\!\left[\prod_{i\in S} m_i\prod_{j\in S^\prime} m_j\,\Big|\, \multinomvec\right]
=\left(\prod_{i\in S\cap S^\prime}\mathbb{E}[m_i^2]\right)\left(\prod_{i\in S\triangle S^\prime}\mathbb{E}[m_i]\right)
=\prod_{i\in S\cap S^\prime} g_\delta(N_i),
\]
where $S\triangle S^\prime=(S\setminus S^\prime)\cup(S^\prime\setminus S)$ is the symmetric difference.

\paragraph{Averaging over the randomness of $\multinomvec$.}By the law of total expectation, we reach
\begin{align*}
    \mathbb{E}_{\nullp}[\mathbb{E}[\Lambda^2\mid \multinomvec]]
&=\frac{1}{\binom{\nbin}{\supportsize}^2}\sum_{\substack{|S|=\supportsize\\|S^\prime|=\supportsize}}
\mathbb{E}_{\nullp}\left[\prod_{i\in S\cap S^\prime} g_\delta(N_i)\right]\\
&\le\frac{1}{\binom{\nbin}{\supportsize}^2}\sum_{\substack{|S|=\supportsize\\|S^\prime|=\supportsize}}\prod_{i\in S\cap S^\prime}\mathbb{E}_{\nullp}[ g_\delta(N_i)]\\
&=\frac{1}{\binom{\nbin}{\supportsize}^2}\sum_{\substack{|S|=\supportsize\\|S^\prime|=\supportsize}}\left[
\frac{1}{2}\left((1+4\frac{\binht^2}{\nbin})^\nsample+(1-4\frac{\binht^2}{\nbin})^\nsample\right)
\right]^{|S\cap S^\prime|},
\end{align*}
where the second step is due to the negative association of multinomial random variables~\citep{joag1983negative}, and the last step applies the probability generating function of the binomial distribution. We reach
\begin{align*}
     \mathbb{E}_{\nullp}[\mathbb{E}[\Lambda^2\mid \multinomvec]]
     &=\mathbb{E}_{S,S^\prime}\left[(g_{\binht/\sqrt{\nbin}}(\nsample))^{|S\cap S^\prime|}.\right]
\end{align*}
We record a useful lemma  for controlling the generating function of the average intersection size.

\begin{lemma}\label{lemma:inter-moment}
    Let $S,S^\prime$ be indepndent $\supportsize$-subsets of $[\nbin]$ and let $L=|S\cap S^\prime|$. For any $t\ge1$,
    \[\mathbb{E}[t^L]
    \le\exp(\frac{\supportsize^2}{\nbin}(t-1)).\]
\end{lemma}
\noindent See Section~\ref{sec:chi-sq-remain} for the proof. 

Applying Lemma~\ref{lemma:inter-moment}, 
\begin{align}\label{eq:lambda-exp-bound}
     \mathbb{E}_{\nullp}[\mathbb{E}[\Lambda^2\mid \multinomvec]]
     &=\mathbb{E}_{S,S^\prime}\left[(g_{\binht/\sqrt{\nbin}}(\nsample))^{|S\cap S^\prime|}\right]\le\exp\left(\frac{\supportsize^2}{\nbin}(g_{\binht/\sqrt{\nbin}}(\nsample)-1)\right).
\end{align}
Denote by $\epsilon=4\binht^2/\nbin$. By definition,
\begin{align*}
    g_{\binht/\sqrt{\nbin}}(\nsample)
    &=\frac{1}{2}\left((1+\epsilon)^\nsample+(1-\epsilon)^\nsample\right)\\
    &\le\frac{1}{2}\left(\exp(\nsample\epsilon)+\exp(-\nsample\epsilon)\right)\\
    &=\cosh(n\epsilon)=1+\frac{(\nsample\epsilon)^2}{2}+O((\nsample\epsilon)^4),
\end{align*}
where the second step is due to the elementary inequality $1+x\le e^x$, and the last step is by the assumption $n\epsilon<o(1)$. Plugging the above back to~\eqref{eq:lambda-exp-bound},
\begin{align*}
    \mathbb{E}_{\nullp}[\mathbb{E}[\Lambda^2\mid \multinomvec]]
     &\le\exp\left(\frac{\supportsize^2}{\nbin}(g_{\binht/\sqrt{\nbin}}(\nsample)-1)\right)\\
     &\le\exp\left(\frac{\supportsize^2}{\nbin}\cdot (\nsample\epsilon)^2\right)
     =\exp\left(\frac{16\nsample^2\supportsize^2\binht^4}{\nbin^3}\right),
\end{align*}
where the second inequality holds for $\nsample\epsilon$ sufficiently small. \paragraph{Putting things together.}
By the definition of chi-squared divergence,
\begin{align*}
    \chi^2(\altq, \nullp)
    &=\mathbb{E}_{\nullp}[\mathbb{E}[\Lambda^2\mid \multinomvec]]-1\le\exp\left(\frac{16\nsample^2\supportsize^2\binht^4}{\nbin^3}\right)-1 \le \frac{\chisquareconstant\nsample^2\supportsize^2\binht^4}{\nbin^3},
\end{align*}
where the last step is due to the assumption $\nsample^2\supportsize^2\binht^4/\nbin^3\le1/\chisquareconstant$ and the elementary inequality $e^x\le 2x+1$ when $0\le x\le1$.

\section{Discussion}
\label{sec:discussion}

This work provides a complete characterization of the cost of adaptivity in batched nonparametric contextual bandits. 
By introducing the notion of \emph{regret inflation}, we quantify how much additional regret is unavoidable when the margin parameter~$\alpha$ is unknown. 
Our main finding is that this adaptivity cost scales as $T^{\optexp}$, where $\optexp$ is the value of a convex variational problem depending on the number of batches~$M$, the smoothness~$\beta$, and the dimension~$d$. 
The matching upper and lower bounds show that this exponent captures the exact statistical price of limited adaptivity, up to logarithmic factors.

\subsection{Phase transition in batched contextual bandits}

Our results, together with those in~\cite{jiang2025batched} (see also Proposition~\ref{proposition:oracle-rate} in this paper), reveal a sharp \emph{phase transition} in the interplay between batching, adaptivity, and oracle knowledge. 

Figure~\ref{fig:phase-diagram} provides a clear visual summary of this landscape. 
The green curve represents the normalized regret exponent of an oracle algorithm (e.g., \myalg) that knows $\alpha$, while the red curve shows the regret exponent of the optimal adaptive algorithm (e.g., \rbalg). The gap between them, which we characterize by $\optexp$, is precisely the ``price of adaptation.'' Our analysis explains the behavior of this gap, revealing three qualitatively distinct regimes determined by the number of batches $M$ relative to the horizon $T$:

\begin{enumerate}
    \item When $M \le c\log\log T$ for a sufficiently small constant~$c$, the batch constraint itself dominates. 
In this \emph{under-batched regime}, even an oracle algorithm that knows the true margin parameter~$\alpha$ cannot attain the fully online regret rate. 
\item In the intermediate range $c\log\log T < M < C\log\log T$, knowing~$\alpha$ becomes beneficial. 
An oracle algorithm with knowledge of the true margin parameter can achieve the online rate up to logarithmic factors, but ignorance of~$\alpha$ induces regret inflation quantified by~$\optexp$. 
This \emph{transition regime} therefore marks the true adaptivity barrier, where batching and parameter uncertainty interact nontrivially.

\item 
Finally, when $M \ge C\log\log T$ for a sufficiently large constant~$C$, the regret inflation becomes constant. 
In this \emph{fully adaptive regime}, the proposed algorithm \rbalg matches the performance of the online oracle up to polylogarithmic factors, and adaptation to the unknown~$\alpha$ becomes effectively free.  
\end{enumerate}

\begin{figure}[t]
    \centering
    \includegraphics[width=0.65\linewidth]{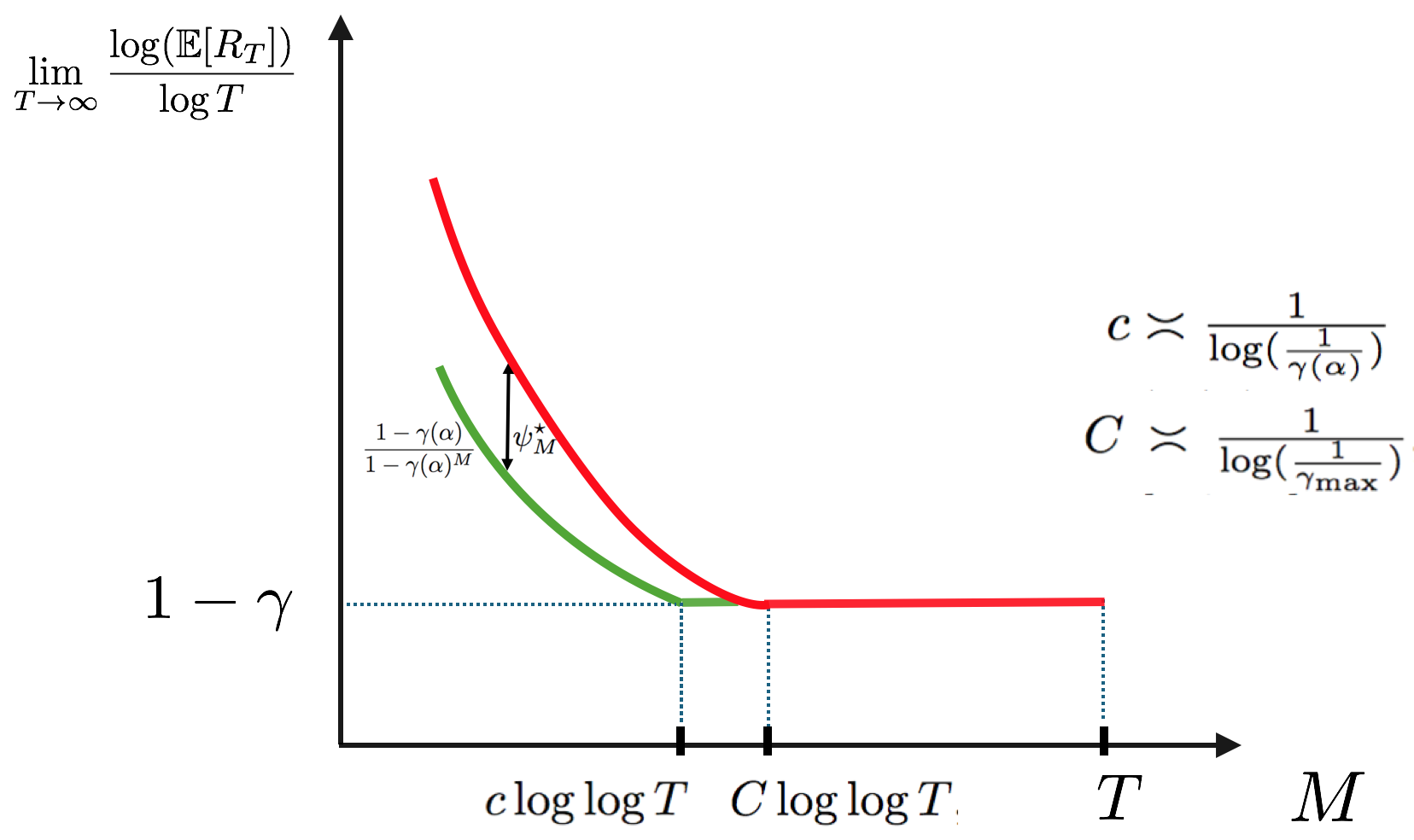}
    \caption{Phase diagram illustrating the three regimes of adaptivity under the batch constraint.  
    The $x$-axis represents the number of batches~$M$, and the $y$-axis shows normalized regret.  
    The green curve corresponds to the performance of the optimal algorithm that knows the margin parameter~$\alpha$, while the red curve corresponds to the optimal adaptive algorithm without knowledge of~$\alpha$. The dotted line shows the exponent for an optimal online algorithm. }
    \label{fig:phase-diagram}
\end{figure}

These three regimes collectively characterize the adaptivity landscape of batched contextual bandits: batching imposes a fundamental statistical cost only below the doubly logarithmic threshold, beyond which adaptation incurs no additional penalty. (More precisely, the lower and upper thresholds satisfy
$c \asymp 1 / \log(1/\gamma(\alpha))$ and $C \asymp 1 / \log(1/\gmax)$,
where $\gmax = \gamma(d/\beta)$ and $\gamma(\alpha) \le \gmax$, giving the ordering $c < C$.)

\subsection{Extension to general knowledge sets}

In practice, one may have partial prior information about the possible range of the margin parameter. 
This can be represented by restricting the \emph{knowledge set}~$\knowledge$ to a subset of $[0, d/\beta]\cup\{\infty\}$, as in Definition~\ref{eq:comp_ratio}. 
Our main results, in particular Theorem~\ref{thm:main}, continue to hold whenever $\knowledge \subseteq [d/(2\beta), d/\beta]\cup\{\infty\}$. 

\newcommand{\optexpK}{\psi_{M,\knowledge}^\star}
\begin{theorem}[General knowledge set]\label{thm:general-K-adaptive}
Fix $\knowledge$ to be a compact subset of $[d/(2\beta), d/\beta]\cup\{\infty\}$. 
Let $\optexpK$ denote the optimal value of the variational problem defined in~\eqref{eq:min-opt} with knowledge set~$\knowledge$.  
Then there exist constants $c_1,c_2>0$, independent of $T$ and $M$, such that
\[
c_1 M^{\Mexponent} (\log T)^{-1} T^{\optexpK} 
\;\le\;
\inf_{\Gamma,\pi}\mathrm{RI}(\Gamma,\pi)
\;\le\;
c_2\, \upperprefactor\, T^{\optexpK}.
\]
\end{theorem}

\noindent
The proof follows closely that of Theorem~\ref{thm:main}, except that Lemma~\ref{lemma:precondition} now holds automatically. 
Indeed, for all $\alpha\in[d/(2\beta),d/\beta]\cup\{\infty\}$, we have 
$\gamma(\alpha)=\frac{(\alpha+1)\beta}{d+2\beta}\ge1/2$, 
which guarantees that the successive hard-instance families $\{Q_1,\ldots,Q_M\}$ cannot be reliably distinguished within a finite number of batches. 
This property---anchored at the boundary point $d/(2\beta)$---preserves the indistinguishability argument used in the lower bound proof. 
By contrast, when the knowledge set $\knowledge$ extends below this threshold, $\gamma(\alpha)$ may fall below $1/2$, breaking the reduction and invalidating the argument.

Nevertheless, if one restricts the $M$-batch policy~$(\Gamma,\pi)$ to use a fixed grid, that is, grid points $(t_1,\ldots,t_{M-1})$ are chosen in advance, then $\optexpK$ still characterizes the optimal regret inflation over this restricted policy class.

\begin{theorem}[General knowledge set, fixed grid]\label{thm:general-K-fixed}
Fix $\knowledge$ to be a compact subset of $[0, d/\beta]\cup\{\infty\}$. 
Then there exist constants $c_1,c_2>0$, independent of $T$ and $M$, such that
\[
c_1 M^{\Mexponent} (\log T)^{-1} T^{\optexpK}
\;\le\;
\inf_{\Gamma,\pi}\mathrm{RI}(\Gamma,\pi)
\;\le\;
c_2\, \upperprefactor\, T^{\optexpK},
\]
where the infimum is taken over all $M$-batch policies with a fixed grid.
\end{theorem}

\noindent
The upper bound follows directly from the performance guarantee of \rbalg. 
For the lower bound, since the policy class is restricted to fixed grids, the indistinguishability argument is unnecessary---one can directly choose the worst-case family $Q_i$ for the bad event~$A_i$ (see Equation~\eqref{eq:bad-event-i}). 
We omit the detailed proof for brevity.

\subsection{Monotonicity of $\optexp$ in $M$}
Empirically, $\optexp$ decreases monotonically with the number of batches~$M$ (see Section~\ref{sec:numerics}).
This behavior is not immediate from the definition of the variational objective in~\eqref{eq:min-opt}:
while adding batches allows a finer schedule and hence lower regret, the benchmark $h_M(\alpha)$ also shrinks, making the improvement subtle to analyze.
For small $M$, we can establish monotonicity formally:
\begin{proposition}\label{prop:monot}
For $M\le4$, we have $\psi_{M+1}^\star < \optexp$.
\end{proposition}
\noindent
The proof appears in Section~\ref{sec:monot-proof}.
Whether this monotonicity extends to all $M$ remains an interesting open question.


\subsection{Future directions}

Several natural extensions remain open.  
First, our analysis focuses on the two-armed setting; extending these results to multiple arms is a natural next step.  
Second, when both the margin and smoothness parameters are unknown, one may ask whether a unified adaptive strategy is possible.  
Smoothness adaptation is known to require additional structure or to incur unavoidable penalties even in the fully online case~\citep{qian2016random,gur2022smoothness,cai2024transfer}.  
Under batch constraints, the challenge intensifies because the size of the initial batch must be determined before smoothness can be estimated.  
A joint variational formulation over~$(\alpha,\beta)$ may shed light on this problem.

Finally, the regret inflation framework may extend to other batched bandit models.  
For example, in high-dimensional sparse linear contextual bandits~\citep{ren2023dynamic}, existing algorithms rely on knowing an upper bound on the sparsity level.  
Quantifying the price of unknown sparsity or other structural complexity parameters would further broaden the reach of our theory.

\subsection*{Acknowledgements} 
C.M. was partially supported by the National Science Foundation via grant DMS-2311127 and the CAREER Award DMS-2443867.

\bibliography{Batch}

\newpage
\appendix

\section{Proof of oracle regret bounds (Proposition~\ref{proposition:oracle-rate})} \label{sec:lower-bound-extension}

We split the proof into two cases: $\alpha = \infty$ and $\alpha \leq d/\beta$.

\subsection{The case of $\alpha=\infty$}
Without loss of generality, let $\mu_1 = \mu^*$ and $\mu_2 = \mu^* - \Delta$ be the mean rewards of two arms, where $\Delta > 0$ is a fixed constant. Clearly, one has $R_T^\star(\alpha)\ge c_1$ for some constant $c_1>0$.

\paragraph{A simple policy. }To achieve the upper bound, we consider the following procedure. Given any batch budget $M\ge2$, we choose to use two batches by setting $t_1\asymp\log(T)$. Define the confidence radius
\[
r(t_1) = \sqrt{\frac{\log(4T/\delta)}{2t_1}},
\] 
where $\delta=1/T^2$. During the first batch, we pull each arm in a round-robin fashion. At the end of batch 1, we eliminate any arm $i\in\{1,2\}$ such that
\[
\widehat{\mu}_i(t_1) + r(t_1) < \max_{j\in\{1,2\}} \big\{ \widehat{\mu}_j(t_1) - r(t_1) \big\}.
\]
During the second batch, we just pull any active arm.

\paragraph{Regret analysis. }
Now we establish the regret guarantee of the above policy. The following lemma ensures that with high probability, the suboptimal arm is eliminated.
\begin{lemma}\label{lemma:se-mab}
With probability at least $1 - \delta$, the suboptimal arm is eliminated by phase
\[
t_1 = \left\lceil \frac{8}{\Delta^2} \log\!\left(\frac{4T}{\delta}\right) \right\rceil.
\]
\end{lemma}
\begin{proof}
By Hoeffding's inequality for bounded rewards,
\[
\Pr\!\big(|\widehat{\mu}_i(t_1) - \mu_i| > r(t_1)\big) \le \frac{\delta}{2T}.
\]
Taking a union bound over both arms,
\[
\mathcal{E} = \Big\{i\in\{1,2\}:\ |\widehat{\mu}_i(t_1) - \mu_i| \le r(t_1)\Big\}
\]
holds with probability at least $1 - \delta$.

On the event $\mathcal{E}$,
\[
\widehat{\mu}_1(t_1) - \widehat{\mu}_2(t_1)
\ge (\mu_1 - \mu_2) - |\widehat{\mu}_1(t_1) - \mu_1| - |\widehat{\mu}_2(t_1) - \mu_2|
\ge \Delta - 2r(t_1).
\]
If $r(t_1) \le \Delta/4$, then
\[
\widehat{\mu}_1(t_1) - \widehat{\mu}_2(t_1) \ge \Delta/2 > 2r(t_1),
\]
which implies
\[
\widehat{\mu}_2(t_1) + r(t_1) < \widehat{\mu}_1(t_1) - r(t_1),
\]
so the suboptimal arm (arm 2) is eliminated at phase $t_1$.

The condition $r(t_1) \le \Delta/4$ means
\[
\sqrt{\frac{\log(4T/\delta)}{2t_1}} \le \frac{\Delta}{4}
\quad \Longleftrightarrow \quad
t_1 \ge \frac{8}{\Delta^2} \log\!\left(\frac{4T}{\delta}\right).
\]
Thus, on event $\mathcal{E}$ (which holds with probability at least $1-\delta$), arm 2 is eliminated by $t_1$.
\end{proof}
Denote by $\mathcal{G}$ the event that the suboptimal arm is eliminated by $t_1$. By Lemma~\ref{lemma:se-mab}, we have
\begin{align*}
    R_T(\bar{\pi}, \Env)\le t_1\cdot\Delta +(T-t_1)\cdot\Delta\cdot\delta\le c_2t_1,
\end{align*}
where we have used the fact that during the second batch  regret is only incurred when $\mathcal{G}^c$ occurs and $\mathbb{P}(\mathcal{G}^c)\le\delta=1/T^2$.

\subsection{The case of $\alpha\le d/\beta$}

For the remaining of the proof, we establish the result for $\alpha\le d/\beta$, which is stated in the following proposition.

\begin{proposition}\label{thm:lower-bound-extend}

Suppose that $\alpha\le d/\beta$. Under Assumptions~\ref{ass:bdd-density}-\ref{ass:margin}. For any $M$-batch policy
$(\Gamma,\pi)$, one has 
\[
M^{\Mexponentoracle}\cdot  T^{h_M(\alpha)}
\apprle\sup_{\Env\in\EnvAlpha}R_{T}(\pi,\Env)
\apprle  M^{} (\log T) \cdot T^{\,h_M(\alpha)}.
\]

\end{proposition}
We use the remaining of the section to prove the above proposition.
\subsection{Proof of the upper bound}
Define
\begin{align}
g_{0}=\lfloor b^{\frac{1}{2\beta+d}}\rfloor,\qquad\text{and}\qquad g_{i}=\lfloor g_{i-1}^{\gamma}\rfloor,i=1,...,M-2. \label{eq:split-factors-oracle}
\end{align}
Denote by $w_i=(\prod_{l=0}^{i-1}g_{l})^{-1}$. In addition, define
\begin{align}
t_{i}-t_{i-1} & =\lfloor l_{i}w_{i}^{-(2\beta+d)}\log(Tw_{i}^{d})\rfloor,1\le i\le M-1,\label{eq:batch-bin-size-oracle}
\end{align}
for $l_{i}>0$ sufficiently large. Let $(\hat{\Gamma}, \hat{\pi})$ be the policy of running \myalg under the above grid choice. By Lemma~\ref{lem:basedb-regret}, for any $P\in\EnvAlpha$,

\begin{align*}
R_{T}(\hat{\Gamma},\hat{\pi};P) 
& \apprle t_{1}+\sum_{i=2}^{M-1}(t_{i}-t_{i-1})\cdot w_{i-1}^{\beta+\alpha\beta}+(T-t_{M-1})w_{M-1}^{\beta+\alpha\beta}.
\end{align*}
Under the choices for the
batch size and the split factors in~(\ref{eq:batch-bin-size-oracle})-(\ref{eq:split-factors-oracle}),  
\begin{align*}
t_{1} & \lesssim T^{\frac{1-\gamma}{1-\gamma^{M}}}\log T,\\
(t_{i}-t_{i-1})\cdot w_{i-1}^{\beta+\alpha\beta} & \lesssim T^{\frac{1-\gamma}{1-\gamma^{M}}}\log T,\qquad\text{for }2\leq i\leq M-1,\\
(T-t_{M-1})w_{M-1}^{\beta+\alpha\beta} & \leq Tw_{M-1}^{\beta+\alpha\beta}\lesssim T^{\frac{1-\gamma}{1-\gamma^{M}}}\log T.
\end{align*}
Combining the above three bounds completes the proof.

\subsection{Proof of the lower bound}
The proof mainly follows the strategy outlined in~\cite{jiang2025batched}, but with a slightly different reward function construction to handle the wider range of $\alpha$.

\subsubsection{Construction of the hard instances}
Define $b\asymp T^{(1-\gamma)/(1-\gamma^{M})}$. For each $1 \leq m \leq M$, we set $T_{m}=\lfloor b^{(1-\gamma^{m})/(1-\gamma)}\rfloor$. Besides, define  
\[z_1=1,\quad z_{m}=\lceil M^{-1}(36T_{m-1}M^{2})^{1/(2\beta+d)}\rceil,\quad\textrm{for } 2\le m\le M.\]

\paragraph{Constructing the covariate distribution.} Split coordinate $x_1$ into $M$ stripes $\mathcal{S}_m=\{x\in[0,1]^d:x_1\in[(m-1)/M,m/M)\}$, $m=1,\ldots,M$. Fix integers $z_m$ and set
\[
w_m=\frac{1}{M z_m},\qquad r_m=\frac{1}{4 M z_m}.
\]
Inside stripe $m$, form an axis-aligned grid of \emph{micro-cells} $\{C_{m,j}\}_{j=1}^{Z_m}$ of side-length $w_m$ by using $z_m$ cuts along $x_1$ (within the stripe) and $M z_m$ cuts along each of the remaining $d-1$ coordinates. Thus
\[
Z_m = z_m\cdot (M z_m)^{d-1} = M^{d-1} z_m^d .
\]
Let $q_{m,j}$ be the center of $C_{m,j}$, and define $\ell_\infty$ balls
\[
B_{m,j}\coloneqq B_{\infty}(q_{m,j},r_m)\subset C_{m,j}.
\]

With this partition, we define $P_X$ to be the uniform distribution on $\bigcup_{m=1}^M\bigcup_{j=1}^{Z_m} B_{m,j}$. Then
\begin{equation}\label{eq:covariate-density-oracle}
 P_X(B_{m,j})=(Mz_m)^{-d}\qquad\text{for all }m,j.   
\end{equation}
It is straightforward to check that $P_X$ obeys Assumption~\ref{ass:bdd-density}.

\paragraph{Designing the reward family. }

Now we are ready to construct the reward functions. 
Across the families, we will let $f_{(-1)} \equiv \frac{1}{2}$. 
Fix a bump $\phi:[0,\infty)\to[0,1]$:
\[
\phi(r)=\begin{cases}
1,& 0\le r<\tfrac14,\\
2-4r,& \tfrac14\le r<\tfrac12,\\
0,& r\ge\tfrac12.
\end{cases}
\]
For level $m$, define
\begin{equation}
\label{eq:xi}
\xi_{m,j}(x)\coloneqq \delta_m\,\phi^{\beta}\big(M z_m\|x-q_{m,j}\|_\infty\big)\,\mathbf 1{\{x\in C_{m,j}\}},
\qquad \delta_m \coloneqq  D_\phi\,(M z_m)^{-\beta},
\end{equation}
with $D_{\phi}=\min(4^{-\beta}L,1/4)$. Then $\xi_{m,j}$ is supported on $B_{m,j}$, equals $\delta_m$ on the inner quarter, and is $(\beta,L)$-H\"older.

Choose a subset $S_m\subset[Z_m]$ with size 
\begin{equation}
\label{eq:sm}
|S_m| = s_m \coloneqq  \Big\lceil M^{-1}(M z_m)^{\,d-\alpha\beta}\Big\rceil ,\quad2\le m\le M, \quad|S_1| = s_1 \coloneqq M^{d-1}.
\end{equation}
Let $Z=\sum_{m=1}^M Z_m$. Denote by $\Omega=\{\pm1\}^Z$. We define the reward family $\mathcal{F}$ to be 
\begin{equation}
\label{eq:family-oracle}
\mathcal{F} = \Big\{\, (f^{(1)}_{\omega}(x) = \tfrac12 + \sum_{m=1}^{M}\ \sum_{j\in S_m} \omega_{m,j}\,\xi_{m,j}(x),\quad f^{(-1)}\equiv\tfrac12):\omega\in\Omega\,\Big\}.
\end{equation}
By Proposition~\ref{prop:smooth-margin}, we have $\mathcal{F}\subset\EnvAlpha$.

\subsubsection{Lower bounding the regret during the $m$-th batch}

Since the worst-case regret is lower bounded by the average regret over the family $\Omega$, 
\begin{align}
\sup_{(f,\frac{1}{2})\in\instfamPREV}&R_{T}(\pi,f)\nonumber \\
&\ge\mathbb{E}_{\omega\sim\omgdistPREV}\eoverpiPREV\left[\sum_{t=1}^{T}\left(f^{\star}(X_{t})-f^{(\pi_t(X_{t}))}(X_{t})\right)\right]\nonumber \\
 & \overset{\mathrm{(i)}}{\ge}\sum_{t=T_{m-1}+1}^{T_{m}}\sum_{j\in S_m}\mathbb{E}_{\omega\sim\omgdistPREV}\eoverpiPREV^{t}\left[D_{\phi}(Mz_{m})^{-\beta}\mathbf{1}\{X_{t}\in\peakmPREV,\pi_{t}(X_{t})\neq\omega_{m,j}^{}\}\right]\nonumber \\
 & =D_{\phi}(Mz_{m})^{-\beta-d}\sum_{t=T_{m-1}+1}^{T_{m}}\sum_{j\in S_m}\frac{1}{2^{Z}}\sum_{\wlooPREV\in\omglooPREV}\underbrace{\sum_{l\in\{\pm1\}}\mathbb{E}_{\pi,\omega_{m,j}^{}=l}^{t}P_{X}(\pi_{t}(X_{t})\neq l\mid X_{t}\in\peakmPREV)}_{U_{m,j}^{t}}.\label{eq:regret-to-test}
\end{align}
Here, step (i) uses the fact that regret is only incurred on $\peakmPREV$'s and the optimal action is specified by $\omega_{m,j}$; we use $\omega_{-(m,j)}$ to represent the vector after leaving out the $j$-th entry in the $m$-th block of $\omega$.
By Le Cam's method, one has
\begin{align*}
U_{m,j}^{t} & \ge1-\|\pminusPREV^{t}-\pplusPREV^{t}\|_{\mathrm{TV}}\\
 & \ge1-\|\pminusPREV^{T_{m}}-\pplusPREV^{T_{m}}\|_{\mathrm{TV}}\\
 & =\int\min\left\{ \mathrm{d}\pminusPREV^{T_{m}},\mathrm{d}\pplusPREV^{T_{m}}\right\} \\
 & \ge\int_{A_{m}}\min\left\{ \mathrm{d}\pminusPREV^{T_{m}},\mathrm{d}\pplusPREV^{T_{m}}\right\} ,
\end{align*}
where the second inequality  is due to $t\le T_m$. Since the available
observations for $\pi$ at $T_{m}$ are the same as those at $T_{m-1}$ under $A_{i}$, we continue to lower bound
\begin{align*}
U_{m,i,j}^{t}  
 &\ge\int_{A_{m}}\min\left\{ \mathrm{d}\pminusPREV^{T_{m-1}},\mathrm{d}\pplusPREV^{T_{m-1}}\right\} \\
 &=\frac{1}{2}\int_{A_{m}}\left(\mathrm{d}\pminusPREV^{T_{m-1}}+\mathrm{d}\pplusPREV^{T_{m-1}}-|\mathrm{d}\pminusPREV^{T_{m-1}}-\mathrm{d}\pplusPREV^{T_{m-1}}|\right)\\
 &\ge\frac{1}{2}\left(\pminusPREV^{T_{m-1}}(A_{m})+\pplusPREV^{T_{m-1}}(A_{m})\right)-\|\pminusPREV^{T_{m-1}}-\pplusPREV^{T_{m-1}}\|_{\mathrm{TV}}\\
 &\ge\frac{1}{2}\left(\pminusPREV(A_{m})+\pplusPREV(A_{m})\right)-\frac{1}{2M},
\end{align*}
where the last step applies Lemma~\ref{lemma:single-bin-tv}.

Plugging the above back to~(\ref{eq:regret-to-test}), we obtain
\begin{align}\label{eq:reg-to-m}
\sup_{f\in\instfamPREV}& R_{T}(\pi,f) \nonumber\\& \ge D_{\phi}(Mz_{m})^{-(\beta+d)}\sum_{t=T_{m-1}+1}^{T_{m}}\sum_{j\in S_m}\frac{1}{2^{Z+1}}\sum_{\wlooPREV\in\omglooPREV}\left(\pminusPREV(A_{m})+\pplusPREV(A_{m})-\frac{1}{M}\right)\nonumber \\
 & =D_{\phi}(Mz_{m})^{-(\beta+d)}\sum_{t=T_{m-1}+1}^{T_{m}}\sum_{j\in S_m}\frac{1}{2}\left(\mathbb{E}_{\omega\sim\omgdistPREV}\mathbb{P}_{\pi,\omega}(A_{m})-\frac{1}{2M}\right)\nonumber\\
 & =\frac{1}{2}D_{\phi}(Mz_{m})^{-(\beta+d)}(T_{m}-T_{m-1})s_{m}\left(\mathbb{E}_{\omega\sim\omgdistPREV}\mathbb{P}_{\pi,\omega}(A_{m})-\frac{1}{2M}\right)\nonumber.
\end{align}
Since $\sum_{k=1}^M\mathbb{E}_{\omega\sim\omgdistPREV}\mathbb{P}_{\pi,\omega}(A_{k})\ge1$, there exists some $m^\star\in[M]$ such that $\mathbb{E}_{\omega\sim\omgdistPREV}\mathbb{P}_{\pi,\omega}(A_{m^\star})\ge1/M$. 
When $m^\star=1$, one has,
\[\sup_{f\in\instfamPREV} R_{T}(\pi,f)
     \apprge M^{-3}\cdot T_1
     \asymp M^{-3}\cdot T^{h_M(\alpha)}\]

When $m^\star\ge2$,
 \begin{align*}
     \sup_{f\in\instfamPREV} R_{T}(\pi,f)
     &\apprge M^{-2}T_{m^\star}(Mz_{m^\star})^{-\beta(1+\alpha)}\\
     &\asymp M^{-2}T_{m^\star}(M^2T_{m^\star-1})^{-\gamma(\alpha)}
     \apprge M^{\Mexponentoracle}\cdot T^{h_M(\alpha)}.
 \end{align*}

\section{Detailed analysis of the variational problem}\label{sec:variational}
We recall that
\[
\psi_M(\bm{u})\;=\;\sup_{\alpha\in\knowledge}
\Psi_{M}(u,\alpha),
\]
where
\[
\Psi_{M}(\bm{u}, \alpha)
=\max\Bigl\{\,u_1,\ u_2-\gamma(\alpha) u_1,\ \dots,\ 
u_{M-1}-\gamma(\alpha) u_{M-2},\ 1-\gamma(\alpha) u_{M-1}\Bigr\}
-h_M(\alpha),
\]
with $\gamma(\alpha)=\tfrac{(\alpha+1)\beta}{2\beta + d}$ for $\alpha\in[0,d/\beta]$,
and $\Psi_{M}(\bm{u},\infty)=u_1$. Here, 
\[\bm{u} \in 
    \gridset=\{\bm{u}\in\mathbb{R}^{M-1}:0\le u_{1}\le\cdots\le u_{M-1}\le1\}.\]

In this section, we collect several useful facts of the variational problem.

\paragraph{Convexity of $\psi_M(\bm{u})$.} For each fixed $\alpha \in \knowledge$, the payoff function $\Psi_{M}(\bm{u}, \alpha)$ is piecewise linear, and hence convex in $\bm{u}$.
As a result,  $\psi_M(\bm{u})=\sup_{\alpha\in\knowledge}\Psi_{M}(\bm{u}, \alpha)$ is a convex function.

\paragraph{Existence of minimizer.}  Note that  $\Psi_{M}(\bm{u}, \alpha):\gridset\times\knowledge\rightarrow\mathbb{R}$ is jointly continuous in $\bm{u}$ and $\alpha$.
We can apply Berge's maximum theorem to show $\psi_M(\bm{u})=\sup_{\alpha\in\knowledge}\Psi_{M}(\bm{u}, \alpha)$ is continuous on $\gridset$. Consequently, by the Weierstrass extreme value theorem, there exists some $\bm{u}^\star\in\gridset$ such that 
$\psi_{M}(\bm{u}^\star)=\psi_{M,\knowledge}^{\star}$.

\paragraph{Positive optimal value. }
We know that for every $\alpha \in \knowledge$, 
\[
\inf_{\bm{u} \in \gridset} \Psi_{M}(\bm{u}, \alpha) = 0,
\]
and 0 is achievable by some $\bm{u}^\star(\alpha) \in \gridset$.
We also know that for any $\alpha_1 \neq \alpha_2 \in \knowledge$, $\bm{u}^\star(\alpha_1) \neq \bm{u}^\star(\alpha_2)$.

Now suppose that $\optexp = 0$, and let $\bm{u}^\star$ be the minimizer, whose existence has been shown above. 
Then we have
\[
\psi_{M}(\bm{u}^\star) = \sup_{\alpha \in \knowledge} \Psi_M(\bm{u}^\star, \alpha) = 0.
\]
That is, for every $\alpha \in \knowledge$, we have $\Psi_{M}(\bm{u}^\star, \alpha) \leq 0$. 
Taking the previous displays together, we arrive at the conclusion that 
\[
\Psi_{M}(\bm{u}^\star, \alpha) = 0, \qquad \text{for all }\alpha \in \knowledge.
\]
However, this contradicts with the fact that 
for different $\alpha$'s, we have different minimizers. 
As a result, we necessarily have $\optexp > 0$.

\paragraph{Subdifferential. }
By the rule of the subdifferential, we know that 
\[
\partial \psi_M(\bm{u})
=\operatorname{conv}\Biggl(\;\bigcup_{\alpha\in\mathcal A(\bm{u})}
\ \partial_{\bm{u}} \Psi_{M}(\bm{u},\alpha)\;\Biggr),
\]
where $\mathcal A(\bm{u})=\{\alpha:\Psi_{M}(\bm{u},\alpha)=\psi_M(\bm{u})\}$ denotes the set of
active maximizers in the $\sup$.

Now we move on to $\partial_{\bm{u}} \Psi_{M}(\bm{u},\alpha)$. 
For each $\alpha < \infty$ the inner maximum has affine pieces with gradients
\[
\bm{g}_1 = \bm{e}_1,\qquad
\bm{g}_i(\gamma) = \bm{e}_i-\gamma \bm{e}_{i-1}\quad (i=2,\dots,M-1),\qquad
\bm{g}_M(\gamma)=-\gamma \bm{e}_{M-1},
\]
where $\bm{e}_i$ is the $i$-th standard basis vector in $\mathbb R^{M-1}$.
At $\alpha=\infty$ only the block $u_1$ is active, with gradient $\bm{g}_1=\bm{e}_1$. 
Therefore,
\[
\partial \psi_M(\bm{u})=\operatorname{conv}\Bigl\{\,
\bm{g}_i(\gamma(\alpha)):\ \alpha\in\mathcal A(\bm{u}),\ i\in\mathcal I(\bm{u},\alpha)\,\Bigr\},
\]
where $\mathcal I(\bm{u},\alpha)$ is the set of indices $i$ attaining the
max in $\Psi_{M}(\bm{u},\alpha)$.

Carath\'eodory's theorem in $\mathbb R^{M-1}$ implies that any point of
$\partial \psi_M(\bm{u})$ can be represented as a convex combination of at most
$M$ vectors. Concretely, 
for any $ \bm{v} \in\partial \psi_M(\bm{u})$ there exist pairs $(\alpha_k,i_k)$ with
$\alpha_k\in\mathcal A(\bm{u})$, $i_k\in\mathcal I(\bm{u},\alpha_k)$ and weights
$\theta_k\ge0$, $\sum_{k=1}^M\theta_k=1$, such that
\begin{align}\label{eq:convex-gradient}
    \bm{v}=\sum_{k=1}^M \theta_k\, \bm{g}_{i_k}(\gamma(\alpha_k)).
\end{align}
Note that if $\alpha_{k} = \infty$, we must have $i_{k} = 1$.

We record a useful property of this subdifferential. 
\begin{lemma}\label{lemma:no-missing-dir}
 For any $\bm{u}$ with  $\mathbf{0}\in\partial\psi_M(\bm{u})$,  we have for each $1 \leq i \leq M$, there exists some $\alpha_i\in \knowledge$ such that 
 $\eta_i(\bm{u},\alpha_i)-h_M(\alpha_i)= \psi_{M}(\bm{u})$.
\end{lemma}
\begin{proof}
    By equation~\eqref{eq:convex-gradient}, there exist pairs $(\alpha_k,i_k)$ with
$\alpha_k\in\mathcal A(\bm{u})$, $i_k\in\mathcal I(\bm{u},\alpha_k)$ and weights
$\theta_k\ge0$, $\sum_{k=1}^M\theta_k=1$, such that
\begin{align*}
    \bm{0}=\sum_{k=1}^M \theta_k\, \bm{g}_{i_k}(\gamma(\alpha_k)).
\end{align*}
Since $\sum_{k=1}^M\theta_k=1$, there exists some $1 \leq k \leq M$ such that $\theta_k > 0$. 
Let $i_{k} $ for the corresponding index for $\bm{g}$, i.e., $\bm{g}_{i_k}$ is included in the convex combination. 
Suppose that $i_k = 1$. 
By the structure of $\bm{g}_{1}$, we know that $\theta_{k} \bm{g}_{1}$ is positive in the first entry. 
To cancel this, we must have $\bm{g}_2$ in the convex combination, which further brings $\bm{g}_3$ into the convex combination.
Chaining this argument, we arrive at the conclusion that all $\{\bm{g}_{i}\}_{1\leq i \leq M}$ must be involved in the convex combination. 
The argument continues to hold if $i_k \geq 2$.

Since the set $\{i_k\}_{1 \leq k \leq M} = \{1,2,\ldots, M\}$, by the definition of $(\alpha_k, i_k)$, we know that for each $1 \leq i \leq M$, we have some $\alpha_i \in \knowledge$ such that 
$\eta_i(\bm{u},\alpha_i)-h_M(\alpha_i)= \psi_{M}(\bm{u})$.
\end{proof}

\paragraph{Explicit conic representation of $N_{\mathcal U}(\bm{u})$.}
Define $u_0\coloneqq 0$ and $u_M\coloneqq 1$.
For $i=1,\dots,M$, set
\[
\bm{d}_i \;\coloneqq \; \bm{e}_{i-1}-\bm{e}_i \in \mathbb R^{M-1},
\qquad\text{with the convention } \bm{e}_0\coloneqq \bm{0},\ \bm{e}_M\coloneqq \bm{0}.
\]
Then define the active set at $\bm{u}$ as
\[
I(\bm{u})\;\coloneqq \;\{\, i\in\{1,\dots,M\}:\ u_i =  u_{i-1} \,\}.
\]
The normal cone is the conic hull of the active normals:
\begin{equation}\label{eq:normal-cone-succinct}
N_{\mathcal U}(\bm{u})
\;=\;
\Bigl\{\, \bm{n}\in\mathbb R^{M-1}:\ 
\bm{n}=\sum_{i\in I(\bm{u})}\lambda_i\,\bm{d}_i,\ \ \lambda_i\ge 0 \Bigr\}.
\end{equation}

\paragraph{The first-order optimality condition.}

The KKT condition $\,\bm{0}\in\partial \psi_{M}(\bm{u}^\star)+N_{\mathcal U}(\bm{u}^\star)\,$
is therefore equivalent to the existence of multipliers
$\{\lambda_i\}_{i=1}^{M}$ with $\lambda_i\ge 0$ and $\lambda_i=0$ if $i\notin I(\bm{u}^\star)$,
and weights $\{\theta_k\}_{k=1}^{M}$ as above, such that
\[
\underbrace{\sum_{k=1}^{M}\theta_k\,\bm{g}_{i_k}\bigl(\gamma(\alpha_k)\bigr)}_{\eqqcolon \bm{v}} \;+\; \underbrace{\sum_{i\in I(\bm{u}^\star)}\lambda_i\, (\bm{e}_{i-1} - \bm{e}_{i} )}_{\eqqcolon \bm{n}}\;=\;\bm{0}.
\]

Now, we are ready to establish an important property about the variational problem.

\subsection{The minimizer lies in the interior}
While we have  demonstrated the existence of a minimizer in $\mathcal{\gridset}$, we can actually show a stronger statement that the minimizer cannot be on the boundary. This fact will be crucial for establishing many subsequent properties. 

For the sake of contradiction, assume that $\bm{u}^\star$ is a minimizer lying on the boundary of $\gridset$. 
By definition, there exists some index $1 \leq j \leq m$ such that 
$
u^\star_j = u^\star_{j-1}
$. 
Here we again implicitly define $u^\star_0 = 0$, and $u^\star_M = 1$. 
Consequently, we have the following lemma.

\begin{lemma}\label{lemma:active-lessthan-max}
    Let $\bm{u}^\star$ be a minimizer. 
    Suppose that $u^\star_{j}=u^\star_{j-1}$ for some $1\le j\le M$, then for any $\alpha\in\knowledge$, we have the inequality  
    \[\eta_j(\bm{u}^\star,\alpha) - h_M(\alpha)<\psi_{M}(\bm{u}^\star).\]
\end{lemma}
\begin{proof}
    We consider the following two cases. 

    \paragraph{Case 1: $u^\star_{j-1} = 0$.} 
    In this case, for any $\alpha \in \knowledge$, we have $\eta_j(\bm{u}^\star, \alpha) - h_{M}(\alpha) = u^\star_{j} -\gamma(\alpha) u^\star_{j-1}- h_{M}(\alpha) = - h_{M}(\alpha) \leq 0$, while $\psi_{M}(\bm{u}^\star) = \optexp > 0$.  
    Hence the desired inequality holds.

    \paragraph{Case 2: $u^\star_{j-1} > 0$.} 
     Let $k \geq 0 $ be the largest index such that $\bm{u}^\star_{k}<\bm{u}^\star_{j-1}$. Such $k$ is guaranteed to exist because $\bm{u}^\star_0 = 0 < \bm{u}^\star_{j-1}$. In this case, we have 
     In other words, $\bm{u}^\star_{k}=\bm{u}^\star_{i-1}$ for $j+1\le k\le i-1$. One has $$
\eta_{j}(\bm{u}^\star,\alpha)  =(1-\gamma(\alpha))\bm{u}^\star_{j-1},$$
while 
$$
\eta_{k+1}(\bm{u}^\star,\alpha)=\bm{u}^\star_{k+1}-\gamma(\alpha)\bm{u}^\star_{k}=\bm{u}^\star_{j-1}-\gamma(\alpha)\bm{u}^\star_{k}>(1-\gamma(\alpha))\bm{u}^\star_{j-1}=\eta_{j}(\bm{u}^\star,\alpha).
$$
Here, the inequality is due to  $\bm{u}^\star_{k}<\bm{u}^\star_{i-1}$. Therefore, we have $\eta_{j}(\bm{u}^\star,\alpha) - h_M(\alpha)<\eta_{k+1}(\bm{u}^\star,\alpha) - h_M(\alpha)\le \psi_{M}(\bm{u}^\star)$.
\end{proof}

Combining Lemma~\ref{lemma:active-lessthan-max} and Lemma~\ref{lemma:no-missing-dir}, 
we see that $\mathbf{0}\notin\partial\psi_M(\bm{u}^\star)$, $\bm{v} \neq \bm{0}$. 
As a result, $\bm{n} \neq \bm{0}$.
Let $j$ be the smallest index in $I(\bm{u}^\star)$ such that $\lambda_{j} > 0$. 
By Lemma~\ref{lemma:active-lessthan-max} again, we know that 
$i_{k} \neq j$ for all $k$'s in the convex combination representation of $\bm{v}$.

First, suppose that $j \geq 2$. 
Consider the coordinate $v_{j-1} + n_{j-1}$. 
By the definition of the normal vector, we know that $n_{j-1} > 0$. 
Since $i_{k} \neq j$ for all $k$'s, we also know that $v_{j-1} \geq 0$. 
This contradicts with $\bm{v} + \bm{n} = \bm{0}$.

Second, suppose that $j=1$, i.e., $u^\star_1 = 0$, and $\lambda_1 > 0$.
If $\lambda_{2} = 0$,
then we must have $n_1 < 0$. 
Note that $i_k \neq 1$ for all $k$, and hence $v_1 \leq 0$. 
This contradicts with the first-order optimality condition. 
Consequently, we must have $\lambda_2 > 0$.
Now consider the second coordinate $v_2 + n_2$. 
If $\lambda_{3} = 0$, then $n_2 <0$. However $i_k \neq 2$ for all $k$, and hence $v_2 \leq 0$.
As a result, we can only have $\lambda_3 > 0$.
Continuing this argument, we must have $\lambda_{j} > 0$ for all $1\leq j \leq M$.
In other words, $u^\star_{j} = u^\star_{j-1}$ for all $1 \leq j \leq M$, which is impossible.

In all, we have proved via contradiction that $\bm{u}^\star$ must lie in the interior of the feasible set $\gridset$.

\subsection{Reduction to a system of equations}
The analysis carried out so far paves the way for studying the original problem in an alternative form, which proves much more convenient for later use.

Since $\bm{u}^\star$ is a minimizer lying in the interior of $\gridset$, we have 
$
\mathbf{0}\;\in\;\partial\psi_M(\bm{u}^{\star})$. By Lemma~\ref{lemma:no-missing-dir}, for each $1 \leq i \leq M$, there exists some $\alpha_i\in \knowledge$ such that 
 $\eta_i(\bm{u}^\star,\alpha_i)-h_M(\alpha_i)= \psi_{M}(\bm{u}^\star)$. For $ i \geq 2$, we have 
\[
\eta_i(\bm{u}^\star,\alpha_i)-h_M(\alpha_i) = u^\star_{i} - \gamma(\alpha) u^\star_{i-1} -h_M(\alpha_i) = \psi_{M}(\bm{u}^\star) > 0.
\]
It is clear that $\alpha_i < \infty$. Otherwise the equality would not hold. 

Denote by $\sset=[\beta/(2\beta+d),(\beta+d)/(2\beta+d)]=[\gmin,\gmax]$ the feasible range of $\gamma(\alpha)$. From now on, we redefine $h_M:\sset\rightarrow\mathbb{R}$ as 
\[h_M(\gamma)=\frac{1-\gamma}{1-\gamma^M}.\]
To avoid notation cluster, we write $\bm{u}$ for $\bm{u}^{\star}$. By the optimality condition,
\begin{align}\label{eq:reduction-to-system}
u_{1} & =\max_{\gamma\in\sset}u_{i}-\gamma u_{i-1}-h_{M}(\gamma),\quad2\le i\le M-1\nonumber\\
 & =\max_{\gamma\in\sset}1-\gamma u_{M-1}-h_{M}(\gamma).
\end{align}
Define the function $\phi_{M}(x)=\min_{\gamma\in\sset}\gamma x+h_{M}(\gamma)$
for $x\in(0,1)$. Rearranging the above equations, we have
\begin{equation}
u_{i}=u_{1}+\phi_{M}(u_{i-1}),\quad2\le i\le M,\label{eq:opt-cond-inf-ver}
\end{equation}
where we write $u_{M}=1$ for convenience. Denote by $\gamma_{i}=\argmin_{\gamma\in\sset}\gamma u_{i-1}+h_{M}(\gamma)$. Clearly, one has $\gamma_i=\beta(1+\alpha_i)/(2\beta+d)$.

The system of equations in~\eqref{eq:reduction-to-system} is an important consequence of the optimality condition. From now on, we will focus on this system rather than the original objective function. As we shall soon see, it allows us to establish several interesting properties about the minimizer $\bm{u}$ and the sequence $\{\gamma_i\}$.

\subsubsection{Monotonicity of $\{\gamma_i\}$}
First, we show the sequence $\{\gamma_i\}$ is non-increasing, which in turn translates to the monotonicity of $\alpha_i$. 
Let $\eta_v(\gamma)=v\cdot\gamma+h_M(\gamma)$. Since $h_M^\prime(\gamma)<0$ and $h_M^{\prime\prime}(\gamma)>0$, 
\[
\eta'(\gamma)=v+h_M'(\gamma)
\]
is strictly increasing in $\gamma$, so $\eta$ is strictly convex and has a unique minimizer $\gamma^*(v)$.

Define the thresholds
\[v_L\coloneqq -h_M^\prime(\gmax),\quad v_U\coloneqq -h_M^\prime(\gmin)
\quad\text{with } 0<v_L<v_U.\]
Then
\[\gamma^*(v)=
\begin{cases}
\gmax, & 0<v<c_L,\\
\textrm{the solution to }-h_M^\prime(\gamma)=v, & v_L\le v \le v_U,\\
\gmin, & v>v_U.
\end{cases}\]
Because $-h_M'(\gamma)$ is decreasing in $\gamma$, the solution of
$-h_M'(\gamma)=v$ becomes smaller when $v$ increases. Thus $\gamma^*(v)$ is nonincreasing in $v$: it is at $\gmax$ for small $v$, moves left continuously through the interior as $v$ grows, and sticks at $\gmin$ for large $v$. Since $\{u_i\}_{i=1}^{M-1}$ is an increasing sequence, the sequence $\{\gamma^*(u_i)\}_{i=1}^{M}$ is non-increasing.

\subsubsection{Behavior of the individual $u_{i}$}
The first lemma provides lower bound to $u_1$. For $a \in (0,1)$ and $n \geq 2$, define $S_{n}(a)\coloneqq  \sum_{k=0}^{n-1}a^{-k}$.

\begin{lemma}\label{lemma:u1-lb-strengthen}Fix any $c\in\sset$. The first component $u_{1}$ is lower bounded by
\[
u_{1}\ge\frac{c^{-1}\gmax{}^{-(M-3)}\bigl(1-h_{M}(c)\bigr)-h_{M}(\gmax)S_{M-2}(\gmax)}{S_{M-2}(\gmax)+\gmax+c^{-1}\gmax{}^{-(M-3)}}.
\]
\end{lemma}
\begin{proof}
In view of the optimality condition~\eqref{eq:reduction-to-system}, we have
\begin{align}
u_{1} & \ge u_{j+1}-\gmax u_{j}-h_{M}(\gmax),\quad1\le j\le M-2,\nonumber \\
u_{1} & \ge1-cu_{M-1}-h_{M}(c).\label{eq:final-inequality}
\end{align}
Multiplying the $j$-th inequality by $\gmax^{-(j-1)}$ and summing over
$1\le j\le M-2$, we obtain
\[
u_{1}\sum_{j=1}^{M-2}\gmax^{-(j-1)}\ge \gmax^{-(M-3)}u_{M-1}-\gmax u_{1}-h_{M}(\gmax)\sum_{j=1}^{M-2}\gmax^{-(j-1)}.
\]
Recall that $S_{n}(\gmax)=\sum_{k=0}^{n-1}\gmax^{-k}$. We have $\sum_{j=1}^{M-2}\gmax^{-(j-1)}=\sum_{j=0}^{M-3}\gmax^{-j}=S_{M-2}(\gmax).$
Use this to rewrite the inequality as
\begin{align*}
u_{1}(S_{M-2}(\gmax)+\gmax) & \ge \gmax^{-(M-3)}u_{M-1}-h_{M}(\gmax)S_{M-2}(\gmax)\\
 & \ge \gmax^{-(M-3)}c^{-1}(1-u_{1}-h_{M}(c))-h_{M}(\gmax)S_{M-2}(\gmax)\\
 & =-c^{-1}\gmax^{-(M-3)}u_{1}+c^{-1}\gmax^{-(M-3)}(1-h_{M}(c))-h_{M}(\gmax)S_{M-2}(\gmax),
\end{align*}
where the second step is due to relation~(\ref{eq:final-inequality}).
Combining terms we reach
\[
u_{1}(S_{M-2}(\gmax)+\gmax+c^{-1}\gmax^{-(M-3)})\ge c^{-1}\gmax^{-(M-3)}(1-h_{M}(c))-h_{M}(\gmax)S_{M-2}(\gmax).
\]
Rearranging terms yields the desired claim. 
\end{proof}

Lemma~\ref{lemma:u1-lb-strengthen} lower bounds the value of the first component $u_1$. Since $u_1=\optexp$ by the optimality condition, it provides a lower bound for the optimal objective value as well. 

Similarly, we have the upper bound on $u_{1}$ in the following lemma.

\begin{lemma}\label{lem:u1-upper-bound}We have 
\[
u_{1}\leq\frac{(M+1)\gamma_{\max}^{M-1}}{(1-\gamma_{\max})^{2}}.
\]

\end{lemma}

\begin{proof}The key identity to establish an upper bound on $u_{1}$
is 
\[
u_{1}+\phi_{M}(u_{M-1})=u_{1}+\inf_{\gamma\in\mathcal{S}}u_{M-1}\gamma+h_{M}(\gamma)=1.
\]
Now we split the proof into two cases: (1) when $\gamma_{M}=\gamma_{\max}$,
and (2) when $\gamma_{M}<\gamma_{\max}$.

\paragraph{Case 1: when $\gamma_{M}=\gamma_{\max}$.} In this case,
using the fact that $\gamma_{i}$ is monotonically decreasing, we
know that 
\[
\gamma_{2}=\gamma_{3}=\cdots=\gamma_{M}=\gamma_{\max}.
\]
This actually allows us to solve for $u_{1}$ exactly: 
\[
u_{1}=\frac{\gamma_{\max}^{M-1}(1-\gamma_{\max})^{2}}{(1-\gamma_{\max}^{M})^{2}}.
\]
\paragraph{Case 2: when $\gamma_{M} < \gamma_{\max}$.}In this case,
we have the relationship 
\[
u_{M-1}\geq-h_{M}'(\gamma_{M}).
\]
This together with the key identity yields 
\begin{align*}
u_{1} & =1-\inf_{\gamma\in\mathcal{S}}u_{M-1}\gamma+h_{M}(\gamma)\\
 & =1-u_{M-1}\gamma_{M}-h_{M}(\gamma_{M})\\
 & \leq1+h_{M}'(\gamma_{M})\gamma_{M}-h_{M}(\gamma_{M})\\
 & \leq1+h_{M}'(\gamma_{\max})\gamma_{\max}-h_{M}(\gamma_{\max}),
\end{align*}
where the last steps uses the fact that $h_{M}'(\gamma)\gamma-h_{M}(\gamma)$
is increasing in $\gamma$. Write this upper bound explicitly to see
that 
\[
u_{1}\leq\frac{\gamma_{\max}^{M}(M(1-\gamma_{\max})-1+\gamma_{\max}^{M})}{(1-\gamma_{\max}^{M})^{2}}.
\]

In both cases, the upper bound can be further relaxed to the one stated in the lemma. 
\end{proof}

The next two lemmas are about the gaps between consecutive $u_{i-1}$ and $u_i$.
\begin{lemma}\label{lemma:precondition} For any $M\ge3$, one has
\[
u_{i-1}^{}\le u_{i}^{}\cdot(\gamma_{i+1}^{}+\frac{1}{2}),\quad\forall2\le i\le M-1.
\]

\end{lemma}
\noindent See Appendix~\ref{sec:proof-precondition} for the proof of Lemma~\ref{lemma:precondition}.

Denote by $\Delta_{i}=u_{i+1}-u_{i}$.
\begin{lemma}\label{lemma:delta-ratio}One has $\gamma_{M}\le\Delta_{M-1}/\Delta_{M-2}\le\gamma_{M-1}$.
\end{lemma}

\begin{proof}
By concavity of $\phi_{M}(\cdot)$,
\[
\phi_{M}^{\prime}(u_{M-1})\le\frac{\phi_{M}(u_{M-1})-\phi_{M}(u_{M-2})}{u_{M-1}-u_{M-2}}\le\phi_{M}^{\prime}(u_{M-2}).
\]
By relation~(\ref{eq:opt-cond-inf-ver}), we have $\phi_{M}(u_{M-1})-\phi_{M}(u_{M-2})=u_{M}-u_{M-1}=\Delta_{M-1}$.
Meanwhile, $\phi_{M}^{\prime}(u_{M-1})=\gamma_{M}$ and $\phi_{M}^{\prime}(u_{M-2})=\gamma_{M-1}$.
Hence, $\gamma_{M}\le\Delta_{M-1}/\Delta_{M-2}\le\gamma_{M-1}$.
\end{proof}


\subsubsection{Monotonicity of $\optexp$}\label{sec:monot-proof}
We will show $\optexp$ is strictly decreasing in $M$ when $M$ is small, which proves Proposition~\ref{prop:monot}.

Let $\bm{u}^{(M)}=\argmin_{\bm{v}\in\gridset} \sup_{\alpha \in \knowledge} \Psi_M(\bm v,\alpha)$ and $\bm{u}^{(M+1)}=\argmin_{\bm{v}\in\mathcal{U}_{M+1}} \sup_{\alpha \in \knowledge} \Psi_{M+1}(\bm v,\alpha)$. By the optimality condition, we have $\optexp=u_1^{(M)}$ and $\psi_{M+1}^\star=u_1^{(M+1)}$. It is equivalent to show $u_1^{(M+1)} < u_1^{(M)}$.

We first give a high-level description of the proof idea.  For the sake of contradiction, suppose that the optimal solution for $M+1$ has $u_1^{(M+1)} \ge u_1^{(M)}$. We can then iteratively solve the recursion in~\eqref{eq:opt-cond-inf-ver} and obtain $u_2^{(M+1)}, \ldots, u_{M}^{(M+1)}$. Then we show that it must contradict the equation 
\[
1 = u_1^{(M+1)} + \phi_{M+1}(u_{M}^{(M+1)}).
\]
By the definition of $\phi_{M+1}$, it suffices to prove that for all $\gamma \in S$, we have 
\begin{align}
   \label{eq:key-contra}
u_1^{(M+1)}  + \gamma u_{M}^{(M+1)} + h_{M+1} (\gamma) - 1 > 0. 
\end{align}
We therefore need a lower bound on $u_{M}^{(M+1)}$ and also a lower bound on $u_{1}^{(M+1)}$. To lower bound $u_1^{(M)}$, we can apply Lemma~\ref{lemma:u1-lb-strengthen} with $c=\gmax$,
\[
u_1^{(M)} \geq \frac{\gmax^{M-1}}{S_{M}^2(\gmax)}.
\]

Next, we turn to lower bounding $u_{M}^{(M+1)}$. By the recursive relation in~\eqref{eq:opt-cond-inf-ver}, we know that 
\[
u_{k+1}^{(M+1)}\;=\;u_1^{(M+1)}\;+\;\phi_{M+1}\!\big(u_k^{(M+1)}\big).
\]
Denote by $\Delta_M(\gamma)=h_M(\gamma)-h_{M+1}(\gamma)$. For $k=1$, one has 
\begin{align*}
u_{2}^{(M+1)}&\;=\;u_1^{(M+1)}\;+\;\phi_{M+1}\!\big(u_1^{(M+1)}\big) \\
& \; = \;u_1^{(M+1)}\;+\;\inf_{\gamma}\!\big( \gamma u_1^{(M+1)} + h_{M+1}(\gamma)\big) \\
&  \ge \;u_1^{(M)}\;+\;\inf_{\gamma}\!\big( \gamma u_1^{(M)} + h_{M}(\gamma) + h_{M+1}(\gamma) - h_{M}(\gamma) \big) \\
&\geq u_1^{(M)} + \inf_{\gamma} ( \gamma u_1^{(M)} + h_{M}(\gamma) ) - \Delta(\gmax) \\
& = u_2^{(M)} - \Delta(\gmax),
\end{align*}
where the first inequality is due to the assumption $u_1^{(M+1)} \ge u_1^{(M)}$, and the second inequality is because the function $\Delta_M(\cdot)$ is increasing. Similarly when $k=2$, we have 
\begin{align*}
u_{3}^{(M+1)}&\;=\;u_1^{(M+1)}\;+\;\phi_{M+1}\!\big(u_2^{(M+1)}\big) \\
& \; = \;u_1^{(M+1)}\;+\;\inf_{\gamma}\!\big( \gamma u_2^{(M+1)} + h_{M+1}(\gamma)\big) \\
&  = \;u_1^{(M+1)}\;+\;\inf_{\gamma}\!\big( \gamma u_2^{(M)} - \gamma u_2^{(M)} + \gamma u_2^{(M+1)} + h_{M}(\gamma) + h_{M+1}(\gamma) - h_{M}(\gamma) \big) \\
&\geq u_1^{(M)}+ \inf_{\gamma} ( \gamma u_2^{(M)} + h_{M}(\gamma) ) - \Delta(\gmax) - \gmax (u_2^{(M)} - u_2^{(M+1)})\\
& = u_3^{(M)} - (1+\gmax)\Delta(\gmax).
\end{align*}
Recursively, we obtain 
\begin{align*}
    u_M^{(M+1)} \geq u_{M}^{(M)}- S_{M-2}(\gmax)  \Delta(\gmax) = 1 - S_{M-2}(\gmax)  \Delta(\gmax). 
\end{align*}
Substituting back into the key equation we aim to prove to see that
\begin{align*}
   u_1^{(M+1)}  + \gamma u_{M}^{(M+1)} + h_{M+1} (\gamma) - 1 
   &\geq u_1^{(M)}  + \gamma (1 - S_{M-2}(\gmax)  \Delta(\gmax)) + h_{M+1} (\gamma) - 1 \\
& \geq \frac{\gmax^{M-1}}{S_{M}^2(\gmax)} + \gamma (1 - S_{M-2}(\gmax)  \Delta(\gmax)) + h_{M+1} (\gamma) - 1. 
\end{align*}
For $M=2,3,4$, the RHS of the above relation simplifies to

\begin{align*}
M=2:\quad&
\frac{\gmax}{(1+\gmax)^2}+\gamma+\frac{1}{1+\gamma+\gamma^2}-1
= \frac{\gmax}{(1+\gmax)^2}+\frac{\gamma^3}{1+\gamma+\gamma^2}>0,\\[2pt]
M=3:\quad&
\frac{\gmax^2}{(1+\gmax+\gmax^2)^2}
+\gamma\Big(1-\frac{\gmax^3}{S_3(\gmax)S_4(\gmax)}\Big)
+\frac{1}{S_4(\gamma)}-1>0,\\[2pt]
M=4:\quad&
\frac{\gmax^3}{S_4(\gmax)^2}
+\gamma\Big(1-\frac{S_2(\gmax)\,\gmax^4}{S_4(\gmax)S_5(\gmax)}\Big)
+\frac{1}{S_5(\gamma)}-1>0,
\end{align*}
for all $\gamma\in \sset$. Consequently, relation~\eqref{eq:key-contra} holds and we reach $u_1^{(M+1)} + \phi_{M+1}(u_{M}^{(M+1)})>1$, which is a contradiction. Therefore, one has $u_1^{(M+1)} < u_1^{(M)}$ when $M=2,3,4$.

\subsection{Proof of Lemma~\ref{lemma:precondition}}\label{sec:proof-precondition}

On a high level, the proof first establishes $
u_{i-1}\le u_{i}\cdot(\gamma_{i+1}+\frac{1}{2})$ for any $2\le i\le M-2$ by showing $\gamma_{i+1}\ge1/2$. To do so, it suffices to prove $\gamma_{M-1}\ge1/2$ and use the fact that $\{\gamma_i\}$ is non-increasing. Next, we show that the last inequality $
u_{M-2}\le u_{M-1}\cdot(\gamma_{M}+\frac{1}{2})$ holds as well. This is achieved through analyzing different ranges of $M$ and establishing tight lower bound on $\gamma_M$ under that range.

\subsubsection{Key lemmas}
In this section, we collect several useful bounds on $u_1$ and $u_{M-1}$. 
When $\gamma_{M}$ is small, we have the following lower bound. 
\begin{lemma}\label{lemma:uM-greater-one}Fix any $c\in(\gmin,1/2]$.
If $\gamma_{M}<c$, then
\[
u_1 + \phi_{M}(u_{M-1})>\frac{c^{-1}\gmax{}^{-(M-3)}\bigl(1-h_{M}(c)\bigr)-h_{M}(\gmax)S_{M-2}(\gmax)}{S_{M-2}(\gmax)+\gmax+c^{-1}\gmax{}^{-(M-3)},}-ch_{M}^{\prime}(c)+h_{M}(c).
\]

\end{lemma}

When $\gamma_{M-1}$ is small, we have the following lower bound. 
\begin{lemma}\label{lemma:g-shift-bound}If $\gamma_{M-1}<1/2$,
then
\[
u_1 + \phi_{M}(u_{M-1})>(1+\gmin)u_{1}+(h_{M}(\frac{1}{2})-\frac{1}{2}h_{M}^{\prime}(\frac{1}{2}))\gmin+h_{M}(\gmin).
\]
\end{lemma}


\begin{figure}         
\centering         
\includegraphics[scale=0.5]{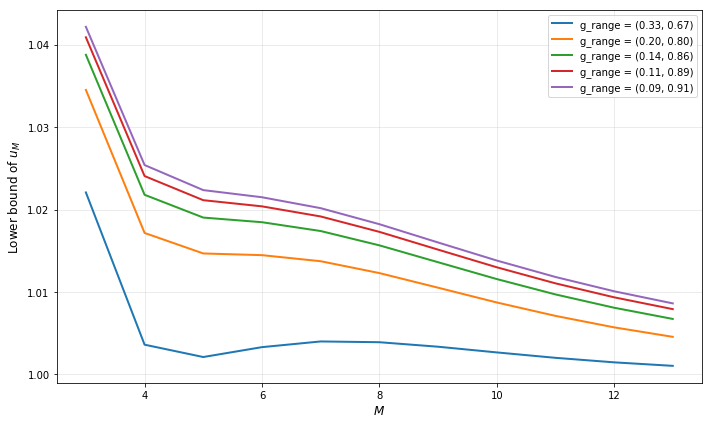}
\caption{Lower bound of $u_M$ vs.~batch budget $M$ when $\gamma_M<0.45$.}
\label{fig:last-g}     
\end{figure}


\subsubsection{Main proof}
Now we are ready to prove Lemma~\ref{lemma:precondition}. 
We start with proving that $
u_{i-1}\le u_{i}\cdot(\gamma_{i+1}+\frac{1}{2})
$
holds for all $2\le i\le M-2$.
Then we finish with proving $u_{M-2}<u_{M-1}(\gamma_{M}+\frac{1}{2})$.

\paragraph{Step 1: Establishing $
u_{i-1}\le u_{i}\cdot(\gamma_{i+1}+\frac{1}{2})\; \forall 2\le i\le M-2$.} 
A key step is to prove that $\gamma_{M}\ge0.45$ and $\gamma_{M-1}\ge0.5$.
We start by lower bounding $\gamma_{M}$ by $0.45$. 
Suppose that $\gamma_{M}<0.45$. Applying Lemma~\ref{lemma:uM-greater-one}
with $c=0.45$, we have $u_1 + \phi_{M}(u_{M-1})>1$, which contradicts with $u_{M} = -u_1 + \phi_{M}(u_{M-1})=1$; see Figure~\ref{fig:last-g}. 
As a result, we have $\gamma_{M} \geq 0.45$  

Next, we lower bound $\gamma_{M-1}$ by $0.5$. Suppose that $\gamma_{M-1}<1/2$. By Lemma~\ref{lemma:g-shift-bound},
we have
\[
u_1 + \phi_{M}(u_{M-1})>(1+\gmin)u_{1}+(h_{M}(\frac{1}{2})-\frac{1}{2}h_{M}^{\prime}(\frac{1}{2}))\gmin+h_{M}(\gmin).
\]
In addition, apply Lemma~\ref{lemma:u1-lb-strengthen} with $c=0.45$ to obtain
\[
u_{1}\ge\frac{\frac{1}{0.45}\gmax{}^{-(M-3)}\bigl(1-h_{M}(0.45)\bigr)-h_{M}(\gmax)S_{M-2}(\gmax)}{S_{M-2}(\gmax)+\gmax+\frac{1}{0.45}\gmax{}^{-(M-3)}}.
\]
Taking the above two displays together, we reach
\begin{align*}
 u_1 + \phi_{M}(u_{M-1})&>(1+\gmin)\left \{ \frac{\frac{1}{0.45}\gmax{}^{-(M-3)}\bigl(1-h_{M}(0.45)\bigr)-h_{M}(\gmax)S_{M-2}(\gmax)}{S_{M-2}(\gmax)+\gmax+\frac{1}{0.45}\gmax{}^{-(M-3)}}\right \} \\
&\quad +(h_{M}(\frac{1}{2})-\frac{1}{2}h_{M}^{\prime}(\frac{1}{2}))\gmin+h_{M}(\gmin).   
\end{align*}
The RHS exceeds 1 and one has $u_1 + \phi_{M}(u_{M-1})>1$, which contradicts with $u_M=u_1 + \phi_{M}(u_{M-1}) =1$; see Figure~\ref{fig:second-last-g}. 
As a result, we must have $\gamma_{M-1}\ge0.5$. 

By the second item of Proposition~\ref{proposition:opt-prop}, we know that $\gamma_i$ is decreasing and hence 
$\gamma_{i}\ge\gamma_{M-1}\ge0.5$ for $i\le M-2$. Since
$\bm{u}$ lies in the interior of $\gridset$, we have $u_{i-1}<u_{i}$, and
this leads to 
\[
u_{i-1}\le u_{i}\cdot(\gamma_{i+1}+\frac{1}{2}),\quad\forall2\le i\le M-2.
\]
In other words, we have established the desired inequalities except
for the last one. The remaining of the proof is devoted to show $u_{M-2}<u_{M-1}(\gamma_{M}+\frac{1}{2}).$

\paragraph{Step 2: Establishing $u_{M-2}<u_{M-1}(\gamma_{M}+\frac{1}{2})$.} 
We split the proof into two cases: (1) $M\le6$, and $M \geq 7$.

A key lemma is the following. 
\begin{lemma}\label{lemma:direct-last-ineq}Assume $b<\gamma_{M}<1/2$
for some $b\in(\gmin,1/2)$. If
\begin{equation}
\gamma_{M-1}<\frac{(1+h_{M}^{\prime}(b))(1+\gamma_{M}(\frac{1}{2}-b))}{\frac{1}{2}-b},\label{eq:g-conditon-for-lem}
\end{equation}
then 
$
u_{M-2}<u_{M-1}(\gamma_{M}+\frac{1}{2}).
$
\end{lemma}

When $M\le6$, setting $b=0.45$, we have
\[
\frac{(1+h_{M}^{\prime}(b))(1+\gamma_{M}(\frac{1}{2}-b))}{\frac{1}{2}-b}>1. 
\]
Since $\gamma_{M-1}<1$, the condition~\eqref{eq:g-conditon-for-lem} holds. We can then apply
Lemma~\ref{lemma:direct-last-ineq} to reach the conclusion that $u_{M-2}<u_{M-1}(\gamma_{M}+\frac{1}{2})$.

In the case when $M\ge7$, we further consider two subcases: (1) $d\ge2$, and (2) $d = 1$. 

When $d \geq 2$, we prove that $\gamma_{M} \geq  1/2$. To see this, suppose that $\gamma_{M} < 1/2$, we apply Lemma~\ref{lemma:uM-greater-one} with $c=1/2$ to obtain $u_1 + \phi_{M}(u_{M-1}) > 1$, which is a contradiction with $u_M =u_1 + \phi_{M}(u_{M-1})=1$. 
Hence we have $\gamma_{M} \geq 1/2$, that further implies $u_{M-2}<u_{M-1}(\gamma_{M}+\frac{1}{2})$.

For $d=1$, there are again two cases. When $M\ge10$, we can again prove 
$\gamma_{M}\ge1/2$ by Lemma~\ref{lemma:uM-greater-one}. 
When $M=7,8,9$,
Lemma~\ref{lemma:uM-greater-one} gives $\gamma_{M}\ge0.47,0.48,0.49,$
respectively. Plugging in the corresponding lower bound value of $\gamma_{M}$
as $b$ to Lemma~\ref{lemma:direct-last-ineq}, one can verify that
the precondition holds and therefore $u_{M-2}<u_{M-1}(\gamma_{M}+\frac{1}{2}).$

\begin{figure}         
\centering         
\includegraphics[scale=0.5]{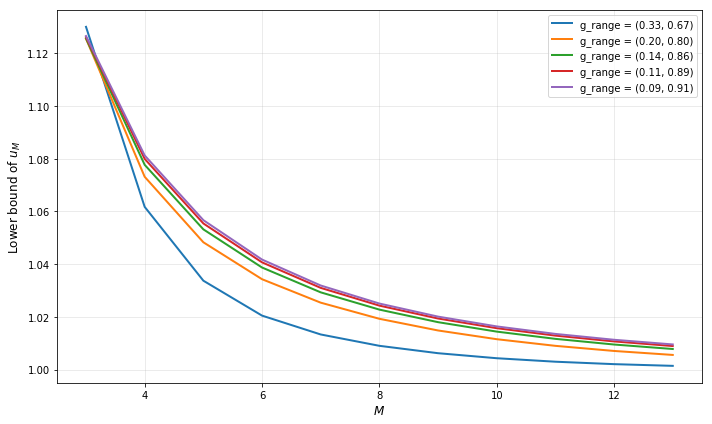}
\caption{Lower bound of $u_M$ vs.~batch budget $M$ when $\gamma_{M-1}<0.5$.}
\label{fig:second-last-g}   
\end{figure}

\subsubsection{Remaining proofs}\label{sec:system-rem-proofs}

\begin{proof}[Proof of Lemma~\ref{lemma:uM-greater-one}]
Recall that $\gamma_{M}=\argmin_{\gamma\in\sset}\gamma u_{M-1}+h_{M}(\gamma)$. Since $\gamma_{M}<c \leq 1/2$, one has either $u_{M-1}+h_{M}^{\prime}(\gamma_{M})=0$ or $\gamma_{M}=\gmin$.
In both cases, convexity implies $u_{M-1}\ge-h_{M}^{\prime}(\gamma_{M})$.
In addition, we know that the function $h_{M}^{\prime}(\cdot)$ is strictly increasing
and $\gamma_{M}<c$, we have $u_{M-1}\ge-h_{M}^{\prime}(\gamma_{M})>-h_{M}^{\prime}(c)$.
Consequently,
\begin{align*}
u_{1}+\min_{\gamma\in\sset}\gamma u_{M-1}+h_{M}(\gamma) & >u_{1}+\min_{\gamma\in\sset}\gamma(-h_{M}^{\prime}(c))+h_{M}(\gamma)\\
 & =u_{1}-ch_{M}^{\prime}(c)+h_{M}(c)\\
 & \ge\frac{c^{-1}\gmax{}^{-(M-3)}\bigl(1-h_{M}(c)\bigr)-h_{M}(\gmax)S_{M-2}(\gmax)}{S_{M-2}(\gmax)+\gmax+c^{-1}\gmax{}^{-(M-3)},}-ch_{M}^{\prime}(c)+h_{M}(c),
\end{align*}
where the second step is due to $\argmin_{\gamma\in\sset}-\gamma h_{M}^{\prime}(c)+h_{M}(c)=c$,
and the last step applies Lemma~\ref{lemma:u1-lb-strengthen}.
\end{proof}

\begin{proof}[Proof of Lemma~\ref{lemma:g-shift-bound}]
Recall that $\gamma_{M-1}=\argmin_{\gamma\in\sset}\gamma u_{M-2}+h_{M}(\gamma)$. 
Since by assumption $\gamma_{M-1} < 1/2$, 
one has either $u_{M-2}+h_{M}^{\prime}(\gamma_{M-1})=0$ or $\gamma_{M}=\gmin$.
In both cases, convexity implies $u_{M-2}\ge-h_{M}^{\prime}(\gamma_{M-1}).$
Since the function $h_{M}^{\prime}(\cdot)$ is strictly increasing
and $\gamma_{M-1}<1/2$, we have $u_{M-1}\ge-h_{M}^{\prime}(\gamma_{M-1})>-h_{M}^{\prime}(1/2)$.
Consequently,
\begin{align*}
u_{1}+\phi_{M}(u_{M-1}) & =u_{1}+\phi_{M}(u_{1}+\phi_{M}(u_{M-2}))\\
 & \ge u_{1}+\phi_{M}(u_{1}+\min_{\gamma\in\sset}-h_{M}^{\prime}(1/2)\gamma+h_{M}(\gamma))\\
 & \overset{}{=}u_{1}+\phi_{M}(u_{1}-\frac{1}{2}h_{M}^{\prime}(\frac{1}{2})+h_{M}(\frac{1}{2}))\\
 & =u_{1}+\min_{\gamma\in\sset}(u_{1}-\frac{1}{2}h_{M}^{\prime}(\frac{1}{2})+h_{M}(\frac{1}{2}))\gamma+h_{M}(\gamma),
\end{align*}
where the penultimate step is due to $\argmin_{\gamma\in\sset}-\gamma h_{M}^{\prime}(1/2)+h_{M}(1/2)=1/2$.
By Lemma~\ref{lemma:u1-lb-strengthen} with $c=0.45$, one has $u_{1}-\frac{1}{2}h_{M}^{\prime}(\frac{1}{2})+h_{M}(\frac{1}{2})>-h_{M}^{\prime}(\gmin)$.
Hence, the minimizer of the above function is attained at $\gamma=\gmin$.
We reach
\begin{align*}
u_{1}+\phi_{M}(u_{M-1}) & \ge u_{1}+\min_{\gamma\in\sset}(u_{1}-\frac{1}{2}h_{M}^{\prime}(\frac{1}{2})+h_{M}(\frac{1}{2}))\gamma+h_{M}(\gamma)\\
 & \overset{}{=}u_{1}+(u_{1}-\frac{1}{2}h_{M}^{\prime}(\frac{1}{2})+h_{M}(\frac{1}{2}))\gmin+h_{M}(\gmin)\\
 & =(1+\gmin)u_{1}+(-\frac{1}{2}h_{M}^{\prime}(\frac{1}{2})+h_{M}(\frac{1}{2}))\gmin+h_{M}(\gmin).
\end{align*}
This completes the proof. 
\end{proof}

\begin{proof}[Proof of Lemma~\ref{lemma:direct-last-ineq}]
Denote by $\Delta_{i}=u_{i+1}-u_{i}$. It suffices to show $u_{M-1}(1/2-\gamma_{M})<\Delta_{M-2}$.
Since
\begin{align*}
u_{M-1}(1/2-\gamma_{M}) & =(1-\Delta_{M-1})(1/2-\gamma_{M}) \le(1-\gamma_{M}\Delta_{M-2})(1/2-\gamma_{M}),
\end{align*}
where the inequality is due to Lemma~\ref{lemma:delta-ratio} and the assumption that $\gamma_{M} < 1/2$. 
It boils down to establishing $(1-\gamma_{M}\Delta_{M-2})(1/2-\gamma_{M})<\Delta_{M-2}$, which is equivalent to showing that 
\[
\Delta_{M-2}>\frac{\frac{1}{2}-\gamma_{M}}{1+\gamma_{M}(\frac{1}{2}-\gamma_{M})}.
\]
Note that
\[
\Delta_{M-2}=\frac{\Delta_{M-2}}{\Delta_{M-1}}\cdot\Delta_{M-1}\ge\frac{1}{\gamma_{M-1}}\cdot\Delta_{M-1}=\frac{1}{\gamma_{M-1}}\cdot(1-u_{M-1}),
\]
where the first inequality again uses Lemma~\ref{lemma:delta-ratio}.
Further note that since $\gamma_{M} \in (\gmin,1/2)$ minimizes $u_{M-1} g + h_M(g)$, 
we obtain $u_{M-1} = - h_M ' (\gamma_{M})$. 
This allows us to lower bound $\Delta_{M-2}$ as 
\[
\Delta_{M-2} \geq \frac{1}{\gamma_{M-1}}\cdot(1+h_M ' (\gamma_{M})).
\]
In all, it suffices to show that 
\begin{equation}\label{eq:final-goal}
    \gamma_{M-1}<\frac{(1+h_{M}^{\prime}(b))(1+\gamma_{M}(\frac{1}{2}-b))}{\frac{1}{2}-b}
\end{equation}
Under the assumption~(\ref{eq:g-conditon-for-lem}), we have 
\[
\gamma_{M-1}<\frac{(1+h_{M}^{\prime}(b))(1+\gamma_{M}(\frac{1}{2}-b))}{\frac{1}{2}-b}.
\]
Note that the RHS as a function of $b$ is increasing when $b \in (\gmin, 1/2)$, we then have 
\[
\gamma_{M-1}<\frac{(1+h_{M}^{\prime}(b))(1+\gamma_{M}(\frac{1}{2}-b))}{\frac{1}{2}-b}<\frac{(1+h_{M}^{\prime}(\gamma_{M}))(1+\gamma_{M}(\frac{1}{2}-\gamma_{M}))}{\frac{1}{2}-\gamma_{M}}.
\]
This is equivalent to our final goal~\eqref{eq:final-goal}. Hence the proof is completed. 
\end{proof}

\section{Remaining proofs for the lower bound}
\subsection{Remaining proofs of Section~\ref{sec:lower-family-design}}\label{sec:lower-remain-pf}

\begin{proof}[Proof of Proposition~\ref{prop:smooth-margin}]
  It is straightforward to check $f^{(1)}_{S,\sigma,i}$ satisfies the smoothness condition.

We now verify the margin condition. 
If $\delta < \delta_i$, by design we have
\[
P_{X}\left(0<\left|f^{(1)}_{S,\sigma,i}(X)-\frac{1}{2}\right|\leq\delta\right)=0.
\]
Otherwise, choose $\ell\le i$ s.t.\ $\delta_\ell\le \delta<\delta_{\ell-1}$. Then $\{0<|f^{(1)}-\tfrac12|\le \delta\}\subset\bigcup_{m\ge \ell}\bigcup_{j\in S_m} B_{m,j}$.
As a result, we have
\begin{align*}
 \mathbb{P}(0<|f^{(1)}-\tfrac12|\le \delta)\ \leq  \sum_{m\ge \ell}^{i} \sum_{j \in S_m} \mathbb{P}(B_{m,j}) = \sum_{m\ge \ell}^{i} s_m (M z_m)^{-d},   
\end{align*}
where the last relation arises from the covariate distribution~\eqref{eq:covariate-density}.

Recall that $s_m = M^{-1}(Mz_m)^{d - \alpha_m \beta}$. 
We further have
\begin{align*}
 \mathbb{P}(0<|f^{(1)}-\tfrac12|\le \delta)\ 
 \leq   \sum_{m\ge \ell}^{i} s_m (M z_m)^{-d} 
 = \sum_{m\ge \ell}^{i} M^{-1}(M z_m)^{-\alpha_m \beta}.
\end{align*}
Since $(Mz_{m})^{-\alpha_{m}\beta}\le (Mz_{l})^{-\alpha_{m}\beta}$ for $m \geq l$ and $\alpha_{m} \geq \alpha_{i}$ for for $m \leq i$, we arrive at 
\begin{align*}
 \mathbb{P}(0<|f^{(1)}-\tfrac12|\le \delta)\ 
 \leq   (M z_l)^{-\alpha_i \beta} \leq (\frac{\delta}{D_{\phi}})^{\alpha_i}.
\end{align*}
We therefore established the smoothness and the margin condition. 
\end{proof}



\subsection{Remaining proofs of Section~\ref{sec:lower-lb-Qi}}\label{sec:lower-lb-Qi-proof}
\begin{proof}[Proof of Lemma~\ref{lemma:single-bin-tv}]
It suffices to bound their KL-divergence. We can compute
\begin{align*}
\mathrm{KL}(\pminus^{n},\pplus^{n}) & \overset{\mathrm{(k)}}{\le}8\mathbb{E}_{\Gamma,\pi;\sigma_{i,j}=-1}[\sum_{t=1}^{n}(f_{\sigma_{i,j}=-1}(X_{t})-f_{\sigma_{i,j}=1}(X_{t}))^{2}\mathbf{1}\{\pi_{t}(X_{t})=1\}]\\
 & \overset{\mathrm{(ii)}}{\le}32D_{\phi}^{2}(Mz_i)^{-2\beta}\mathbb{E}_{\Gamma,\pi;\sigma_{i,j}=-1}[\sum_{t=1}^{n}\mathbf{1}\{\pi_{t}(X_{t})=1,X_{t}\in \peaki\}]\\
 & \overset{\mathrm{(iii)}}{=}32D_{\phi}^{2}(Mz_i)^{-(2\beta+d)}\sum_{t=1}^{n}\mathbb{P}_{\Gamma,\pi;\sigma_{i,j}=-1}^{t}(\pi_{t}(X_{t})=1\mid X_{t}\in \peaki)\\
 & \overset{\mathrm{(iv)}}{\le}32D_{\phi}^{2}(Mz_i)^{-(2\beta+d)}n\le2n(Mz_i)^{-(2\beta+d)}.
\end{align*}
Here, step (k) uses the standard decomposition of KL divergence and
Bernoulli reward structure; step (ii) is due to the definition of
$f_{\omega}$; step (iii) uses $\mathbb{P}(X_{t}\in \peaki)=1/(Mz_i)^{d}$,
and step (iv) arises from $\mathbb{P}_{\Gamma,\pi;\sigma_{i,j}=-1}^{t}(\pi_{t}(X_{t})=1\mid X_{t}\in \peaki)\le1$
for any $1\le t\le n$. 

By Pinsker's inequality, 
\[\|\pminus^{n}-\pplus^{n}\|_\mathrm{TV}\le\sqrt{\frac{1}{2}\mathrm{KL}(\pminus^{n},\pplus^{n})}\le\sqrt{n(Mz_i)^{-(2\beta+d)}}.\]
This finishes the proof. 
\end{proof}

\subsection{Remaining proofs of Section~\ref{sec:mix-indis-proof}}\label{sec:chi-sq-remain}
\begin{proof}[Proof of Lemma~\ref{lemma:moment-of-m}]
    Throughout the proof, we drop the subscript on $i$. Recall
\[
r(\sigma)\coloneqq \frac{(\tfrac{1}{2}+\sigma\binht)^R(\tfrac{1}{2}-\sigma\binht)^{N-R}}{(\tfrac{1}{2})^N}
=(1+2\sigma\binht)^R(1-2\sigma\binht)^{N-R},\quad \sigma\in\{+1,-1\},
\]
and
\[
m\coloneqq \tfrac{1}{2}\big(r(+1)+r(-1)\big).
\]

We want to compute $\mathbb{E}_0[m\mid N]$ and $\mathbb{E}_0[m^2\mid N]$. To start with,
\begin{align*}
\mathbb{E}_0[r(\sigma)\mid N]
&=\sum_{x=0}^N \binom{N}{x}2^{-N}\,(1+2\sigma\binht)^x(1-2\sigma\binht)^{N-x}\\
&=2^{-N}\sum_{x=0}^N \binom{N}{x}\big(1+2\sigma\binht\big)^x\big(1-2\sigma\binht\big)^{N-x}\\
&=2^{-N}\,\big[(1+2\sigma\binht)+(1-2\sigma\binht)\big]^N \quad\text{(binomial theorem)}\\
&=2^{-N}\,(2)^N = 1.
\end{align*}
Therefore,
\[
\mathbb{E}_0[m\mid N]=\tfrac{1}{2}\big(\mathbb{E}_0[r(+1)\mid N]+\mathbb{E}_0[r(-1)\mid N]\big)=1.
\]

Next, we deal with the second moment. Expanding $m^2$,
\[
m^2=\tfrac{1}{4}\Big(r(+1)^2+2\,r(+1)r(-1)+r(-1)^2\Big).
\]
We will evaluate the three expectations separately.
\begin{align*}
\mathbb{E}_0[r(+1)^2\mid N]
&=\sum_{x=0}^N \binom{N}{x}2^{-N}(1+2\binht)^{2x}(1-2\binht)^{2(N-x)}\\
&=2^{-N}\big[(1+2\binht)^2+(1-2\binht)^2\big]^N\\
&=2^{-N}\big(2+8\binht^2\big)^N=(1+4\binht^2)^N.
\end{align*}
By symmetry, $\mathbb{E}_0[r(-1)^2\mid N]=(1+4\binht^2)^N$. For the cross-product,
\begin{align*}
r(+1)r(-1)
&=\big[(1+2\binht)^R(1-2\binht)^{N-R}\big]\big[(1-2\binht)^R(1+2\binht)^{N-R}\big]\\
&=\big((1+2\binht)(1-2\binht)\big)^R \big((1-2\binht)(1+2\binht)\big)^{N-R}\\
&=(1-4\binht^2)^R(1-4\binht^2)^{N-R}=(1-4\binht^2)^N,
\end{align*}
which is constant in $R$, hence
\[
\mathbb{E}_0[r(+1)r(-1)\mid N]=(1-4\binht^2)^N.
\]

Putting things together,
\begin{align*}
\mathbb{E}_0[m^2\mid N]
&=\tfrac{1}{4}\Big(\mathbb{E}_0[r(+1)^2\mid N]+2\,\mathbb{E}_0[r(+1)r(-1)\mid N]+\mathbb{E}_0[r(-1)^2\mid N]\Big)\\
&=\tfrac{1}{4}\Big((1+4\binht^2)^N+2(1-4\binht^2)^N+(1+4\binht^2)^N\Big)\\
&=\tfrac{1}{2}\Big((1+4\binht^2)^N+(1-4\binht^2)^N\Big).
\end{align*}
This finishes the proof.
\end{proof}

\begin{proof}[Proof of Lemma~\ref{lemma:inter-moment}]
    By the law of total expectation,
    \begin{align*}
        \mathbb{E}[t^L]
        &=\mathbb{E}[\mathbb{E}[t^L\mid S]]\\
        &=\mathbb{E}[\mathbb{E}[t^{\sum_{i\in S}\mathbf{1}\{i\in S^\prime\}}\mid S]]\\
        &\le\mathbb{E}[\mathbb{E}[t^{\mathrm{Binomial}(\supportsize,\frac{\supportsize}{\nbin})}\mid S]],
    \end{align*}
    where the last step applies Lemma 1.1 in~\cite{bardenet2015concentration}. Using the PGF of binomial distribution, we obtain
    \begin{align*}
        \mathbb{E}[t^L]
        &\le\mathbb{E}[\mathbb{E}[t^{\mathrm{Binomial}(\supportsize,\frac{\supportsize}{\nbin})}\mid S]]\\
        &=\left(1+\frac{\supportsize}{\nbin}(t-1)\right)^\supportsize
        \le\exp\left(\frac{\supportsize^2}{\nbin}(t-1)\right).
    \end{align*}
    This completes the proof.
\end{proof}

\section{Proof of the upper bound}\label{subsec:proof-lem-basedb-regret}


This section is devoted to establishing Lemma~\ref{lem:basedb-regret}, whose proof follows the framework developed in \cite{perchet2013multi,jiang2025batched}. 

To start with, recall $\mathcal{T}$ is a tree of depth $M$, whose root (depth 0) represents the whole covariate space $\mathcal{X}$.  The tree is recursively defined as the following: for any $i\ge 1$, each node at depth $i-1$ is split into $g_{i-1}^d$ children. Consequently, a node at depth $i$ has width $w_i=g_{i-1}^{-1}\cdot  w_{i-1}=(\prod_{l=0}^{i-1}g_{l})^{-1}$. For any bin $C\in\mathcal{T}$, denote its parent by $\mathsf{p}(C)=\{C^{\prime}\in\mathcal{T}:C\in\mathsf{child}(C^{\prime})\}$. Define $\mathsf{p}^{1}(C)=\mathsf{p}(C)$ and $\mathsf{p}^{k}(C)=\mathsf{p}(\mathsf{p}^{k-1}(C))$
for $k\geq2$. Let $\mathcal{P}(C)=\{C^{\prime}\in\mathcal{T}:C^{\prime}=\mathsf{p}^{k}(C)\textrm{ for some }k\ge1\}$ be the set of ancestors of the bin $C$. Denote by $\mathcal{L}_{0}=\{\mathcal{X}\}$, and let $\mathcal{L}_{t}$ be the set of active bins at
time $t$. It is easy to see $\mathcal{L}_{t}=\mathcal{B}_{1}$ for $1\leq t\leq t_{1}$, where $\mathcal{B}_{1}$ are all the bins in the first layer. 

\subsection{Introducing the good events}

Fix a batch $i\ge1$, for any $C\in\mathcal{L}_{t_{i-1}+1}$, define 
\[
m_{C,i}\coloneqq \sum_{t=t_{i-1}+1}^{t_{i}}\mathbf{1}\{X_{t}\in C\},
\]
which is the number of times the covariates land into bin $C$ during batch $i$. The expectation of  $m_{C,i}$ is equal to 
\[
m_{C,i}^{\star}=\mathbb{E}[m_{C,i}]=(t_{i}-t_{i-1})P_X(X\in C).
\]
Since $\bm{u}$ lies in the interior of $\gridset$ and the split factors satisfy equation~\eqref{eq:split-factors}, we have $|C|=w_i=(\prod_{l=0}^{i-1}g_{l})^{-1}\ge T^{-1/(2\beta+d)}.$ The lemma below says $m_{C,i}$ stays closely to its expectation $m_{C,i}^{\star}$ for all $C\in\mathcal{T}$.
\begin{lemma} \label{lemma:clean-event-1} 

Assume that $M\le D_{1}\log(T)$ for some constant $D_{1}>0$. With
probability at least $1-1/T$, for all $1\leq i\leq M$ and $C\in\mathcal{L}_{t_{i-1}+1}$,
one has 
\[
\frac{1}{2}m_{C,i}^{\star}\le m_{C,i}\le\frac{3}{2}m_{C,i}^{\star}.
\]
\end{lemma}

\paragraph{Proof of Lemma~\ref{lemma:clean-event-1}}

Fix the batch index $i$, and a node $C$ in layer-$i$ of the tree
$\mathcal{T}$. If $P_X(C)=0$, then $m_{C,i}=m_{C,i}^{\star}=0$ almost surely. For the remaining part of the proof, we assume $P_X(C)>0$. By relation (\ref{eq:batch-bin-size}), we have 
\begin{align*}
m_{C,i}^{\star} & =(t_{i}-t_{i-1})P_X(X\in C)\\
 & \asymp|C|^{-(2\beta+d)}\log(T|C|^{d})P_X(X\in C)\\
 & \overset{\mathrm{(i)}}{\apprge}|C|^{-2\beta}\overset{\mathrm{}}{\ge}g_{0}^{2\beta},
\end{align*}
where step (i) uses Assumption~\ref{ass:bdd-density}. 
Since $g_{0}=\lfloor T^{\frac{1}{2\beta+d}\cdot u_{1}^{}}\rfloor$ and $u_1>0$, we reach
$m_{C,i}^{\star}\ge\frac{3}{4}\log(2T^{2})$ for all $i$
and $C$, as long as $T$ is sufficiently large. This allows us to invoke
Chernoff's bound to obtain that with probability at most $1/T^{2}$,
\[
\left|\sum\nolimits_{t=t_{i-1}+1}^{t_{i}}\mathbf{1}\{X_{t}\in C\}-m_{C,i}^{\star}\right|\ge\sqrt{3\log(2T^{2})m_{C,i}^{\star}}.
\]
Denote $E^{c}=\{\exists1\le i\le M,C\in\mathcal{L}_{t_{i-1}+1}\text{ such that }\mid\sum_{t=t_{i-1}+1}^{t_{i}}\mathbf{1}\{X_{t}\in C\}-m_{C,i}^{\star}\mid\ge\sqrt{3\log(2T^{2})m_{C,i}^{\star}}\}$.
Applying union bound to reach 
\begin{align*}
\mathbb{P}(E^{c}) & \le\sum_{C\in\mathcal{T}}\frac{1}{T^{2}}
\overset{\mathrm{(ii)}}{\le}\frac{1}{T^{2}}\left(\sum_{i=1}^{M}(\prod_{l=0}^{i-1}g_{l})^{d}\right)\overset{\mathrm{(iii)}}{\le}\frac{1}{T^{2}}\cdot M\cdot(\prod_{l=0}^{M-1}g_{l})^{d},
\end{align*}
where step (ii) sums over all possible nodes of $\mathcal{T}$ across
batches, and step (iii) is due to $(\prod_{l=0}^{i-1}g_{l})^{d}\le(\prod_{l=0}^{M-1}g_{l})^{d}$
for any $1\le i\le M$. Since $g_{M-1}=1$, we further obtain
\begin{align*}
\mathbb{P}(E^{c}) & \le\frac{1}{T^{2}}\cdot M\cdot(\prod_{l=0}^{M-2}g_{l})^{d}\overset{\mathrm{(iv)}}{\le}\frac{1}{T^{2}}\cdot M\cdot t_{M-1}^{\frac{d}{2\beta+d}}\overset{\mathrm{}}{\le}D_{1}\frac{1}{T^{2}}\cdot
T^{\frac{d}{2\beta+d}}
\le\frac{1}{T},
\end{align*}
where step (iv) invokes relation (\ref{eq:batch-bin-size}). This completes
the proof. 

Denote the above event by $E$. By assumption $M\le D_{1}\log(T)$, we use Lemma~\ref{lemma:clean-event-1}
to reach 
\[
\mathbb{E}[R_{T}(\hat{\Gamma}_{\bm{u}}, \hat{\pi}_{\bm{u}})\mathbf{1}(E^{c})]\le T\mathbb{P}(E^{c})=1,
\]
which means the regret incurred when $E$ does not happen is negligible. For the remaining proof, the task becomes controlling $\mathbb{E}[R_{T}(\hat{\pi})\mathbf{1}(E)]$. 

Next, we turn to the arm elimination part. For each bin $C\in\mathcal{L}_{t_{i}}$, denote by $\mathcal{I}_{C}^{\prime}$ the set of remaining arms at the end of batch $i$, i.e., after
Algorithm~\ref{algo-subroutine} is invoked. Define
\[
\underline{\mathcal{I}}_{C}=\left\{ k\in\{1,-1\}:\sup_{x\in C}f^{\star}(x)-f^{(k)}(x)\le c_{0}|C|^{\beta}\right\} ,
\]
\begin{align*}
\bar{\mathcal{I}}_{C} & =\left\{ k\in\{1,-1\}:\sup_{x\in C}f^{\star}(x)-f^{(k)}(x)\le c_{1}|C|^{\beta}\right\} ,
\end{align*}

where $c_{0}=2Ld^{\beta/2}+1$ and $c_{1}=8c_{0}$. By definition,
\[
\underline{\mathcal{I}}_{C}\subseteq\bar{\mathcal{I}}_{C}.
\]
Define the event $\mathcal{A}_{C}=\{\underline{\mathcal{I}}_{C}\subseteq\mathcal{I}_{C}^{\prime}\subseteq\bar{\mathcal{I}}_{C}\}$. Besides, define $\mathcal{G}_{C}=\cap_{C^{\prime}\in\mathcal{P}(C)}\mathcal{A}_{C^{\prime}}$. For $i\ge1$, recall that $\mathcal{B}_{i}$ is the collection of bins $C$ with $|C|=(\prod_{l=0}^{i-1}g_{l})^{-1}=w_{i}$. The following lemma adapted from~\cite{jiang2025batched} shows successive elimination succeeds with high probability.

\begin{lemma}\label{lemma:clean-event-2}For any $1\le i\le M-1$
and $C\in\mathcal{B}_{i}$ such that $P_X(C)>0$, we have 
\[
\mathbb{P}(E\cap\mathcal{G}_{C}\cap\mathcal{A}_{C}^{c})\leq\frac{4m_{C,i}^{\star}}{T|C|^{d}}.
\]
\end{lemma}


\subsection{Regret decomposition}

For any bin $C\in\mathcal{T}$, we consider the following two sources of regret incurred on it. First, define 
\[
r_{T}^{\textrm{live}}(C)\coloneqq \sum_{t=1}^{T}\left(f^{\star}(X_{t})-f_{\pi_{t}(X_{t})}(X_{t})\right)\mathbf{1}(X_{t}\in C)\mathbf{1}(C\in\mathcal{L}_{t}).
\]
Besides, denote by $\mathcal{J}_{t}\coloneqq \cup_{s\le t}\mathcal{L}_{s}$ the set of bins that have been live up to time $t$. Define 
\[
r_{T}^{\textrm{born}}(C)\coloneqq \sum_{t=1}^{T}\left(f^{\star}(X_{t})-f_{\pi_{t}(X_{t})}(X_{t})\right)\mathbf{1}(X_{t}\in C)\mathbf{1}(C\in\mathcal{J}_{t}).
\]
Due to the structure of the tree $\mathcal{T}$, we have 
\begin{align*}
r_{T}^{\textrm{born}}(C) & =r_{T}^{\textrm{live}}(C)+\sum_{C^{\prime}\in\mathsf{child}(C)}r_{T}^{\textrm{born}}(C^{\prime})\\
 & =r_{T}^{\textrm{born}}(C)\mathbf{1}(\mathcal{A}_{C}^{c})+r_{T}^{\textrm{live}}(C)\mathbf{1}(\mathcal{A}_{C})+\sum_{C^{\prime}\in\mathsf{child}(C)}r_{T}^{\textrm{born}}(C^{\prime})\mathbf{1}(\mathcal{A}_{C}).
\end{align*}
The following regret decomposition is an immediate consequence of  iteratively applying the relation above to each level of the tree.
\begin{align*}
R_{T}(\hat{\Gamma}_{\bm{u}}, \hat{\pi}_{\bm{u}}) & =\mathbb{E}[r_{T}^{\textrm{born}}(\mathcal{X})]\\
 & =\sum_{C^{\prime}\in\mathsf{child}(\mathcal{X})}\mathbb{E}[r_{T}^{\textrm{born}}(C^{\prime})]\\
 & =\sum_{1\le i \leq M-1}\left(\underbrace{\sum_{C\in\mathcal{B}_{i}}\mathbb{E}[r_{T}^{\textrm{born}}(C)\mathbf{1}(\mathcal{G}_{C}\cap\mathcal{A}_{C}^{c})]}_{\eqqcolon U_{i}}+\underbrace{\sum_{C\in\mathcal{B}_{i}}
 \mathbb{E}[r_{T}^{\textrm{live}}(C)\mathbf{1}(\mathcal{G}_{C}\cap\mathcal{A}_{C})]}_{\eqqcolon V_{i}}\right)\\
 & \quad+\sum_{C\in\mathcal{B}_{M}}
 \mathbb{E}[r_{T}^{\textrm{live}}(C)\mathbf{1}(\mathcal{G}_{C})],
\end{align*}
where the second step is due to $r_{T}^{\textrm{live}}(\mathcal{X})=0$ (note $\mathcal{X}\notin\mathcal{L}_{t}$ for any $1\leq t\leq T$). Now that we have a decomposition of the regret, the task becomes bounding $V_{i},U_{i}$ and the regret of the last batch separately. 

We first consider the case of $\alpha\le d/\beta$.

\paragraph{Upper bounding term $V_{i}$.} Fix some $1\leq i\leq M-1$, and some bin $C\in\mathcal{B}_{i}$. The event $\mathcal{G}_{C}$ implies $\mathcal{I}_{\mathsf{p}(C)}^{'}\subseteq\bar{\mathcal{I}}_{\mathsf{p}(C)}$. Namely, for any $k\in\mathcal{I}_{\mathsf{p}(C)}^{'}$, 
\[
\sup_{x\in\mathsf{p}(C)}f^{\star}(x)-f^{(k)}(x)\le c_{1}|\mathsf{p}(C)|^{\beta}.
\]
Consequently, for any $x\in C$, and $k\in\mathcal{I}_{\mathsf{p}(C)}^{'}$,
\begin{equation}
\left(f^{\star}(x)-f^{(k)}(x)\right)\bm{1}\{\mathcal{G}_{C}\}\leq c_{1}|\mathsf{p}(C)|^{\beta}\mathbf{1}(0<\left|f_{1}(x)-\negrev(x)\right|\le c_{1}|\mathsf{p}(C)|^{\beta}).\label{eq:remaining-arm-gap}
\end{equation}
This leads to
\begin{align*}
& \mathbb{E}[r_{T}^{\textrm{live}}(C)\mathbf{1}(\mathcal{G}_{C}\cap\mathcal{A}_{C})]\\  
 &\quad  \overset{\mathrm{}}{\le}\mathbb{E}\left[\sum_{t=1}^{T}c_{1}|\mathsf{p}(C)|^{\beta}\mathbf{1}(0<\left|f_{1}(X_{t})-\negrev(X_{t})\right|\le c_{1}|\mathsf{p}(C)|^{\beta})\mathbf{1}(X_{t}\in C,C\in\mathcal{L}_{t})\mathbf{1}(\mathcal{G}_{C}\cap\mathcal{A}_{C})\right]\\
 &\quad \overset{\mathrm{(i)}}{\le}c_{1}|\mathsf{p}(C)|^{\beta}\mathbb{E}\left[\sum_{t=t_{i-1}+1}^{t_{i}}\mathbf{1}(0<\left|f_{1}(X_{t})-\negrev(X_{t})\right|\le c_{1}|\mathsf{p}(C)|^{\beta},X_{t}\in C)\mathbf{1}(\mathcal{G}_{C}\cap\mathcal{A}_{C})\right]\\
 &\quad \overset{\mathrm{(ii)}}{\le}c_{1}|\mathsf{p}(C)|^{\beta}\sum_{t=t_{i-1}+1}^{t_{i}}\mathbb{P}(0<\left|f_{1}(X_{t})-\negrev(X_{t})\right|\le c_{1}|\mathsf{p}(C)|^{\beta},X_{t}\in C)\\
 &\quad =c_{1}|\mathsf{p}(C)|^{\beta}(t_{i}-t_{i-1})\mathbb{P}(0<\left|f_{1}(X)-\negrev(X)\right|\le c_{1}|\mathsf{p}(C)|^{\beta},X\in C).
\end{align*}
Here, step (i) can be deduced from considering the cases of whether $C$ is split or not; step (ii) is because $\mathbf{1}(\mathcal{G}_{C}\cap\mathcal{A}_{C})\le1$.

Summing over all bins in $\mathcal{B}_{i}$, we reach 

\begin{align}
\sum_{C\in\mathcal{B}_{i}}\mathbb{E}[r_{T}^{\textrm{live}}(C)\mathbf{1}(\mathcal{G}_{C}\cap\mathcal{A}_{C})] & \le\sum_{C\in\mathcal{B}_{i}}c_{1}w_{i-1}^{\beta}(t_{i}-t_{i-1})\mathbb{P}(0<\left|f_{1}(X)-\negrev(X)\right|\le c_{1}|\mathsf{p}(C)|^{\beta},X\in C)\nonumber \\
 & =c_{1}w_{i-1}^{\beta}(t_{i}-t_{i-1})\sum_{C\in\mathcal{B}_{i}}\mathbb{P}(0<\left|f_{1}(X)-\negrev(X)\right|\le c_{1}w_{i-1}^{\beta},X\in C)\nonumber\\
 &=c_{1}w_{i-1}^{\beta}(t_{i}-t_{i-1})\mathbb{P}(0<\left|f_{1}(X)-\negrev(X)\right|\le c_{1}w_{i-1}^{\beta}),\label{eq:vi-calculation}
\end{align}
where the penultimate step is due to $|\mathsf{p}(C)|=w_{i-1}$. 
We can apply the margin condition to obtain
\begin{align*}
V_i=\sum_{C\in\mathcal{B}_{i}}\mathbb{E}[r_{T}^{\textrm{live}}(C)\mathbf{1}(\mathcal{G}_{C}\cap\mathcal{A}_{C})] & \le(t_{i}-t_{i-1})\cdot[c_{1}w_{i-1}^{\beta}]^{1+\alpha}\cdot D_{0}.
\end{align*}

\paragraph{Upper bounding term $U_i$.} Fix some $1\leq i\leq M-1$, and some bin $C\in\mathcal{B}_{i}$ such that $P_X(C)>0$.
By relation~\eqref{eq:remaining-arm-gap},

\begin{align*}
\mathbb{E}&[r_{T}^{\textrm{born}}(C)\mathbf{1}(\mathcal{G}_{C}\cap\mathcal{A}_{C}^{c})] \\
&\le\mathbb{E}\left[\sum_{t=1}^{T}c_{1}|\mathsf{p}(C)|^{\beta}\mathbf{1}(0<\left|f_{1}(X_{t})-\negrev(X_{t})\right|\le c_{1}|\mathsf{p}(C)|^{\beta})\mathbf{1}(X_{t}\in C,C\in\mathcal{J}_{t})\mathbf{1}(\mathcal{G}_{C}\cap\mathcal{A}_{C}^{c})\right]\\
 & \le c_{1}|\mathsf{p}(C)|^{\beta}T\mathbb{P}(0<\left|f_{1}(X)-\negrev(X)\right|\le c_{1}|\mathsf{p}(C)|^{\beta},X\in C)\mathbb{P}(\mathcal{G}_{C}\cap\mathcal{A}_{C}^{c})\\
 & \le c_{1}|\mathsf{p}(C)|^{\beta}T\mathbb{P}(0<\left|f_{1}(X)-\negrev(X)\right|\le c_{1}|\mathsf{p}(C)|^{\beta},X\in C)\frac{4m_{C,i}^{\star}}{T|C|^{d}}\\
 & \le4\bar{c}c_{1}w_{i-1}^{\beta}\mathbb{P}(0<\left|f_{1}(X)-\negrev(X)\right|\le c_{1}w_{i-1}^{\beta},X\in C)(t_{i}-t_{i-1}),
\end{align*}
where the penultimate step uses Lemma~\ref{lemma:clean-event-2}, and the last step is due to Assumption~\ref{ass:bdd-density}. 
Consequently, we reach
\begin{align}\label{eq:ui-calculation}
U_i
&=\sum_{C\in\mathcal{B}_{i}}\mathbb{E}[r_{T}^{\textrm{born}}(C)\mathbf{1}(\mathcal{G}_{C}\cap\mathcal{A}_{C}^{c})]\nonumber \\
&\le4\bar{c}c_{1}w_{i-1}^{\beta}(t_{i}-t_{i-1})\sum_{C\in\mathcal{B}_{i}}\mathbb{P}(0<\left|f_{1}(X)-\negrev(X)\right|\le c_{1}w_{i-1}^{\beta},X\in C)\nonumber\\
&=4\bar{c}c_{1}w_{i-1}^{\beta}(t_{i}-t_{i-1})\mathbb{P}(0<\left|f_{1}(X)-\negrev(X)\right|\le c_{1}w_{i-1}^{\beta}),
\end{align}
By the margin condition, we get
\[U_i\le4D_{0}\bar{c}(t_{i}-t_{i-1})[c_{1}w_{i-1}^{\beta}]^{1+\alpha}.\]

\paragraph{Regret of last Batch.} For $C\in\mathcal{B}_{M}$, similarly we have 
\begin{align*}
\mathbb{E}[r_{T}^{\textrm{live}}(C)\mathbf{1}(\mathcal{G}_{C})] & \le c_{1}|\mathsf{p}(C)|^{\beta}(T-t_{M-1})\mathbb{P}(0<\left|f_{1}(X)-\negrev(X)\right|\le c_{1}|\mathsf{p}(C)|^{\beta},X\in C).
\end{align*}
Summing over $C\in\mathcal{B}_{M}$ gives
\begin{align}\label{eq:lastbatch-calculation}
\sum_{C\in\mathcal{B}_{M}}\mathbb{E}[r_{T}^{\textrm{live}}(C)\mathbf{1}(\mathcal{G}_{C})] & \le\sum_{C\in\mathcal{B}_{M}}c_{1}|\mathsf{p}(C)|^{\beta}(T-t_{M-1})\mathbb{P}(0<\left|f_{1}(X)-\negrev(X)\right|\le c_{1}|\mathsf{p}(C)|^{\beta},X\in C)\nonumber\\
&=c_{1}w_{M-1}^{\beta}(T-t_{M-1})\mathbb{P}(0<\left|f_{1}(X)-\negrev(X)\right|\le c_{1}w_{M-1}^{\beta}) \\
 & \le c_{1}w_{M-1}^{\beta}(T-t_{M-1})D_{0}\cdot\left[c_{1}w_{M-1}^{\beta}\right]^{\alpha}\nonumber\\
 & =D_{0}(T-t_{M-1})[c_{1}w_{M-1}^{\beta}]^{1+\alpha}\nonumber.
\end{align}

\paragraph{Putting things together.} By combining the bounds of $V_i, U_i$ and the regret of the last batch, we obtain

\begin{align*}
R_{T}(\hat{\Gamma}_{\bm{u}}, \hat{\pi}_{\bm{u}}) 
&=\sum_{1\le i<M}\left(U_i+V_i\right)+\sum_{C\in\mathcal{B}_{M}}
 \mathbb{E}[r_{T}^{\textrm{live}}(C)\mathbf{1}(\mathcal{G}_{C})]\\
 & \le c\left(t_{1}+\sum_{i=2}^{M-1}(t_{i}-t_{i-1})\cdot w_{i-1}^{\beta+\alpha\beta}+(T-t_{M-1})w_{M-1}^{\beta+\alpha\beta}\right),
\end{align*}
where $c$ is a constant that depends on $(\beta,d)$. 

Finally, we deal with the case of $\alpha=\infty$. By relations~\eqref{eq:vi-calculation} and ~\eqref{eq:ui-calculation}, one has
\begin{align}\label{eq:ainf-general-bd}
    U_i+V_i\le (1+4\bar{c})c_{1}w_{i-1}^{\beta}(t_{i}-t_{i-1})\mathbb{P}(0<\left|f_{1}(X)-\negrev(X)\right|\le c_{1}w_{i-1}^{\beta}).
\end{align}
When $i=1$, the above relation simplifies to $U_1+V_1\le(1+4\bar{c})c_{1}t_1$ because the probability term is upper-bounded by 1. 

For $i\ge2$, one has $c_{1}w_{i-1}^{\beta}<\delta_0$ due to the definition of $w_{i-1}$ and by the margin condition with $\alpha=\infty$, we get $\mathbb{P}(0<\left|f_{1}(X)-\negrev(X)\right|\le c_{1}w_{i-1}^{\beta})=0$. Consequently, relation~\eqref{eq:ainf-general-bd} reduces to $U_i+V_i=0$.

 For the last batch, relation~\eqref{eq:lastbatch-calculation} similarly reduces to 0. Hence,
\[R_{T}(\hat{\Gamma}_{\bm{u}}, \hat{\pi}_{\bm{u}})\le (1+4\bar{c})c_{1}t_1.\]

\end{document}